\documentclass[amsppt,11pt]{amsart}
\usepackage{curves}
\usepackage{amsmath}
\usepackage[dvipsnames]{xcolor}  
    \usepackage{pdfsync} 

\newtheorem {theorem}{Theorem}[section]
\newtheorem {definition}{Definition}[section]
\newtheorem {proposition}[theorem]{Proposition}
\newtheorem {lemma}[theorem]{Lemma}
\newtheorem {corollary}[theorem]{Corollary}
\newtheorem {assumption}[theorem]{Assumption}

\newcounter{conjecture}
\newcounter{remark}

\newcommand{\eqnsection}{
   \renewcommand{\theequation}{\thesection.\arabic{equation}}
   \makeatletter
   \csname @addtoreset\endcsname{equation}{section}
   \makeatother}

\def \be{\begin{equation}}
\def \ee{\end{equation}} 
\def \bt{\begin{theorem}}
\def \et{\end{theorem}}
\def \bea{\begin{eqnarray}}
\def \eea{\end{eqnarray}}
\def \bas{\begin{eqnarray*}}
\def \eas{\end{eqnarray*}}

\def \bl{\begin{lemma}} 
\def \el{\end{lemma}}

\def \al{\alpha}
\def \bb{\beta}
\def \ga{\gamma}
\def \Ga{\Gamma}
\def \de{\delta}
\def \De{\Delta}
\def \ep{\epsilon}

\newcommand{\eps}{\varepsilon}

\def \la{\lambda}

\def \ka{\kappa}

\def \si{\sigma}

\def \th{\theta}

\def \ze{\zeta}

\def \ff{\infty}
\def \wh{\widehat}
\def \wt{\widetilde}

\def \rar{\rightarrow}

\newcommand{\ls}[1]
   {\dimen0=\fontdimen6\the\font \lineskip=#1\dimen0
\advance\lineskip.5\fontdimen5\the\font \advance\lineskip-\dimen0
\lineskiplimit=.9\lineskip \baselineskip=\lineskip
\advance\baselineskip\dimen0 \normallineskip\lineskip
\normallineskiplimit\lineskiplimit \normalbaselineskip\baselineskip
\ignorespaces }

\def \R{{\bf R}}

\def \Z{{\bf Z}}

\def \T{{\bf T}}

\def \S{{\bf S}}

\def \AA{{\mathcal A}}
\def \BB{{\mathcal B}}
\def \CC{{\mathcal C}}

\def \GG{{\mathcal G}}
\def \HH{{\mathcal H}}
\def \II{{\mathcal I}}
\def \JJ{{\mathcal J}}
\def \KK{{\mathcal K}}
\def \LL{{\mathcal L}}

\def \PP{{\mathcal P}}

\def \UU{{\mathcal U}}
\def \VV{{\mathcal V}}
\def \WW{{\mathcal W}}
\def \Qb{{\mathbb Q}}

\def \Iyz{\II_{y,z}}
\def \Iyzp{\II_{y',z}}
\def \Iyzh{\widehat\II_{y,z}}
\def \Iyzhp{\widehat\II_{y',z}}
\def \Iyzhpk{\widehat\II_{y', k, z}}

\def \({\left(}
\def \){\right)}

\def \lc{\left\{}
\def \rc{\right\}}

\def \nn{\nonumber}

\def \bc{\begin{center} }
\def \ec{\end{center} }

\eqnsection

\newcommand{\beq}[1]{\begin{equation}\label{#1}}
\newcommand{\eeq}{\end{equation}}

\newcommand{\E}{{\Bbb E}}

\newcommand{\beqn}[1]{\begin{eqnarray}\label{#1}}
\newcommand{\eeqn}{\end{eqnarray}}

\def\squarebox#1{\hbox to #1{\hfill\vbox to #1{\vfill}}}
\renewcommand{\qed}{\hspace*{\fill}
            \vbox{\hrule\hbox{\vrule\squarebox{.667em}\vrule}\hrule}\smallskip}

\newcommand{\beaa}{\begin{eqnarray*}}
\newcommand{\eeaa}{\end{eqnarray*}}

\def \rzero{ r_{0} }

\def \npp{0}
\newcommand{\Pbm}{\mathbb{P}}
\newcommand{\Ebm}{\mathbb{E}}
\newcommand{\Pgw}{P^{\mbox{\fontsize{.03in}{1in}\rm GW}}}

\newcommand{\hit}{H}

\newcommand{\BM}{W}
\newcommand{\BMS}{X}
\newcommand{\dwa}{d^1_{\mbox{\fontsize{.03in}{1in}\rm Wa}}}
\newcommand{\gykm}{\GG_{y,k}(N_{k,a})}
\newcommand{\trav}[7]{\ensuremath{T_{#5,#6\to #7}^{#2,#4\overset{#1}{\to}#3} }}
\newcommand{\travshort}[7]{\trav{#1}{#2}{#3}{#4}{#5}{#6}{#7}}
\def\corJ{}
\def\corO{}
\title[Tightness for Cover Times]
{Tightness for the Cover Time of the two dimensional sphere}

\author[David Belius\;\; Jay Rosen\;\; Ofer Zeitouni]
{David Belius\;\;  Jay Rosen\;\; 
Ofer Zeitouni}

\date{January 18, 2018. Revised April 14, 2019, July 29, 2019,
November 15, 2019 and January 28, 2020}

\thanks{Jay Rosen was partially supported by the Simons Foundation.}
\thanks{This project has received funding from the European Research Council (ERC) under the European Union's Horizon 2020 research and innovation programme (grant agreement No. 692452).}
\begin{document}

\begin{abstract}
  Let $\CC^\ast_{\ep,\S^2}$ denote the cover 
  time of the two dimensional sphere by
  %compact manifold $M$ by 
a Wiener sausage of radius $\ep$. We prove that 
$$\sqrt{\CC^{\ast}_{\ep,\S^2} } 
-\sqrt{\frac{2A_{\S^2}}{\pi}}\(\log \ep^{-1}-\frac14\log\log \ep^{-1}\)$$ is tight, where $A_{\S^2}=4\pi$ denotes the Riemannian area of $\S^2$.   
\end{abstract}
\maketitle

\section{Introduction}
Let $M$ be a smooth, compact, connected,
two-dimensional Riemannian manifold without boundary. For each $x\in M$ let
$\CC_{x,\ep,M}$ be  the amount of time needed for the
 Brownian motion to come within (Riemannian) distance
$\ep$ of $x$. 
Then  $\CC^{\ast}_{\ep,M}=\sup_{x}\CC_{x,\ep,M}$ is the 
$\ep$-cover time of $M$.
It is shown in \cite[Theorem 1.3]{DPRZ} that
\begin{equation}
  \lim_{\eps\rar 0}\frac{\CC^{\ast}_{\ep,M}}{
\left(\log \ep\right)^2}=\frac{2A_{M}}{\pi}  \quad a.s.,
\label{mp.10}
\end{equation} 
where $A_{M}$ denotes the Riemannian area of $M$. 
For the special case where $M$ is the two dimensional torus $\T^2$
with $A_{\T^2}=1$,
\cite{BK} showed that the rescaled cover time 
$\CC^{\ast}_{\ep,\T^2}/(\frac{1}{\pi} \log \epsilon^{-1})$
has a log-log correction term:
\begin{equation}
  \label{eq-BK1.2}
  \frac{\CC^{\ast}_{\ep,\T^2}}{ \frac{1}{\pi} \log \epsilon^{-1} } = 
  2 \log \ep^{-1} - \log\log \ep^{-1} + o(\log\log \ep^{-1}),
\end{equation}
see \cite[(1.2)]{BK}. (An analogue of 
\eqref{eq-BK1.2} for the cover time
of the discrete
torus by simple random walk was recently obtained in \cite{Abe}.)
Note that, with
$c^*_{M}=\sqrt{\frac{2A_{M}}{\pi}}$ and
\begin{equation}
  \label{eq-meps}
  m_{\ep,M}=c^*_M\left(\log \ep^{-1}-\frac{1}{4}
  \log \log \epsilon^{-1}\right),
\end{equation}
\eqref{eq-BK1.2} can also be written as
\begin{equation}
  \label{eq-BK1.2bis}
  \sqrt{\CC^{\ast}_{\eps,\T^2}}- m_{\ep,\T^2}
  =o(\log\log \epsilon^{-1}).
\end{equation}

\subsection{Tightness of cover time}
In spite of recent progress concerning the study of the maximum of various correlated fields, see Section \ref{Background} for details,
improving on \eqref{eq-BK1.2} has remained elusive.
Our goal in this paper is to improve on 
\eqref{mp.10} and 
\eqref{eq-BK1.2} by  proving tightness, 
in the case of the standard two dimensional sphere, $M=\S^2$, where
$A_{\S^2}=4\pi$ and $c^*_{\S^2}=2\sqrt{2}$. (We comment below on our choice
of working with $M=\S^2$.)
%Our main result reads as follows.
%\begin{theorem}
%  \label{theo-tightnessM}
%The sequence of random variables
%$$ \sqrt{\CC^{\ast}_{\eps,M}}- m_{\eps,M} $$
%is tight. More explicitly, 
%\begin{equation}
%  \label{eq-tightM}
%  \lim_{K\to\infty}\limsup_{\ep\to 0}\Pbm\left(
%  \left|\sqrt{\CC^{\ast}_{\eps,M}}- m_{\eps,M}\right|>K\right)=0.
%\end{equation}
%\end{theorem}
%Here,
%Let $\Pbm^x$ denote the probability measure for  
%Brownian motion on the sphere starting at $x$, and 
%whenever probabilities do not depend on the starting point $x$
%of the Brownian motion, we write $\Pbm$ instead of $\Pbm^x$.
%
%
%
%Our result for general manifolds $M$ will be derived from the special case where  
Let $\Pbm^x$ denote the probability measure for  
Brownian motion on the sphere starting at $x$, and 
whenever probabilities do not depend on the starting point $x$
of the Brownian motion, we write $\Pbm$ instead of $\Pbm^x$.
We use $B_{d}\(   a,r\)$ denote the ball in $\S^2$ centered at $a$ of radius $r$ in the standard metric $d$ for $\S^2$, see Section \ref{sec-isot} and \eqref{eq-Bdxr}.
Our main result reads as follows.
\begin{theorem}
  \label{theo-tightness}
The sequence of random variables
$$ \sqrt{\CC^{\ast}_{\eps,\S^2}}- m_{\eps,\S^2} $$
is tight. More explicitly, 
\begin{equation}
  \label{eq-tight}
  \lim_{K\to\infty}\limsup_{\ep\to 0}\Pbm\left(
  \left|\sqrt{\CC^{\ast}_{\eps,\S^2}}- m_{\eps,\S^2}\right|>K\right)=0.
\end{equation}

In addition, for any $B_{d}\(   a,r\)\subseteq \S^2$, the same result holds if $\CC^{\ast}_{\eps,\S^2}$ is replaced by $\CC^{\ast}_{\eps,\S^2, B_{d}\(   a,r\)}$, the 
$\ep$-cover time of $B_{d}\(   a,r\)\subseteq \S^2$ by  Brownian motion on $\S^2$.
\end{theorem}
Note that \eqref{eq-tight} is equivalent to the statement 
\begin{equation}
  \label{eq-tightbis}
  \lim_{K\to\infty}\limsup_{\ep\to 0}\Pbm\left(
  \left|\frac{\CC^{\ast}_{\eps,\S^2}}{\frac12 (c^*_{\S^2})^2\log \ep^{-1}}-
  \left(2\log \ep^{-1}-\log\log\ep^{-1}\right)\right|>K\right)=0.
\end{equation}
With a slight abuse of language, we use the statement 
\textit{the cover time of $M$ is tight} to mean
either \eqref{eq-tight} or \eqref{eq-tightbis}, with $\S^2$ replaced by $M$.

As in \cite{DPRZ} and subsequent work on finer results including
\cite{BK} and \cite{BRZ},
the key to Theorem \ref{theo-tightness} is obtaining 
good upper and lower bounds on
the right tail of the distribution of the centered cover time (with
possibly different centerings for the upper and lower bounds).
Our main technical
result is in this spirit, this time with precisely the correct centering for both bounds.
\bt\label{theo-spheretight}
On $\S^{2}$, for some $0<C, C',d<\ff$ and all $z\geq 0$,
\begin{equation}
  \limsup_{\ep\to 0}  \Pbm\(\sqrt{\CC^{\ast}_{\ep,\S^2} } -
  m_{\ep,\S^2}
  \geq z\)\leq C  e^{ -\sqrt{2}z+d\sqrt{z}}.\label{goal.1asto}
\end{equation}
and
\begin{equation}
  \liminf_{\ep\to 0}  \Pbm\(\sqrt{\CC^{\ast}_{\ep,\S^2} } 
  -m_{\ep,\S^2}
  \geq z\)
  \geq
  C'  e^{ -\sqrt{2}z-d\sqrt{z}} 
  .\label{goal.1astow}
\end{equation}
\et
We believe, in analogy with \cite{BRZ}, that 
the right side of \eqref{goal.1asto} and \eqref{goal.1astow} should
be multiplied by a factor $(z\vee 1)$ and that
the factor $d\sqrt{z}$ should not be present in the exponent. Obtaining 
such precision is beyond our current method of proof, but would be important 
if one tried to establish a limit law for $\sqrt{\CC^{\ast}_{\ep,\S^2} } -
m_{\ep,\S^2}$. See subsection  \ref{sec-open} below.

We believe that Theorem \ref{theo-spheretight} should extend to other
two-dimensional compact manifolds.\footnote{We expect that
  the case of the torus $\mathbf{T}^2$ could be handled by our methods, 
although we do not address this in the paper.}
 While this is currently
outside our reach\footnote{An earlier version of this article claimed 
  such a result,
based on a faulty reduction from general manifolds to planar Brownion motion.}, the following
may serve as intermediate step.
Let $B_{e}\(   a,r\)$  denote the ball in $\R^2$ centered at 
$a$ of radius $r$.
Let $\corJ{\CC_{\epsilon,P}^*}$ denote the $\epsilon$-cover time of the unit 
disc by planar Brownian motion $W_t$. For $R>0$, set
\begin{equation}
  \label{eq-CR}
  \CC_{\ep,P,R}^*=\int_0^{\CC_{\ep,P}^*}{\bf 1}_{W_t\in B_e(0,R)} dt,
\end{equation}
   \corJ{the amount of time the path needs to spend in $B_e(0,R)$ to come within $\ep$ of each point in the unit 
disc. }
Set 
$c^*_{P,R}=\sqrt{2}R$ and,
in analogy with \eqref{eq-meps}, set
\begin{equation}
  \label{eq-mepsplane}
  m_{\ep,P,R}=c^*_{P,R}\left(\log \ep^{-1}-\frac{1}{4}
  \log \log \epsilon^{-1}\right).
\end{equation}
\begin{corollary}
  \label{cor-plane}
  Let $R\geq 1$. Then the sequence of random variables
  $$ \sqrt{\CC^{\ast}_{\eps,P,R}}- m_{\eps,P,R} $$
is tight.
\end{corollary}

In the rest of the paper except for Section \ref{sec-genman} , 
we drop $\S^2$ from the notation and write 
$\CC^\ast_\ep$, $c^*$ instead of
$\CC^\ast_{\ep,\S^2}$, $c^*_{\S^2}$.

\subsection{Methods and limitations}
Theorem  \ref{theo-tight} is the counterpart of our earlier
result \cite[Theorem 1.3]{BRZ}
concerning the cover time of the binary tree of depth $n$ by 
a random walk, and has the same form. In fact, the underlying method 
of proof is, at a high level, similar. 
The uninitiated 
reader may find  it helpful to
read 
\cite[Section 5]{BRZ} prior to going over the details of the proofs 
in this paper. However, the tree 
possesses a certain decoupling property that is not present in $\S^2$, and
for this reason the proof for $\S^2$ is much more intricate. 
Two crucial new ideas, which could be considered as the main innovations of this paper, are needed
in order to handle the manifold case. First, continuity estimates, which provide control of correlations in short scales,
are developed in Section \ref{sec-continest}. This is a point where the argument differs from all previous works, and it is precisely 
in order to derive such estimates that we work with $\S^2$ instead of a general manifold. Note that such issues are not present
for the cover time of the binary tree, discussed in \cite{BRZ}. The second main new idea relates to decoupling at coarse scales, where we introduce and use an $L^1$-Wasserstein distance in order to couple certain dependent Brownian excursions to a collection of independent ones. This is explained below and in Section \ref{sec-Decoupling}.
%We now describe the 
%main ideas,  emphasizing why working with the special case $M=\S^2$ 
%simplifies the argument.
In the rest of this subsection, we give an outline of our argument.

As in 
\cite{DPRZ} and \cite{BK}, rather than work directly with  
$\CC^\ast_\ep$, 
we work with an object  $t^{\ast}_{L}$ 
which we call the $h_{L}$-cover local excursion time. Here
$h_{l}=2\arctan ( \rzero e^{ -l}/2)$ with $\rzero$ small
(the appearance of $\arctan$ is due to our use of
isothermal coordinates, see
Section \ref{sec-isot} below), 
and 
$L$ is chosen
so that $h_L\sim \ep$, i.e. $L\sim \log(1/\ep)$. 

To define $t_L^*$, let $T_{l}^{x, n }$ be the number of excursions from  $\partial B_{d}\left(x,h_{l-1}\right)$ to $ \partial B_{d}\left(x,h_{l}\right)$ prior 
to completing the $n$'th excursion from  
$\partial B_{d}\left(x,h_{1}\right)$ to $ \partial B_{d}\left(x,h_{0}\right)$. 
The processes $T^{x,n}_l,l\ge0,$ are in fact critical Galton-Watson 
processes with geometric offspring distribution, which
explains why the estimates in \cite{BRZ} are relevant.

Let 
  \begin{equation}
t^{\ast}_{x, L}=\inf  \{n\,| T_{L}^{x,n}\neq 0\},\label{deftast}
 \end{equation}
 the number of excursions from $\partial B_{d}\left(x,h_{1}\right)$ to $ \partial B_{d}\left(x,h_{0}\right)$ needed   for the
 Brownian motion to come within (Riemannian) distance
$h_{L}$ of $x$, and let $F_{l}$ be the centers of an $h_{l}/1000$ minimal cover of $\S^{2}$.
Then
\begin{equation}
  \label{eq-tsar}
  t^{\ast}_{L}:=\sup_{x\in F_L} t^{\ast}_{x, L},
\end{equation}
 can be thought of as a cover time, but with time measured  locally, 
 that is,  in terms of the number of
 excursions from $\partial B_{d}\left(x,h_{1}\right)$ 
 to $ \partial B_{d}\left(x,h_{0}\right)$,
for each $x$. Note that  
\begin{equation}
  \label{eq:ttotz}
  \Big\{t^{\ast}_{ L}> t\Big\}=\Big\{T_{L}^{y,t}=0\mbox{ for some }y\in F_{L}\Big\}. 
\end{equation}

\begin{equation}
\rho_{L}= 2-\frac{\log L}{2L},\label{eq:defofts}
\end{equation}
and 
\begin{equation}
t_{z}=t_{L,z}=\(\rho_{L}L+z\)^{2}/2=L\left(2L-\log L+2z\right)+z^{2}/2+O\(z\log L\).\label{eq:defoftsr}
\end{equation}
In terms of excursion counts, our main result is the following.
\bt\label{theo-tight}
Fix $r_0$ small. On $\S^{2}$,   for some $C,C'<\ff$ and all $z\geq 0$,
\begin{equation}
 \limsup_{L\to\ff} \Pbm\left( T_{L}^{y,t_{z}}=0\mbox{ for some }y\in F_{L}\right)\leq C(1+z)e^{ -2z },\label{goal.1}
\end{equation}
and  
\begin{equation}
 \liminf_{L\to\ff} \Pbm\left( T_{L}^{y,t_{z}}=0\mbox{ for some }y\in F_{L}\right)\geq C'ze^{ -2z}.\label{goal.1L}
\end{equation}
\et
Once Theorem \ref{theo-tight} is established, in order
to relate excursion counts to time, we use the fact that 
there are many excursions at the macroscopic level before
the cover time, and hence a law of large numbers should allow one to 
transfer excursion counts to running time. The actual argument is 
somewhat more complicated, mostly because one is dealing with excursion counts
between circles of different centers, and so continuity considerations are
important - one needs to show that, with high probability,
the function $x\mapsto t^{\ast}_{x, L}$
is essentially continuous. Since the same issue also arises 
in the study of the upper bound \eqref{goal.1}, we discuss it in greater
detail below, and only mention that this is one place where the assumption
that $M=\S^2$ is used crucially.

We now discuss the  proof of the upper bound
\eqref{goal.1}. The basic idea is simple, and 
reminiscent of similar computations done in 
the context of branching random walks, going back to \cite{BM78},
see \cite{ZLN,K13,LP} for a review;
we follow \cite{BK} closely in the
precise mapping of our cover time problem to the language of branching random walk.
Using the fact that 
$T^{x,n}_l$ is a critical geometric Galton-Watson process, 
one may attempt a union bound of the form
$$\Pbm\left( T_{L}^{y,t_{z}}=0\mbox{ for some }y\in F_{L}\right)
\leq \sum_{y\in F_L}
\Pbm\left( T_{L}^{y,t_{z}}=0\right).$$
By standard estimates for Galton--Watson processes, see 
\eqref{eq:NotHitByrLa} below, one sees
that this computation is not sharp enough and would work only if
one decreased the $\log L$  correction term in the definition
of $  \rho_L$. (Indeed, this is what is done in \cite{DPRZ}, which only 
provides
the correct leading order.) 
Instead,  informed by the branching random walk analogy of \cite{BK},
one observes that the process $l \to \sqrt{2T_l^{y,t_{z}}}$ should behave
like a particle in branching random walk, and should therefore satisfy
a barrier condition.
Indeed, by barrier estimates
for geometric
Galton--Watson processes that were derived in our earlier work \cite{BRZ}, 
see Lemma \ref{prop:BarrierSecGWPropUB}, 
the estimate \eqref{goal.1} would follow
by a union bound
if instead of the event  $\{T_{L}^{y,t_{z}}=0\}$, one would
consider the event
$$\{\sqrt{2T_l^{y,t_{z}}}\geq \alpha(l), l=1,\ldots,L-1, T_{L}^{y,t_{z}}=0\}$$
where $\alpha(\cdot)$ is the  barrier
\begin{equation}
\alpha\left(l\right) =  \rho_{L}(L-l)  -
l_{L}^{\ga}, \quad 
l_{L} = (l\wedge (L-l)),\quad \ga=0.4, \;\mbox{\rm $l$ integer}.
\label{eq:AlphaBarrierDefd}
\end{equation}
Thus, it remains to handle the event that there exists $y\in F_L$ for which
$\{T_{L}^{y,t_{z}}=0\}$ \textit{and}
$\{\sqrt{2T_l^{y,t_{z}}}<  \alpha(l)$ for some $l\leq L-1\}$.

It is at this point that the strategy diverges from the 
case of branching random walks: one cannot take a union bound 
over all $y\in F_L$  concerning events at level $l$ - the decay probability
for the event $T_l^{y,t_{z}}<  \alpha(l)^2$ has exponential decay
rate 
roughly $-2l$ but the exponential growth rate of $|F_L|$ is
$2L$. Instead, we must effectively reduce the cardinality 
of points $y$ to consider for events involving scale $l$. 
This point was already present in previous work relating 
extreme problems to extremes of branching random walks, see
e.g. \cite{ABH}, and appears also in the context of cover times
in \cite{BK}. In the latter paper
this problem
was solved 
by allowing a 
margin of error and using a \textit{deterministic} sandwiching of excursions 
between slightly smaller/larger balls to relate $T_l^{y,t_{z}}$ and $T_l^{y',t_{z}}$
for $y,y'$ with $d(y,y')\ll e^{-2l}$.
Here, we cannot afford such errors, and instead resort to a probabilistic 
estimate:  because of the symmetry of the sphere, during one excursion 
between concentric circles $S_1,S_2$ 
centered at a point $x$ started uniformly on the inner circle $S_1$, the 
expected number of excursions between concentric smaller circles 
centered around a point $y$ well inside $S_1$ does not depend on $y$,
see Lemma \ref{lem-equalmean}.
This, together with a chaining argument (which can be traced back, in this context, to
\cite{ABH}),
allows us to obtain concentration bounds,
which we refer to as ``continuity estimates'', that are strong enough
to control the event that the barrier has been crossed by some 
``particle'' at some
intermediate level $l$.

A significant effort has to be invested in the proof of
the lower bound \eqref{goal.1L}. As in \cite{DPRZ,BK,BRZ}, we use
a second moment method. Similarly to
\cite{BK,BRZ}, we apply it 
together with a (linear) barrier, which means that
we need to compute probabilities of events of the form
\begin{equation}
  \label{eq-eqinA}
  \bigcap_{y\in \{y_1,y_2\}}\{\sqrt{2T_l^{y,t_{z}}}\geq  \rho_L (L-l)
  , l=1,\ldots,L-1, T_{L}^{y,t_{z}}=0 \}.
\end{equation}
(Actually, we consider a more complicated notion of \textit{good event}, see
\eqref{eq:TruncatedSummandLBzz}.)
 In the case
of the tree detailed in \cite{BRZ}, there is complete decoupling between
the excursions in different branches of the tree, given the number of excursions
at the edges below the common root of the branches. This exact 
decoupling is not true on $M$. In \cite{DPRZ} and \cite{BK}, this 
obstacle was circumvented by disregarding several levels in the barrier 
(i.e., those layers corresponding roughly to $\log d(y_1,y_2)/2\pm O(\log L)$),
and applying an estimate on the Poisson kernel for Brownian motion.
This results in a loss in the upper bound on the probability of the event in
\eqref{eq-eqinA}, which we cannot afford, especially
when $d(y_1,y_2)$ is relatively large.
Our way to circumvent this issue is to observe that when $l<L/2$,
there are many excursions at level $l$ before the cover time, and hence
the empirical measure of the angle between the starting and ending points
of each excursion approaches the equilibrium measure (in Wasserstein
distance). We then add the event that this Wasserstein distance
is not large to our definition of good event. On the good event, 
we can use Poisson kernel estimates to obtain a good decoupling, sufficient 
for an application of the second moment method, see 
Lemma \ref{lem-decoup}. Here too, working with $M=\S^2$
somewhat simplifies the analysis, but not significantly: the proof of the
lower bound carries over to general compact two-dimensional manifolds, although
we do not carry this out here.

Having invested all this work, transferring the results from $\S^2$ to the plane
is
then straightforward. The result concerning excursion counts extends immediately, by using stereographic projection. The control of cover time (measured only inside the disc $B_e(0,R)$) is then a concentration result. The details are provided in  
Section \ref{sec-genman}.
%We cover $M$ by a finite number of local patches that can be mapped  isothermally to  the plane, and compare these with the stereographic projection of a ball in  $\S^2$ centered at the south pole.
%This allows us to consider individually
%the number of excursions needed to cover each such patch.
%
%
%
%We cover $M$ by a finite number of local patches that can be maped isothermally (through the plane) to $\S^2$, and consider individually
%the number of excursions needed to cover each such patch. 
%Then, concentration
%inequalities control the fluctuations of the cover time of each of the
%patches and yield the tightness of the square-root of that cover time. The
%details are in Section \ref{sec-genman}.

\subsection{Background}\label{Background}
There is convincing evidence that, at the leading order,
the cover time of graphs and manifolds 
is closely related to extremes of the Gaussian free field  on the same space. 
Perhaps most striking is the sequence 
\cite{DLP,Ding,Zhai}, where it is established that for discrete graphs with
bounded degrees, the 
cover time normalized by the size of the edge set is asymptotic to
a (universal) constant multiple of the (square of) the 
maximum of the Gaussian free field (GFF)
on the graph. 
In dimension $d=2$, the GFF obtained by a discretization of $M$ 
with finer and finer mesh is a
logarithmically correlated centered 
Gaussian field.  In recent years, a theory has 
emerged concerning the leading order of the maximum of such fields
\cite{BDG01}, the tightness of the maximum \cite{BZ,DRZlog}, and even
the fluctuations of the maximum, see
\cite{BDZ,BL,DRZlog}. In particular, one has a log-log 
correction term for the centering of the maximum, and the
centered maximum has tight fluctuations as the mesh-size 
tends to $0$. However, as pointed out already in \cite{DZ} (for the tree) and
\cite{BK} (for the torus), the log-log terms do not match what they are in the
case of the GFF. 
As is clear from the latter papers and 
\cite{BRZ}, the mismatch in the log-log correction term
is not due to a basic difference between the behavior of the
cover time and the associated GFF. Rather, it is because Gaussian random walks
(in the case of maximum of the GFF) are essentially 
replaced by Bessel processes. However, even after this replacement, much work
needs to be done to obtain a properly decoupled tree structure,
and it is precisely this extra step that the current paper addresses.
(Compare with \cite{BRZ}, where cover time results for the tree are obtained in
a relatively straight-forward way from the barrier estimates for
critical geometric Galton Watson processes.)

\subsection{Open problems}
\label{sec-open}
We expect that Theorem \ref{theo-tightness} extends to general smooth two-dimensional compact manifolds. In addition,
   based on the analogy with the extrema of Branching random walks 
    and log-correlated Gaussian fields, one expects that Theorem 
    \ref{theo-tightness} should be replaced by the statement 
    \begin{eqnarray} 
    \label{eq-OP}
    &&\mbox{\em The sequence of random variables 
      $ \sqrt{\CC^{\ast}_{\eps,M}}- m_{\eps,M} $ converges }\nonumber\\
    && \mbox{\em  
    in distribution
  to a randomly shifted Gumbel random variable.}
\end{eqnarray}
As mentioned above, 
 a key step in proving such convergence would be 
the improvement of the tail estimates in Theorem \ref{theo-spheretight} for $z$
large,
which in turn would require a corresponding improvement of Theorem
\ref{theo-tightness}.

A first step in the direction of proving
(\ref{eq-OP}),  
by resolving the analogous question for random walk on the binary tree 
has recently been taken in \cite{oren}, by methods different from those employed in this paper. For a proof based on the methods here,
 see  the forthcoming \cite{amirjay}.
%\item Theorem \ref{theo-tightness} should remain true for any 
%  smooth, compact, connected,  
%  two-dimensional Riemannian manifold $M$ without boundary. 
%  The main step missing in obtaining such 
%  a result is replacing Lemma \ref{lem-equalmean} with a quantitative 
%  estimate on the difference in means appearing there. One would then need to
%  carry such a quantitative estimate through the proof.
%\end{enumerate}

\subsection{Structure of the paper}
Sections \ref{sec-isot}--\ref{sec-etrt} are devoted to a proof of
(\ref{goal.1}). The proof employs barrier estimates from \cite{BRZ} which are 
adapted to our needs in Appendix \ref{sec:BoundaryCrossing}, and
a comparison of excursion counts to excursion times around different 
centers, see 
Theorem \ref{theo-etrt} (whose proof is given in
Section \ref{sec-fromex-to-ct})
for a precise statement. The  comparison
heavily relies on the continuity estimates 
contained in Section \ref{sec-continest}, see Lemma \ref{lem-cont1}.
Theorem \ref{theo-etrt} is used again in Section \ref{sec-fromex-to-ct} 
to show that  
(\ref{goal.1}) implies   (\ref{goal.1asto}), which gives  
one side of \eqref{eq-tight}.  
Section \ref{sec-lowerboundex} is devoted to the proof of the
lower bounds (\ref{goal.1astow}) and \eqref{goal.1L}. 
After quickly showing, using again Theorem \ref{theo-etrt}, that
\eqref{goal.1L} implies \eqref{goal.1astow}, 
the rest of the section and most of the effort 
are devoted to the proof of \eqref{goal.1L};
a key ingredient is Lemma \ref{lem-decoup} (the decoupling lemma), proved 
in subsection \ref{sec-Decoupling}. Equipped with the lemma and 
the barrier estimates of Appendix \ref{sec:BoundaryCrossing}, the argument 
employs a second moment method (of a counting statistic). 
Subsection \ref{subsec-UB}
is devoted to a lower bound
on the first moment of the statistic, while subsections 
\ref{subsec-bulk}-\ref{subsec-veryearly} are devoted to an upper bound
on the second moment, divided into 
cases according to the distance between the points involved.
Finally, Section \ref{sec-genman} is devoted to the proof of Corollary 
\ref{cor-plane}.\\[2em]

\noindent
{\bf Acknowledgement} We thank two anonymous referees for a detailed reading of the paper. We particularly thank one of the referees for raising doubts concerning our reduction 
of the general compact manifold case.

\newpage
\subsection{Index of Notation}
The following are frequently used notation, and a pointer to the location where the definition appears.

$ \begin{array}{ll}
  t^{\ast}_{x, L},t^{\ast}_{L},F_l&\mbox{\eqref{deftast},\eqref{eq-tsar}}\\
\rho_{L}& \mbox{(\ref{eq:defofts})}\\
t_{z}&\mbox{(\ref{eq:defoftsr})}\\
l_{L},\alpha(l) & \mbox{(\ref{eq:AlphaBarrierDefd})}\\
B_{d}(x,r)&\mbox{\rm \eqref{eq-Bdxr}}\\
r_l,h_l&\mbox{\eqref{dr.1}}\\
  T_{x,l}^{y,n}=\trav{n}{y}{r_0}{r_1}{x}{r_{l-1}}{r_l},\trav{n}{x}{\tilde R}{\tilde r}{u}{R}{r}& \mbox{\eqref{eq-calgary}}\\
 T^{x,n}_l= \trav{n}{x}{r_0}{r_1}{x}{r_{l-1}}{r_l} & (\ref{eq-traver1})\\
 \tau_{x}(m)& \mbox{(\ref{eq-taux})}\\
s(z), s_{L}(z)& \mbox{(\ref{clt.2})}\\
I_u& \mbox{\rm \eqref{notation.31}}\\
\tilde c, q_0& \mbox{\rm \eqref{eq-q0c}}\\
\mathcal{A}_{z,d}&\mbox{\rm \eqref{17.0a}}\\
\wh z&  \mbox{\rm 
  \eqref{eq-hatz}}\\
\mathcal{B}_{t,l}&\mbox{\rm \eqref{17.defb}}\\
b(l,L,z,\th)&\mbox{\rm \eqref{eq-bdef}}\\
\mathcal{D}_{0,t,l}(   j)&
\mbox{\rm \eqref{eq-Ddef}}\\
 \mathcal{Q}&
 \mbox{\rm \eqref{defQ}}\\
\mathcal{G}_{l}&
\mbox{\rm   \eqref{defG}}\\
\wt t_{z}&
\mbox{\rm   \eqref{geq: to prove Glf}}\\
\mathcal{C}_{0,  t,l}&
\mbox{\rm \eqref{Cdef}}\\
T_{y,\wt r_{l}}^{0, t}, T_{\wt r_{l}}^{0, t}&
 \eqref{scale2}\\
\mathcal{B}^{ \ga ,k}_{t,l}& \mbox{\rm \eqref{zdek}}\\
\ga(l)& \mbox{(\ref{14.1})}\\
\Iyzh,\Iyz&
\mbox{\rm \eqref{eq:TruncatedSummandLBz},\eqref{eq:TruncatedSummandLBzz}}\\
\WW_{y,k}(n)& \mbox{(\ref{eq:LBTruncatedSumz})}\\
N_{k,a}&\mbox{\rm \eqref{nkdef.2}}\\
\HH_{k,a}&
\mbox{\rm \eqref{hka.1}}\\
k^+,k^{++}& \mbox{\rm \eqref{eq:DefOfKPlus}}\\
\BB_{y,k,L},\KK_{k,p,a}& \mbox{\rm \eqref{eq:JDownnew},\eqref{eq-kpa}}\\
\Bbb{  B}_{y,k+3,k^{ +} }^{j' },\widehat{\mathfrak{B}}_{y,k^{ ++}+1,L}& 
\mbox{\rm \eqref{167.1},\eqref{167.2}}\\
%\wt \BB_{y, i,b,j,L},\BB_{y,i,j, b}&
%\mbox{(\ref{eq:JDownzout})},\mbox{(\ref{eq:JDownz})}\\
 % \wh \BB_{y,i,j, b}& \mbox{(\ref{eq:JDownzi})}\\
{\mathcal G}_{k^{ +}}^y& \mbox{\rm \eqref{g-sigma}}\\
\Bbb{  B}_{y,k+3,k^{ +}, k^{ ++} }^{j', j'' }&\mbox{\rm \eqref{167.1q}}\\
\VV_{y,k}(   n),\gykm& \mbox{\rm \eqref{eq:LBTruncatedSumz8},\eqref{short}}\\
  \Phi_{k,a }, \mathcal{A}_{N,k}& \mbox{\rm \eqref{was6a},\eqref{12was.0}}\\
 % \mathcal{A}_{N,k}&\mbox{\rm \eqref{12was.0}}\\
  \ka_{a,b}&\mbox{\rm \eqref{kap.1}}
\end{array} $
\section{Isothermal coordinates and notation}
\label{sec-isot}

As explained in \cite[Section 8]{DPRZ}, the existence of a 
smooth isothermal coordinate system in each small disc allows 
us to  transform Brownian motion on $M$ to a time changed   
Brownian motion in the plane. Since hitting probabilities 
do not depend  on a time change, it follows that in these coordinates, 
for $\rho_{1}<\rho_{2}<\rho_{3}$,
\begin{equation}
\Pbm^{v\in\partial B_{e}\left(y,\rho_{2}\right)}\left(\hit_{\partial B_{e}\left(y,\rho_{1}\right)}<\hit_{\partial B_{e}\left(y,\rho_{3}\right)}\right)=
\frac{\log \({ \rho_{2}/ \rho_{3}}\)}{\log \({ \rho_{1}/ \rho_{3}}\)},\label{isc.1}
\end{equation}
where $\hit_{A}$ is the first  hitting time for 
$A$ and  $B_{e}\left(x,r \right)$ is the open Euclidean ball 
of radius $r$ centered at $x$. 

For $\S^2$,
isothermal coordinates are nothing but
stereographic projection. That is, we consider $\S^{2}$ as the unit sphere 
in $\R^{3}$ centered at $(0,0,1)$, so that $\S^{2}$ is tangent to 
$R^{2}:=\R^{2}\times\{0\}\subset \R^{3}$. 
The stereographic projection of $p=(p_{1}, p_{2},p_{3})\in \S^{2}$ is 
the point $\si (p)$ where the line between $p$ and the `North Pole' $\textbf{0}:=(0,0,2)$
intersects $R^{2}$. Thus
\begin{equation}
  \si (p_{1}, p_{2},p_{3})={\frac{(p_{1}, p_{2})}{ 1-p_{3}/2}}.\label{isc.2}
\end{equation}
We will refer to these as the isothermal coordinates centered at $(0,0,0)$. 
It can be shown that $\si$ maps circles into circles or lines.
As shown in \cite[p. 335]{spivak}, the stereographic projection 
 is an isometry if we give $R^{2}$ the metric
 \begin{equation}
   \frac1{(1+{\frac14}(x^{2}+y^{2}))^{2}}\(dx\otimes dx+dy\otimes dy\). \label{itc.3}
 \end{equation}
In this metric the   distance from $(0,0)$ to $(  \rho,0)$ is given by
\begin{equation}
  d((0,0), (  \rho,0))=\int_{0}^{ \rho}{\frac{1}{(1+{\frac14 }x^{2} )} }\,dx
=2\arctan ( \rho/2).\label{itc.4}
\end{equation}
Let
\begin{equation}
  \label{eq-Bdxr}
  B_{d}(x,r)= \{y\in \S^2: d(x,y)<r\}
\end{equation}
denote  the open
ball of radius $r$ in the spherical metric  centered at $x$. Setting
\begin{equation}
h(r)=2\arctan (r/2),\label{hdef}
\end{equation}
we get from (\ref{isc.1}) that 
\begin{equation}
\Pbm^{v\in\partial B_{d}\left(0,h (\rho_{2}) \right)}\left(\hit_{\partial B_{d}\left(0,h (\rho_{1})\right)}<\hit_{\partial B_{d}\left(0,h (\rho_{3})\right)}\right)=
\frac{\log \({ \rho_{2}/ \rho_{3}}\)}{\log \({ \rho_{1}/ \rho_{3}}\)}.
\label{isc.5}
\end{equation}

As in \cite{DPRZ} and \cite{BK}, an 
important role is played by the Poisson kernel  
for $B_{d}(   0, r)\subseteq \S^{ 2}$, denoted
$p_{B_{d}(   0, r)}(   z,x)$.
%denote 
% the Poisson kernel  for $B_{d}(   0, r)\subseteq \S^{ 2}$,  given by
Indeed, using isothermal coordinates, one readily sees that
  \begin{equation}
p_{B_{d}(0,r)}(   z,x)=\frac{\sin^{ 2} (   r/2)-\sin^{ 2} (   d(  0,z)/2)}{\sin^{ 2} ( d(  z,x)/2)}  \label{poisson}
  \end{equation}
  for $d(0,x)=r$ and $d(0,z)<r$, see \cite[p. 439]{E}. 
\subsection{Notation}
Throughout this paper, unless otherwise stated, \corO{constants $c$, $c_i$ and $C_i$ may change from line to line
but, unless noted otherwise, are universal 
and their value does not depend on additional parameters. Other constants (e.g., $\tilde c$, $q_0$, etc) will be fixed.}
Given $a,b>0$ which may depend on other parameters, we write
$a\asymp b$ if the ratio $a/b$ and $b/a$ are
bounded above. \corO{We write $a\ll b$ if $a/b\to 0$ as function of an implicit parameter, which is always clear from the context.}

Recall that $\Pbm^x$ denotes the probability measure for
Brownian motion
on the sphere started at a point $x$. We let $X_t$ denote the cannonical process
under 
$\Pbm^x$. When probabilities do not depend on $x$, we use $\Pbm$ instead of
$\Pbm^x$.

We need to introduce notation for various hitting times, excursion counts and
excursions times.
For a set $A$ of positive capacity, we let $\hit_A$ denote the hitting time of $A$ by Brownian motion.

We fix \corO{$r_0\in (0,10^{-6})$} small enough so that 
\begin{equation}
  \label{dr.1}
  r_l=r_0e^{-l}\; \mbox{\rm and} \; 
h_l=h(r_l)\;
\mbox{\rm satisfy} \;
0.9\,r_{l} \leq h_{l} \leq r_{l}\, \mbox{\rm for all $l=0,1,\ldots$},
\end{equation}
see \eqref{hdef}.

In addition to the traversal counts $T^{x,n}_l$,
the number of excursions from  $\partial B_{d}\left(x,h_{l-1}\right)$ to $ \partial B_{d}\left(x,h_{l}\right)$ prior to 
completing the $n$'th excursion from  
$\partial B_{d}\left(x,h_{1}\right)$ to $ \partial B_{d}\left(x,h_{0}\right)$,
we will often consider traversal counts between sometimes 
non-concentric annuli. The following notation will be particularly
useful.
\begin{definition}\label{def-trav}For any $0<r<R<\tilde{r}<\tilde{R}$,
	let $\trav{n}{x}{\tilde R}{\tilde r}{u}{R}{r}$ be the number of
	traversals  
	$\partial B_{d}(u,h(R))\to\partial B_{d}(u,h(r))$
	during $n$ excursions $\partial B_{d}(x,h(\tilde{r}))\to\partial B_{d}(x,h(\tilde{R}))$. 
\end{definition}
Note that with this notation
\begin{equation}
  \label{eq-traver1}
  T^{x,n}_l = \trav{n}{x}{r_0}{r_1}{x}{r_{l-1}}{r_l}.
\end{equation}
We will often abbreviate the notation in Definition 
\ref{def-trav}, writing e.g.
\begin{equation}
  \label{eq-calgary}
  T_{x,l}^{y,n}=\trav{n}{y}{r_0}{r_1}{x}{r_{l-1}}{r_l}.
\end{equation}

  We will need to consider certain traversal processes that ``start
at lower scales''. For each $k\ge1$ we define 
\begin{equation}
  \label{eq-traver2}
  T_{l}^{y,k,m} 
  =\trav{m}{y}{r_{k-1}}{r_{k}}{y}{r_{l-1}}{r_{l}}
\end{equation}
to be the number of traversals from scale $l-1$ to $l$ during the
first $m$ excursions from scale $k$ to scale $k-1$. Note the crucial ``compatibility'' property that
\begin{equation}
T_{l}^{y,k,m}=T_{l}^{y,t_{z}}\mbox{ for }l\ge k,\mbox{\,\ on }\left\{ m=T_{k}^{y,t_{z}}\right\}.\label{eq:TkandTCompatability}
\end{equation}

We also need to keep track of the time it takes to complete a prescribed
number of excursions. We set
\begin{eqnarray}
  \label{eq-taux}
  \nonumber
  \tau_x(m)&=&
  \mbox{\rm 
  time needed for Brownian motion to complete $m$ excursions}\\
  &&\mbox{\rm from  
$\partial B_{d}\left(x,h_{1}\right)$ to $\partial B_{d}\left(x,h_{0}\right)$}.
\end{eqnarray}
Similarly, for
$0<a<b<\pi$ we 
set
\begin{eqnarray}
  \label{eq-tauab}
  \nonumber
  \tau_{x, a,b}(m)&=& \mbox{\rm  time needed for Brownian motion
  to complete $m$ excursions}\\
 &&\mbox{\rm   from 
   $\partial B_{d}\left(x,a\right)$ to $\partial B_{d}\left(x,b\right)$}.
 \end{eqnarray}

 In addition to the notation $t_z$, see \eqref{eq:defoftsr}, 
 it is sometimes convenient 
 to consider its rescaled linear approximation, defined as
 \begin{equation}
s\left(z\right)=s_{L}\left(z\right)=L\left(2L-\log L+z\right).\label{clt.2}
\end{equation}

We let $\Pgw_n$ denote the law of a critical Galton-Watson process with geometric offspring distribution with initial offspring $n$. Using (\ref{isc.5}) and the strong Markov property,
 it is easy to see that
\begin{equation}
  \label{eq-Pn}
  \mbox{the }\Pbm\mbox{-law of }T^{x,n}_l,l\ge0,\mbox{ is }\Pgw_n.
\end{equation}

For any number $u$ we write
\begin{equation}
  I_{u}=[u,u+1).\label{notation.31}
\end{equation}

\section{The upper bound}
\label{sec-righttailex}

Let $r_l$ be as in \eqref{dr.1} and 
recall that $F_{l}$ are the centers of a minimal $h_{l}/1000$ cover of $M$.
We can and will assume that $F_{l}\subseteq F_{l+1}$.
We record for future
use that  
\begin{equation}
\left|F_{l}\right|\asymp r_{l}^{-2}=r^{-2}_{0} e^{2l},l\ge0.\label{dr.2}
\end{equation}
Recall also the excursion counts 
$T_{l}^{y, t_{z} }$ and the notation $t_z$ and $s(z)$, see \eqref{eq-traver1},
\eqref{eq:defoftsr} and \eqref{clt.2}.
\corO{Let 
\begin{equation}
\label{eq-q0c}
\tilde{c}=q_0/2<1/2
\end{equation}
 denote small constants, with $q_0$ chosen according to the deviation estimates in Lemma \ref{lem-cont2}.
For each $l$ we then choose $x_{i}\in F_{L}$ so that
 \begin{equation}
 \label{eq-FLgrid}
 F_{L}=\cup_{i=1}^{ce^{2l}}\{ F_{L}\cap B_{d}\left(x_{i},\tilde{c}h_{l}\right)\} \, \mbox{\rm  for some $c=c(\tilde{c})<\ff$.}
 \end{equation}}

Our goal in this section is to prove the upper bounds in
Theorems \ref{theo-tight} and \ref{theo-spheretight}, 
namely to prove \eqref{goal.1} and \eqref{goal.1asto}. 
As it turns out, both parts rely on an accurate comparison of real time needed 
to complete roughly $t_z$ traversals
between concentric circles on the sphere.
 Recall that
$\tau_{x}\(m\)$ denotes
the time needed to complete $m$ excursions from $\partial B_{d}\left(x,h_{1}\right)$ to $ \partial B_{d}\left(x,h_{0}\right)$, see \eqref{eq-taux}.

\bt\label{theo-etrt} Fix $r_0>0$ small. 
Then there exist
constants $d, z_{0},c$ (possibly depending on $r_0$) 
such that for $L $ sufficiently  large and all $z$ with $z_{0}\leq |z|\leq L^{1/2}\log^{2} L$,
\begin{equation}
\Pbm\( 4s_{L}\left(z-d\sqrt{|z|}\right)     \leq \tau_{x}\(s_{L}\left(z \right) \)\leq   4s_{L}\left(z+d\sqrt{|z|}\right),\,\forall x\in F_{L}\)\geq 1-ce^{-4|z|}.\label{r1.1}
\end{equation}
\et
\corO{The constant $4$ plays no particular role - except that it is important that $4>2$.} The proof of Theorem \ref{theo-etrt} appears in
Section \ref{sec-etrt}, and is based on the 
continuity
estimates provided in Section \ref{sec-continest}, see Lemma \ref{lem:4.58m}. 
Given Theorem \ref{theo-etrt}, most of the work in this section 
is in the proof of
the following proposition.

\begin{proposition}
\label{gprop:MainGoalEasyRegimemm}  There exists a constant  $c >0$ such that for
all $L $ sufficiently  large,   and all $  z\geq  0$,
\be
\Pbm\left(  
\inf_{y\in F_{L}} T_{L}^{y,t_{z}}=0
 \right)\le c(1+z)e^{-2z }e^{-\frac{z^{2}}{20L}\wedge { \frac{z}{4}}}. \label{g17.0mm}
\ee
\end{proposition}

To see what is involved in the proof,
we begin with a simple estimate. 
\bl
\label{lem:NotHitByrL} 
There exists a $c'>0$ so that,
for all $y\in R^{2}$, $x\notin B_{d}\left(y,h_{1}\right)$
and $z>0$,
\begin{equation}
\Pbm^{x}\left( T_{L}^{y,t_{z}}=0\right)\leq c'e^{-2L}L e^{ -2z-z^{2}/4L}.\label{eq:NotHitByrL}
\end{equation}
\el
\begin{proof}
It follows from (\ref{isc.5}) with $\rho_1=r_0, \rho_2=r_1$ and
$\rho_3=r_L$, and the strong Markov property, that
\be
\Pbm^{x}\left( T_{L}^{y,t_{z}}=0\right)=\left(1-\frac{1}{L}\right)^{t_{z}} \leq  e^{-\frac{t_{z}}{L}}.\label{base.1}
\ee
The estimate 
(\ref{eq:NotHitByrL})  then follows from (\ref{eq:defoftsr}) and the fact that the 
$ O(z\log L)/L$ term is bounded by $1+z^{2}/4L$.
\end{proof}
 
Using Lemma \ref{lem:NotHitByrL} and \eqref{dr.2} with $l=L$, 
a union bound  would give 
\begin{equation}
\Pbm^{x}\left( T_{L}^{y,t_{z}}=0\mbox{ for some }y\in F_{L}\right)\leq C L e^{ -2z-z^{2}/4L}.\label{eq:NotHitByrLa}
\end{equation}
The factor $L$ on the right hand side 
destroys any chance of using \eqref{eq:NotHitByrLa}
to obtain (\ref{g17.0mm}). However, if $  z> L^{1/2}\log  L$ it is easily seen that  \eqref{eq:NotHitByrLa}
implies  (\ref{g17.0mm}).
It thus remains to prove    (\ref{g17.0mm})    for $ 0\leq z\leq  L^{1/2}\log  L$.

To improve on \eqref{eq:NotHitByrLa},
we will use the fact that 
the events  $\{T_{L}^{y,t_{z}}=0\}_{y=y_1,y_2}$ are correlated because
if $\log d(y_1,y_2)\ll r$ then the number $T_{L-r}^{y_i,t_{z}}$
of traversals  around $y_i$ at level $L-r$ will be almost the same for $i=1,2$.
To deal effectively with this, we 
recall 
the barrier $\alpha\left(l\right)$ and the notation $l_L$, see
\eqref{eq:AlphaBarrierDefd}:
\[
\alpha\left(l\right) =\rho_{L}(L-l)  -
l_{L}^{\ga}, \quad 
l_{L} = (l\wedge (L-l)),\qquad \ga=.4 \]
Since $\alpha(L)=0$,  Proposition \ref{gprop:MainGoalEasyRegimemm} will follow from the next proposition. 

\begin{proposition}
\label{gprop:MainGoalEasyRegime}  There exists $c $ such that for
all $L $ sufficiently  large,   and all $ 0\leq z\leq  L^{1/2}\log^{2} L$,
\be
\Pbm\left( 
\exists x\in F_{L}\mbox{ and }1\le l\le L\mbox{ such that }
\sqrt{2T_{l}^{x,t_{z}}}\le\alpha(l)
 \right)\le c(1+z)e^{-2z }e^{-\frac{z^{2}}{20L}\wedge { \frac{z}{4}}}. \label{g17.0}
\ee
\end{proposition}

The proof of Proposition \ref{gprop:MainGoalEasyRegime}
will be provided in Sections 
\ref{sec-beearly}-\ref{sec-beearly2},
and is
split  into two cases. 
For  $l$ which are not too large, i.e. $l\leq L/2$,
we can deal with 
(\ref{g17.0}) one level at a time. This is 
the content of Section \ref{sec-beearly}. For larger $l$'s, 
which are handled in Section  \ref{sec-begen}, 
and in particular for $l=L$, we need to proceed 
inductively and make use of the facts established for lower levels.
This method can be traced back to Bramson's work \cite{BM78}.
Some crucial auxiliary estimates are postponed to Section \ref{sec-beearly2}.
Finally, Section \ref{sec-fromex-to-ct}  is devoted to the proof of
\eqref{goal.1asto}.

\subsection{$l $ not too large}\label{sec-beearly}

We begin with rephrasing the part of Proposition \ref{gprop:MainGoalEasyRegime}
pertaining to $l$ not too large.

\begin{proposition}
\label{prop:MainGoalEasyRegime}  There exists $c <\ff$ so that,
for all $L $ sufficiently  large   and all $ 0\leq z\leq   L^{1/2}\log^{2} L$,
\be
\Pbm\left(\begin{array}{c}
\exists x\in F_{L}\mbox{ and }1\le l\le L/2\\
\mbox{ such that }\sqrt{2T_{l}^{x,t_{z}}}\le\alpha(l)
\end{array}\right)\le ce^{-2z }e^{-\frac{z^{2}}{20L}\wedge { \frac{z}{4}}}. \label{17.0}
\ee
\end{proposition}

\begin{proof}
  Note the statement always holds if $z\leq z_0$ by 
  increasing $c$ if necessary. Hence, it suffices to prove the claim for 
  $z\geq z_0$, for some fixed $z_0$ to be determined.

  Set
  % \begin{equation}
  % \label{eq-hatz}
  % \wh z=z-d \sqrt{z}
% \end{equation}
% and 
\begin{equation}
  \mathcal{A}_{z,d}=\{ \tau_{y}\(t_{z-d\sqrt{z}}\)\leq \tau_{x}\(t_{z}\)\leq \tau_{y}\(t_{z+d \sqrt{z}}\),\,\forall x,y \in F_{L}\}.\label{17.0a}
\end{equation}
($\mathcal{A}_{z,d}$ is the good event in which the time to complete the 
``right'' number of excusions is comparable for different balls.)
Noting that $ t_{z}= s(2z)+O(z^2+z\log L)$
%$s(z)\sim  t_{z/2}$ 
for $z$ in the stated range, 
it follows from Theorem \ref{theo-etrt} (by modifying $d$ there if necessary)
that we can find  $d, z_{0},c$ such that for $L $ sufficiently  large and all $z_{0}\leq z\leq L^{1/2}\log^{2} L$
\begin{equation}
\Pbm\(\mathcal{A}_{z,d}\)\geq 1-ce^{-4z}.\label{17.0b}
\end{equation}
Therefore, 
we need only show  that for  $z_{0}\leq z\leq   L^{1/2}\log^{2} L$,
\be
\Pbm\left(\begin{array}{c}
\exists x\in F_{L}\mbox{ and }1\le l\le L/2\\
\mbox{ such that }\sqrt{2T_{l}^{x,t_{z}}}\le\alpha(l), \,\,\mathcal{A}_{z,d}
\end{array}\right)\le ce^{-2z }e^{-\frac{z^{2}}{20L}\wedge { \frac{z}{4}}}. \label{17.00}
\ee
Since 
$e^{-l_{L}^{\ga}}$ is summable, it thus suffices to show that 
 for all $1\le l\le L/2$,
\begin{equation}
  \mathfrak{P}_1^{(l)}:=\Pbm\left(\exists x\in F_{L}\mbox{ such that }\sqrt{2T_{l}^{x,t_{z}}}\le\alpha(l), \,\mathcal{A}_{z,d}\right)\le ce^{-2z    -
l_{L}^{\ga}}e^{-\frac{z^{2}}{20l}\wedge { \frac{z}{4}}}.\label{eq:MainGoalEasyRegime}
\end{equation}

 By a union bound, \corO{recall \eqref{eq-q0c} and \eqref{eq-FLgrid},} and then using $\mathcal{A}_{z,d}$ 
 and introducing the notation
 %the notation $\wh z$ from  \eqref{eq-hatz}, 
 %Set
   \begin{equation}
   \label{eq-hatz}
   \wh z=z-d \sqrt{z}
 \end{equation}
 we have
\begin{eqnarray}
  &&\;\;\mathfrak{P}_1^{(l)}
\le ce^{2l}\Pbm\left(\exists x\in F_{L}\cap B_{d}\left(0,\tilde{c}h_{l}\right)\mbox{ such that }\sqrt{2T_{l}^{x,t_{z}}}\le\alpha(l), \mathcal{A}_{z,d}\right)
\label{eq:displayfol} \\
&&\!\!\!\!\!\!\!\!\le 
ce^{2l}\Pbm\left(\exists x\in F_{L}\cap B_{d}\left(0,\tilde{c}h_{l}\right)\mbox{ such that }\sqrt{2T_{x,l}^{0,t_{\wh z}}}\le\alpha(l)\right)=:c e^{2l} \mathfrak{P}_2^{(l)},
\nn
\end{eqnarray}
where  $0$ is a fixed point
in $F_{L}$
(which could be taken as the south pole) and  we  recall 
that
  $T_{x,l}^{0,n}$
is the number of
traversals from  
$\partial B_{d}\left(x,h_{l-1}\right)\to\partial B_{d}\left(x,h_{l}\right)$
during $n$ excursions  
from $\partial B_{d}\left(0,h_{1}\right)\to\partial B_{d}\left(0,h_{0}\right)$, see
\eqref{eq-calgary}.

Hence to obtain (\ref{eq:MainGoalEasyRegime}) it suffices to show that
\be 
\mathfrak{P}_2^{(l)}
\le ce^{-2l-2z -l_{L}^{\ga}}e^{-\frac{z^{2}}{20l}\wedge { \frac{z}{4}}}. \label{ze13f}
\ee 
We  write $\mathfrak{P}_2^{(l)}\leq \mathfrak{P}_{21}^{(l)}+\mathfrak{P}_{22}^{(l)}$ where
\begin{eqnarray}
  \label{eq-P21}
  \mathfrak{P}_{21}^{(l)}&\!\!\!
:=&\!\!\!\!
\Pbm\left(\sqrt{2T_{l-2}^{0,t_{\wh z}}}\le\alpha(l)+1\right)\\
\label{eq-P22}
\mathfrak{P}_{22}^{(l)}&\!\!\!
:=&
\!\!\!\!
\Pbm\left(\mathcal{B}_{t_{\wh z},l}
\mbox{ and }\sqrt{2T_{l-2}^{0,t_{\wh z}}}\ge\alpha(l)+1\right),
\end{eqnarray}
and
\be
\mathcal{B}_{t,l}=\left\{ \exists x\in F_{L}\cap B_{d}\left(0,\tilde{c}h_{l}\right)\mbox{ s.t. }\sqrt{2T_{x, l}^{0,   t}}\le \alpha(l)\right\},\label{17.defb}
\ee
It is important to remember that $\mathcal{B}_{t,l}$
involves traversal counts centered at points which can differ from $0$.

Before proceeding we need to state some deviation inequalities of 
Gaussian type for
the Galton-Watson process $T_l,l\ge0$ under $\Pgw_n$, see 
\eqref{eq-Pn}.
The proof is very similar to \cite[Lemma 4.6]{BK},
and is therefore omitted.
\begin{lemma}\label{lem: GW proc LD}
There exists a constant $c$ such that for all $n, l =1,2,3,\ldots$,
\begin{equation}
\Pgw_n\left(\left|\sqrt{2T_{l}}-\sqrt{2T_{0}}\right|\ge \theta
	\right)
	\le ce^{-\frac{\theta^{2}}{2l}},\quad \theta\geq 0.\label{eq: tail bound}
\end{equation}
\end{lemma}
Recall, see \eqref{eq-Pn}, that 
$\{T_{l}^{x,t_{z}}\}_{l\geq 0}$
under $\Pbm$ is distributed like $\{T_l\}_{l\geq 0}$ under $\Pgw_{t_z}$.
Therefore,  we obtain the following estimates
from Lemma \ref{lem: GW proc LD}, for $z\geq 0$ and $\theta\in \R$:
\begin{equation}
  \Pbm\left(\sqrt{2T_{l}^{x,t_{z}}}\le\alpha(l)+\th\right)\le ce^{-\(\sqrt{2t_{z}}-\alpha(l)-\th\)^{2} /2l},\mbox{ if } \(\alpha(l)+\th\)^{2}/2\leq t_{z},
  \label{17.1p}
\end{equation}
and 
\begin{equation}
  \Pbm\left(\sqrt{2T_{l}^{x,t_{z}}}\ge\alpha(l)+\th\right)\le c e^{-\(\sqrt{2t_{z}}-\alpha(l)-\th\)^{2} /2l}, \mbox{ if } \(\alpha(l)+\th\)^{2}/2\geq t_{z}.
  \label{17.1gp}
\end{equation}
Using the definitions of $\alpha(l)$ and $t_z$, see \eqref{eq:defoftsr} and
\eqref{eq:AlphaBarrierDefd}, we have that 
\begin{eqnarray*}
 && \frac{(\sqrt{2t_{z}}-\alpha(l)-\th)^2}{2l}
  = \frac{(2l- (l/2L)\log L+z+l_{L}^{\ga}-\th)^2}{2l}\\
%\end{eqnarray*}%
%If $\sqrt{2t_{z}}=\rho_{L}L+z= 2L-{\frac{1}{2}}\log L+z$ and 
%$\alpha\left(l\right) = \rho_{L}(L-l) -
%l_{L}^{\ga}=2(L-l)- \frac{L-l}{2L}\log L  -
%l_{L}^{\ga}$, then
%\begin{equation}
%\sqrt{2t_{z}}-\alpha(l)-\th=2l- \frac{l}{2L}\log L+z+l_{L}^{\ga}-\th.\label{multdef.0}
%\end{equation}
%We write out
%\begin{eqnarray}
%&&\left(\sqrt{2t_{z}}-\alpha(l)-\th \right)^{2}/2l
%\label{multdef.2}\\
%&& = \left(2l- \frac{l}{2L}\log L+z+l_{L}^{\ga}-\th\right)^{2}/2l \nn\\
  &&=2l+ 2(z +   l_{L}^{\ga} - \th) -\frac{l\log L}{L}  
%\nn \\
%&&\hspace{1 in}
  + \frac{\left((z +   l_{L}^{\ga} - \th) - \frac{l}{2L}\log L\right)^{2}}
  {2l}\\
%\end{eqnarray*}
%Using  $(r-s)^{2}\geq r^{2}/2-s^{2}$ we see that 
%\begin{equation}
&&\geq
2l+ 2(z +   l_{L}^{\ga} - \th) -\frac{l\log L}{L}+
%\left((z +   l_{L}^{\ga} - \th) - \frac{l}{2L}\log L\right)^{2}/2l \geq  
\frac{(z +   l_{L}^{\ga} - \th)^{2}}{ 4l}-o_{L}(1).\label{multdef.1}
\end{eqnarray*}
%\end{equation} 
Therefore, with 
\begin{equation}
  \label{eq-bdef}
  b(l,L,z,\th):=L^{ l/L}e^{-2l-2(z-\th+ l_{L}^{\ga})-(z +   l_{L}^{\ga} - \th)^{2} /4l },
\end{equation}
we have 
\begin{equation}
  \Pbm\left(\sqrt{2T_{l}^{x,t_{z}}}\le\alpha(l)+\th\right)\le 
  cb(l,L,z, \th),
  %cL^{ l/L}e^{-2l-2(z-\th+ l_{L}^{\ga})-(z +   l_{L}^{\ga} - \th)^{2} /4l },
  \mbox{ if } \(\alpha(l)+\th\)^{2}/2\leq t_{z},
  \label{17.1}
\end{equation}
and 
\begin{equation}
  \Pbm\left(\sqrt{2T_{l}^{x,t_{z}}}\ge\alpha(l)+\th\right)\le 
  cb(l,L,z, \th),
  %cL^{ l/L}e^{-2l-2(z-\th+ l_{L}^{\ga})-(z +   l_{L}^{\ga} - \th)^{2} /4l}, 
  \mbox{ if } \(\alpha(l)+\th\)^{2}/2\geq t_{z}.
  \label{17.1g}
\end{equation}

Applying 
\eqref{17.1} with $\theta=1$ and $x=0$,
%and \eqref{17.1g}  
and using that
$L^{ l/L}e^{-  l_{L}^{\ga}/2}\leq 1$ for 
$1\le l\le L/2$,
we have that
for $l$ in that range
\be
\mathfrak{P}_{21}^{(l)}
\le ce^{-2l-2z+ 2d \sqrt{z}-1.5\,l_{L}^{\ga}-\wh z^{2}/4l}.\label{17.1xa}
\ee
With $\ga=0.4$, if $z^{1.4}\leq l$ then $z^{0.54}\leq l^{\ga}=l_{L}^{\ga}$, while if $z^{1.4}\geq l $ then $z^{2}/l\geq    z^{0.6} $. 
It follows that
\be
\mathfrak{P}_{21}^{(l)}
\le o_{z}(1) e^{-2l-2z -l_{L}^{\ga}-z^{2}/8l}.\label{17.1x}
\ee

Hence to conclude the proof of Proposition \ref{prop:MainGoalEasyRegime} it will suffice to show that
\be
\label{17.1s}
\mathfrak{P}_{22}^{(l)}\leq
c e^{-2l-2z -l_{L}^{\ga}-\( { \frac{z^{2}}{ 20l}}\wedge { \frac{z}{4}}\)}.
\ee

To obtain \eqref{17.1s}, our strategy will be to 
replace the events 
$\mathcal{B}_{t_{\wh z},l}$ 
by events involving excursions around $0$. Toward this end, recall
that $I_u=[u,u+1)$, see 
\eqref{notation.31},  and define the ``$(l-2)$-endpoint event''
\begin{equation}
  \label{eq-Ddef}
\mathcal{D}_{0,t,l}(   j)=\left\{ \sqrt{2T_{l-2}^{0,t}}\in 
I_{\alpha(l)+j}\right\} .
\end{equation}
Then
\bea
\mathfrak{P}_{22}^{(l)}&=&\sum_{j=1}^{ \ff}
\Pbm
\left(\mathcal{B}_{t_{\wh z},l}\cap \mathcal{D}_{0,t_{\wh z },l}(   j)\right)
=\sum_{j=1}^{ \ff}
\Pbm\left(\mathcal{B}_{t_{\wh z},l}\,\big|\,\mathcal{D}_{0,t_{\wh z },l}(   j)\right)\cdot \Pbm\left(\mathcal{D}_{0,t_{\wh z },l}(   j)\right)\nn\\
&&\leq \sum_{j=1}^{ \ff}
\Pbm\left(\mathcal{B}_{t_{\wh z},l}\,\big|\,\mathcal{D}_{0,t_{\wh z },l}(   j)\right)\,\,ce^{-2l-2z+2d \sqrt{z}+2j-1.5\, l_{L}^{\ga}-{(\wh z+l_{L}^{\ga} - j)}^{2} /4l },
\label{eq-finalP22}
\eea
where the last inequality follows again from  
the 
deviations estimates 
(\ref{17.1}) or (\ref{17.1g}) as appropriate.
We now state the following lemma, whose proof is postponed to
subsection 
 \ref{sec-beearly2}.
\begin{lemma}
  \label{lem-sec3.3}
  There exist positive constants  $\tilde{c}$ and $j_0$ so that,
  with $\mathcal{B}_{t_{\wh z},l}$ as in \eqref{17.defb}, one has
  that for all $  z \geq 0$,
\begin{equation}
\Pbm\left(\mathcal{B}_{t_{\wh z},l}\,\big|\, \mathcal{D}_{0,t_{\wh z },l}(   j)\right)\leq e^{ -4j}, 
\mbox{ for all } j\geq j_0.\label{17.1sv}
\end{equation}
\end{lemma}

Substitute \eqref{17.1sv} into \eqref{eq-finalP22}  and consider separately the case where $j\leq \wh z/2$  and  $j\geq \wh z/2$. In the first case we have the bound 
\be 
 \leq \sum_{j=1}^{z/2}
e^{ -4j}\,\,ce^{-2l-2z+2d\sqrt{z}+2j- 1.5\, l_{L}^{\ga} -{\wh z}^{2} /16l }
\label{multdef.3}\ee
which can be bounded by (\ref{17.1s}) as before. In the second case we simply use
\be 
 \leq \sum_{j=z/2}^{\ff}
e^{ -4j}\,\,ce^{-2l-2z+2d \sqrt{z}+2j- 1.5\, l_{L}^{\ga}  }.
\label{multdef.4}\ee
Since $j\geq {\wh z}/2$ this gives 
\eqref{17.1s} and completes the proof of Proposition 
\ref{prop:MainGoalEasyRegime}.
\end{proof}

\subsection{$l$   large and proof of Proposition \ref{gprop:MainGoalEasyRegime}}\label{sec-begen}
Recall
%, see \eqref{eq-taux},
that 
$\tau_{x}\(m\)$ denotes the time needed to 
complete $m$ excursions from  
$\partial B_{d}\left(x,h_{0}\right)$ 
to $\partial B_{d}\left(x,h_{1}\right)$,
 see \eqref{eq-taux}.
We begin by stating a (simpler) version of Theorem
\ref{theo-etrt}, whose proof is also given in Section
\ref{sec-etrt}.

\bt\label{theo-etrtc} There exists $c>0$ so that
for $L $ sufficiently  large and all $0\leq z\leq L^{1/2}\log^{2} L$,
\bea
&&
\Pbm\( \tau_{y}\(t_{z}-10\)\leq \tau_{x}\(t_{z}\)\leq \tau_{y}\(t_{z}+10\),\,\forall x,y \in F_{L}, \, d(   x,y)\leq h(r_{L/2})\)\nn\\
&&\hspace{3.5 in}\geq 1-ce^{-L/2}.\label{r1.1cl}
\eea
\et

Set 
\begin{equation}\hspace{.2in}
  \mathcal{Q}=\{ \tau_{y}\(t_{z}-10\)\leq \tau_{x}\(t_{z}\)\leq \tau_{y}\(t_{z}+10\),\,\forall x,y \in F_{L}, \, d(   x,y)\leq h(r_{L/2})\}.\label{defQ}
\end{equation}
It follows from Theorem \ref{theo-etrtc} that   
for $L $ sufficiently  large and all $ 0\leq  z\leq L^{1/2}\log^{2} L$,
\begin{equation}
\Pbm\(\mathcal{Q}\)\geq 1-ce^{-L/2}.\label{g17.0b}
\end{equation} 

We fix a 
small  constant $\tilde{c}$, to be chosen later. 
Introduce the ``barrier event''
\begin{equation}
  \label{defG}
\mathcal{G}_{l}=\left\{ \sqrt{2T_{l'}^{y,t_{z}}}\ge \alpha(l')\mbox{ for all }l'=1,\ldots,l\mbox{ and }\forall y\in F_{L}\cap B_{d}\left(0,\tilde{c}h_{0}\right)\right\} .
\end{equation}
Let $L'=L/2$.
By Proposition \ref{prop:MainGoalEasyRegime}, we have that
for all $L$ large and $0\leq z\leq L^{1/2}\log^{2} L$,
\begin{equation}
  \label{eq-thanks1}
  \Pbm\left(\mathcal{G}^{c}_{L'}  \right)\leq c e^{-2z }e^{-\frac{z^{2}}{20L}\wedge { \frac{z}{4}}}.
\end{equation}
We will prove in this section the following lemma. The proof of the lemma
uses
some continuity estimates from Section \ref{sec-continest} below, and barrier
estimates from \cite{BRZ} which are discussed in Appendix 
\ref{sec:BoundaryCrossing}.

\begin{lemma}
  \label{lem-westjet}
There exists a constant $c>0$  so that, for all $l>L'$ and $1\leq z\leq L^{1/2}\log^{2} L$,
\begin{equation}
\Pbm\left(\mathcal{G}_{l}^{c}\cap\mathcal{G}_{l-2}\cap\mathcal{Q}\right)\le 
c(z+1)e^{- l_{L}^{\ga}-2z}e^{-\frac{z^{2}}{20l}\wedge { \frac{z}{4}}}.\label{geq: to prove Gl}
\end{equation}
\end{lemma}

Assuming Lemma \ref{lem-westjet}, we can complete 
the proof of Proposition \ref{gprop:MainGoalEasyRegime}.
\begin{proof}[Proof of Proposition \ref{gprop:MainGoalEasyRegime}.]
  From 
\eqref{geq: to prove Gl}, one has
\bea
\Pbm\left(\mathcal{G}_{L}^{c}\cap\mathcal{Q}\right)&&\le\sum_{l=L'+1}^{L}\Pbm\left(\mathcal{G}_{l}^{c}\cap\mathcal{G}_{l-1}\cap\mathcal{Q}\right)+\Pbm\left(\mathcal{G}^{c}_{L'}  \right)\nn\\
&&\le\sum_{l=L'+1}^{L}\Pbm\left(\mathcal{G}_{l}^{c}\cap\mathcal{G}_{l-2}\cap\mathcal{Q}\right) +\Pbm\left(\mathcal{G}^{c}_{L'}  \right)\nn\\
&&\le\sum_{l=L'+1}^{L}cze^{-
   l_{L}^{\ga}-2z }e^{-\frac{z^{2}}{20l}\wedge { \frac{z}{4}}}+\Pbm\left(\mathcal{G}^{c}_{L'} 
  \right) \nn\le cze^{-2z }e^{-\frac{z^{2}}{20L}\wedge { \frac{z}{4}}},\label{cbsum}
\eea
where the last inequality used 
\eqref{eq-thanks1}. Combined with \eqref{g17.0b}, we conclude that
$\Pbm\left(\mathcal{G}_{L}^{c}\right)\leq 
   cze^{-2z }$. A simple union bound (over $\sim (1/\tilde c h_0)^2$ balls) then completes
   the proof of Proposition \ref{gprop:MainGoalEasyRegime}.
 \end{proof}

We turn to proving 
Lemma \ref{lem-westjet}.
\begin{proof}[Proof of Lemma \ref{lem-westjet}]
By a union bound as in \eqref{eq:displayfol},
%the display following (\ref{eq:MainGoalEasyRegime}), 
$\Pbm\left(\mathcal{G}_{l}^{c}\cap\mathcal{G}_{l-2}\cap\mathcal{Q}\right)$
is bounded above by
\begin{equation}
ce^{2l} \Pbm\left(\left\{ \exists x\in F_{L}\cap B_{d}\left(0,\tilde{c}h_{l}\right)\mbox{ s.t. }\sqrt{2T_{l}^{x,t_{z}}}\le \alpha(l)\right\} \cap\mathcal{G}_{l-2}\cap\mathcal{Q}\right).\label{geq: union 1}
\end{equation}
On $\mathcal{Q}$ we have that $\left\{ \exists x\in F_{L}\cap B_{d}\left(0,\tilde{c}h_{l}\right)\mbox{ s.t. }\sqrt{2T_{l}^{x,t_{z}}}\le \alpha(l)\right\} $
implies the event, see \eqref{17.defb}, 
\[
\mathcal{B}_{\wt t_{z},l}=\left\{ \exists x\in F_{L}\cap B_{d}\left(0,\tilde{c}h_{l}\right)\mbox{ s.t. }\sqrt{2T_{x, l}^{0,  \wt t_{z}}}\le \alpha(l)\right\},
\]
where $\wt t_{z}=t_{z}-10 $ and, recall \eqref{eq-calgary},
$T_{x,l}^{0,n}$ is the number of
traversals from  $\partial B_{d}\left(x,h_{l-1}\right)\to\partial B_{d}\left(x,h_{l}\right)$
during $n$ excursions
from $\partial B_{d}\left(0,h_{1}\right)\to\partial B_{d}\left(0,h_{0}\right)$.
Hence to prove (\ref{geq: to prove Gl}) it suffices to show that
  for all $l>L'$
\begin{equation}
\Pbm\left(\mathcal{B}_{\wt t_{z},l}\cap\mathcal{G}_{l-2} \right)\le c(z+1)e^{-2l-
 l_{L}^{\ga}-2z}e^{-\frac{z^{2}}{16l}\wedge { \frac{z}{4}}}.\label{geq: to prove Glf}
\end{equation}
Set
\be
A_{l}=T_{l}^{0,t_{z}}-T_{l}^{0,\wt t_{z}}\overset{\mbox{law}}{=}T_{l}^{0,10}.\label{gcb.9.1}
\ee
We bound 
\[\Pbm\left(\mathcal{B}_{\wt t_{z},l}\cap\mathcal{G}_{l-2} \right)\leq 
\Pbm\left(\mathcal{B}_{\wt t_{z},l} ;\,A_{L/2}>0 \right)+
\Pbm\left(\mathcal{B}_{\wt t_{z},l}\cap\mathcal{G}_{l-2};\,A_{L/2}=0 \right), 
\]
and estimate each term separately.

Using the independence of $\mathcal{B}_{\wt t_{z},l}$ and $A_{l}$ we have that
\begin{equation}
\Pbm\left(\mathcal{B}_{\wt t_{z},l} ;\,A_{L/2}>0 \right)= 
\Pbm\left(\mathcal{B}_{\wt t_{z},l} \right)
\Pbm\left(A_{L/2}>0 \right)\leq 
c\Pbm\left(\mathcal{B}_{\wt t_{z},l} \right)L^{ -1},\label{ez.123}
\end{equation}
where we have used \eqref{eq-Pn} and the 
fact that the survival probability 
up to generation $L/2$ 
under $\Pgw_{10}$
is bounded by $cL^{ -1}$.

We claim that 
\begin{equation}
\Pbm\left(\mathcal{B}_{\wt t_{z},l} \right)\leq 
cLe^{-2l- l_{L}^{\ga}-2z}e^{-\frac{z^{2}}{4l}\wedge { \frac{z}{4}}},\;\; l\leq L.\label{new.4}
\end{equation}
To see this, we bound $\Pbm\left(\mathcal{B}_{\wt t_{z},l} \right)\leq \mathfrak{P}_{31}^{(l)}+\mathfrak{P}_{32}^{(l)}$ where
\begin{eqnarray}
  \label{eq-P21n}
  \mathfrak{P}_{31}^{(l)}&\!\!\!
:=&\!\!\!\!
\Pbm\left(\sqrt{2T_{l-2}^{0,{\wt t_{ z}}}}\le\alpha(l)+1\right)\\
\label{eq-P22n}
\mathfrak{P}_{32}^{(l)}&\!\!\!
:=&
\!\!\!\!
\Pbm\left(\mathcal{B}_{t_{\wt z},l}
\mbox{ and }\sqrt{2T_{l-2}^{0,{\wt t_{ z}}}}\ge\alpha(l)+1\right).
\end{eqnarray}
It follows from \eqref{17.1} and \eqref{17.1g} that $\mathfrak{P}_{31}^{(l)}$ is bounded by the right hand side of  (\ref{new.4}). To bound
$\mathfrak{P}_{32}^{(l)}$ we follow the proof of the bound (\ref{17.1s}) we obtained for $\mathfrak{P}_{22}^{(l)}$. 
The only difference is that now in the analogue of (\ref{eq-finalP22}) we obtain an extra factor of $L$.

Combining (\ref{new.4}) with (\ref{ez.123}) we see that to establish (\ref{geq: to prove Glf})  it suffices to show 
that for all $l>L'$,
\begin{equation}
\Pbm\left(\mathcal{B}_{\wt t_{z},l}\cap\mathcal{G}_{l-2};\,A_{L/2}=0  \right)\le 
c(z+1)e^{-2l-l_{L}^{\ga}-2z}e^{-\frac{z^{2}}{16l}\wedge { \frac{z}{4}}}.\label{geq: to prove Glg}
\end{equation}
However, $A_{L/2}=0 $ implies that $T_{m}^{0,t_{z}}=T_{m}^{0,\wt t_{z}}$ 
for all $m\geq L/2$. Since the $x$'s in $\mathcal{B}_{\wt t_{z},l}$ 
are all in $B_{d}\left(0,\tilde{c}h_{l}\right)$, it follows that for such 
$x$'s, $ T_{x, l}^{0,  \wt t_{z}} = T_{x, l}^{0,  t_{z}} $.  
Thus on $\{A_{L/2}=0\}$ we have 
$\mathcal{B}_{\wt t_{z},l}=\mathcal{B}_{t_{z},l}$.

Since $\mathcal{G}_{l-2}\subset\mathcal{C}_{0,  t_{z},l}$, where 
\begin{equation}
  \label{Cdef}
\mathcal{C}_{0,  t,l}:=\left\{ \sqrt{2T_{l'}^{0,t}}\ge \alpha(l')\mbox{ for all }l'=1,\ldots,l-2\right\},
\end{equation} 
it suffices to show that for $l\geq L'$,
\begin{equation}
\Pbm\left(\mathcal{B}_{t_{z},l} \cap\mathcal{C}_{0,  t_{z},l}  \right)\le c(z+1)e^{-2l-
    l_{L}^{\ga}-2z}e^{-\frac{z^{2}}{16l}\wedge { \frac{z}{4}}}.\label{gsuf1}
\end{equation}

Let
\[
  \mathcal{D}_{0,  t_{z},l}(j)=\left\{ \sqrt{2T_{l-2}^{0,  t_{ z }}}\in I_{\alpha(l)+j}\right\} .
\]
Using \eqref{17.1}-\eqref{17.1g}, by (\ref{gsuf1}) it suffices to show that 
\begin{equation}
 \sum_{j=0}^{8L}\Pbm\left(\mathcal{B}_{t_{z},l}\cap\mathcal{C}_{0,  t_{z},l} 
\cap\mathcal{D}_{0,  t_{z},l}(j)\right)\le
c(z+1)e^{-2l- l_{L}^{\ga}-2z}e^{-\frac{z^{2}}{16l}\wedge { \frac{z}{4}}}.\label{gcb.9dec}
\end{equation}

The following analogue of Lemma \ref{lem-sec3.3} will be proved in
Section \ref{sec-beearly2} below.

\begin{lemma}
  \label{lem-sec3.3a}
 There exist constants $j_0,C,z_0$ such that, for 
 all $j\geq j_{0}$, $l\leq L$  and $  z \geq 0$, 
\begin{equation}
\Pbm\left(\mathcal{B}_{t_{z},l}\,\big |\,\mathcal{C}_{0,  
t_{z},l} \cap\mathcal{D}_{0,  t_{z},l}(j)\right)      \le Ce^{-4j }. \label{cb.15y}
\end{equation} 
\end{lemma}

 We continue with the proof of 
Lemma \ref{lem-westjet}.
Note that $\mathcal{C}_{0,  t_{z},l} \cap\mathcal{D}_{0,  t_{z},l}(j)$ is a 
barrier event in the sense discussed in \cite{BRZ}. Based on the latter
paper,  we 
develop 
in the appendix the barrier estimates in the form that we need here. 
In particular, 
it follows from (\ref{18.20}) in the appendix that
\bea
&&
\Pbm\left(\mathcal{C}_{0,  t_{z},l} \cap\mathcal{D}_{0,  t_{z},l}(j)\right)\leq  
ce^{-2l-2z-2 l_{L}^{\ga}+2j}\label{cb.decp7}\\
&&\hspace{1.5 in}  \times \left(1+z+l_{L}^{\ga}
\right)\left(1+j   \right)e^{-\frac{\left( z +  l_{L}^{\ga} - j \right)^{2}}{4l}}.\nn
\eea

   We break the sum in (\ref{gcb.9dec}) into a sum over two intervals, $[0,( z +  l_{L}^{\ga})/2] $, 
  and $[( z +  l_{L}^{\ga})/2, 8L] $.   In the first interval we use
 \[e^{-\frac{\left( z +  l_{L}^{\ga} - j \right)^{2}}{4l}}\leq e^{-\frac{(z +  l_{L}^{\ga})^{2}}{16l}}\leq e^{-\frac{z^{2}}{16l}}.\] 
 For the last interval we ignore the last factor in (\ref{cb.decp7}) and use $e^{-j}\leq e^{-z/2}$.

Putting this all together, 
we can bound from 
above the left hand side of (\ref{gcb.9dec}) by  
\begin{equation}
  C(z+1)e^{-2l-2z- l_{L}^{\ga}}e^{-\frac{z^{2}}{16l}\wedge { \frac{z}{4}}} \sum_{j=0}^{\infty}e^{-3j1_{\{j\geq j_{0}\}} +2j}
(1+|j|)\label{gcb.9dec3},
\end{equation}
which proves (\ref{gcb.9dec}) and completes the proof of Lemma \ref{lem-westjet}.
\end{proof}

\subsection{Proof of the conditional barrier estimates}
\label{sec-beearly2}
We prove in this section Lemmas \ref{lem-sec3.3} and \ref{lem-sec3.3a},
whose statements boil down to the estimates 
$\Pbm\left(\mathcal{B}_{t_{\wh z},l}\,\big|\, 
\mathcal{D}_{0,t_{\wh z },l}(   j)\right)\leq e^{ -4j}$, 
and  
$\Pbm\left(\mathcal{B}_{ t_{z},l}\,\big |\,\mathcal{C}_{0,  t_{z},l}
\cap\mathcal{D}_{0,  t_{z},l}(j)\right)      \le Ce^{-4j }, $
see 
(\ref{17.1sv}) and
(\ref{cb.15y}).
We intend to give a proof that will cover both cases. 
It will be seen from the proof that the time, $t_{\wh z}$ or $  t_{z}$, 
does not play a role in the proof. Hence we shall write it as $t$. 
In addition, we will see that the extra conditioning on 
$\mathcal{C}_{0,  t_{z},l} $ present in \eqref{cb.15y} is
not significant. 
\begin{proof}[Proof of Lemma \ref{lem-sec3.3a}]
Fix $\beta\in (0,1/2)$,
\begin{equation}
\wt r_{l-1}^+=r_{l-1}\left(1+\bb \right) ,\hspace{.2 in}\wt r_{l}=r_{l}\left(1-
\bb \right), \label{scale2}
\end{equation} 
and consider the excursions count 
$T_{y,\wt r_{l}}^{0, t}:=
\trav{t}{0}{r_0}{r_1}{y}{\wt r_{l-1}^+}{\wt r_l}$,
compare with \eqref{eq-calgary}, writing
$T_{0,\wt r_{l}}^{0, t}=
T_{\wt r_{l}}^{0, t}$, compare with \eqref{eq-traver1}.
Note that
for $y,y'$ with 
$d\left(y,y'\right)\le{\bb  r_{l} / 2}$, we have (using \eqref{dr.1}) that
\[
B_{d}\left(y',h(\wt r_l)\right)\subset B_{d}\left(y,h_{l}\right)
\subset B_{d}\left(y,h_{l-1}\right)\subset 
B_{d}\left(y',h(\wt r_{l-1}^+)\right),
\]
and therefore, writing $t=t_z$ throughout,
\be
T_{y',\wt r_{l}}^{0,t}\le T_{y,l}^{0,t}\mbox{ for all }y
\mbox{ and }y'\mbox{ such that }d\left(y,y'\right)\le{\bb  r_{l} / 2}.\label{cent1}
\ee

Let
\be
 \mathcal{B}^{ \bb ,k}_{t,l}=\left\{ \exists y\in F_{k}\cap B_{d}\left(0,\tilde{c}h_{l}\right)\mbox{ such that }\sqrt{2T_{y,\wt r_{l}}^{0, t}}\le \alpha(l)\right\}.\label{zdek}
\ee
From now on we fix
\be
\bb =\frac{1}{\alpha(l)+j}\quad\mbox{\rm and} \quad
k=\log\left(2(\alpha(l)+j)\right)+l.\label{conv-k}
\ee
We will show that with these values,
\begin{equation}
  \mathfrak{P}_3=
  \mathfrak{P}_3(j)
  :=\Pbm\left(\mathcal{B}^{  \bb ,k}_{t,l}\,\bigg|\,\mathcal{C}_{0,t,l}\cap\mathcal{D}_{0,t,l}(   j)\right)      \le Ce^{-4j}.\label{eq:to show}
\end{equation}
Using (\ref{cent1}) this will imply \eqref{cb.15y},
since  
for each $y\in F_{L}\cap B_{d}\left(0,\tilde{c}h_{l}\right)$ there exists  a representative $y'\in F_{k}\cap B_{d}\left(0,\tilde{c}h_{l}\right)$
such that 
\[
d\left(y,y'\right)\le r_{k}=\frac{1}{2(\alpha(l)+j)}\,\,r_{l}={\bb  r_{l} / 2}.
\]

We thus turn to proving \eqref{eq:to show}.
We bound
\bea
\label{eq-P3}
\mathfrak{P}_3
&\le& \Pbm\left(\sqrt{2T_{\wt r_{l}}^{0, t }} \le\alpha(l)+\frac{j}{2}\,
\bigg|\,\mathcal{C}_{0,t,l}\cap\mathcal{D}_{0,t,l}(j)\right)\nn\\
&& +\Pbm\left(\mathcal{B}^{ \bb ,k}_{t,l}\cap 
\Big\{\sqrt{2T_{\wt r_{l}}^{0, t }}>\alpha(l)+\frac{j}{2}\Big\}\,\bigg|\,
\mathcal{C}_{0,t,l}\cap\mathcal{D}_{0,t,l}(j)\right)=:\mathfrak{P}_{31}+
\mathfrak{P}_{32}.
\eea
(Do not confuse $\mathfrak{P}_{31}$ and $\mathfrak{P}_{32}$
with $\mathfrak{P}_{31}^{(l)}$ and $\mathfrak{P}_{32}^{(l)}$
from \eqref{eq-P21n} and \eqref{eq-P22n}.
Note also 
that both $\mathfrak{P}_{31}$ and $\mathfrak{P}_{32}$ depend on $j$, but
we continue to supress the dependence in the notation.)

We first bound $\mathfrak{P}_{31}$.
Note that   given $T_{l'}^{0, t}$ for all  $l'=1,\ldots,l-2$, it follows from the Markov property that 
$T_{\wt r_{l}}^{0, t }$ depends only on $T_{l-2}^{0,t }$, \corJ{and if $m=T_{l-2}^{0,t }$ then  $T_{\wt r_{l}}^{0, t }=\trav{t}{0}{r_0}{r_1}{0}{\wt r_{l-1}^+}{\wt r_l}=\trav{m}{0}{r_{l-3}}{r_{l-2}}{0}{\wt r_{l-1}^+}{\wt r_l}$}. Hence 
\be
\label{eq-thanks3}
\mathfrak{P}_{31}=\Pbm\left(\sqrt{2T_{\wt r_{l}}^{0, t }} \le\alpha(l)+\frac{j}{2}\,\bigg|\,\mathcal{D}_{0,t,l}( j)\right).
\ee
Hence, 
\begin{eqnarray}
  \label{bu1}
  \mathfrak{P}_{31}&=&
  \Pbm\left(\sqrt{2T_{\wt r_{l}}^{0, t }} \le\alpha(l)+\frac{j}{2}\,\bigg|\, \sqrt{2T_{l-2}^{0,t }}\in I_{\alpha(l)+j}\right)\nn\\
&=& \Pbm\left( 
\sqrt{2T_{\wt r_{l}}^{0, r_{l-2},T_{l-2}^{0,t }}}\le\alpha(l)+\frac{j}{2}\,\bigg|\, \sqrt{2T_{l-2}^{0,t}}\in 
I_{\alpha(l)+j}\right),\nn\\
&\le& \sup_{u\in I_{\al(l)+j}}
\Pbm\left( 
\sqrt{2T_{\wt r_{l}}^{0, r_{l-2},u^{2}/2}}\le\alpha(l)+\frac{j}{2} \right),
\end{eqnarray}
where we write $T_{\wt r_{l}}^{0, r_{l-2},m}=\corJ{\trav{m}{0}{r_{l-3}}{r_{l-2}}{0}{\wt r_{l-1}^+}{\wt r_l}}$, compare with \eqref{eq-traver2}.

Set $u=\alpha(l )+j+\ze$, where $0\leq \ze\leq 1$.
It follows from  
\cite[Lemma 4.6]{BK},  after correcting a typo, 
that
\begin{equation}
\Pbm\left( 
\sqrt{2T_{\wt r_{l}}^{0, r_{l-2},u^{2}/2}}\le\alpha(l)+\frac{j}{2} \right)\leq e^{-\left(\sqrt{q} (\alpha(l)+j+ \ze) -\sqrt{p}(\alpha(l)+\frac{j}{2})\right)^{2}/2}\label{bu2a}
\end{equation}
where
\begin{equation}
  \label{eq-q}q:=
\frac{\log r_{l-3}-\log r_{l-2}}{\log r_{l-3}-\log (r_{l}\left(1-\bb \right))}=\frac{1}{3+O\left(\bb \right)}
\end{equation}
and
\begin{equation}
  \label{eq-p}
p:=\frac{\log (r_{l-1}\left(1+\bb \right))-\log (r_{l}\left(1-\bb \right))}{\log r_{l-3}-\log (r_{l}\left(1-\bb \right))}=\frac{1+O\left(\bb 
\right)}{3+O\left(\bb \right)}.
\end{equation}
Indeed,
to apply Lemma 
\ref{lem: GW proc LD} or
\cite[Lemma 4.6]{BK} it suffices to show  that 
\[\alpha(l)+\frac{j}{2}\leq (\alpha(l )+j)\sqrt{q/p}
=(\alpha(l )+j)(1-O\left(\bb \right))=(\alpha(l )+j)-O\left(1\right),\]
since $\bb \left(\alpha(l)+j\right)=1$. Thus we can use (\ref{bu2a}) for all $j\geq c_{3}$ for some $c_{3}<\ff$. For such $j$
we therefore have
\[
\Pbm\left( 
\sqrt{2T_{\wt r_{l}}^{0, r_{l-2},u^{2}/2}}\le\alpha(l)+\frac{j}{2} \right)\le ce^{ -\frac{1}{3}\left(\frac{j}{2}+\ze+O\left(\bb \left(\alpha(l)+\frac{j}{2}\right)\right)\right)^{2}/2},
\]
and using again  $\bb \left(\alpha(l)+j\right)=1$
we obtain  
\be
\mathfrak{P}_{31}\leq c'e^{-c_{4}j^{2}}\label{bu2}
\ee
for all $j\geq c_{3}$. 
By enlarging $c'$ we then have (\ref{bu2}) for all $j$.

We turn to bounding $\mathfrak{P}_{32}$.
  Assign to each $y\in F_{l+m}\cap B_{d}\left(0,\tilde{c}h_{l}\right)$
a unique ``parent'' $\tilde{y}\in F_{l+m-1}\cap B_{d}\left(0,\tilde{c}h_{l}\right)$ such that
$d\left(\tilde{y},y\right)\le r_{l+m}$. In particular, for $m=1$ we set $\tilde{y}=0$,  and set $\tilde{y}=y$ if $y\in F_{l+m-1}$. Let $ q= q( \tilde{y},y  )=d\left(\tilde{y},y\right)/r_{l}$ (not to be  
confused with \eqref{eq-q}) and set
\begin{equation}
  \mathcal{A}_{m}=\Big\{\underset{y\in F_{l+m}\cap B_{d}\left(0,\tilde{c}h_{l}\right)}{\sup}\left|T_{y,\wt r_{l}}^{0,t }-T_{\tilde{y},\wt r_{l}}^{0, t }\right|\le d_{0}j m\left(\alpha(l)+j\right)\sqrt{ q}\Big\},\label{s1.1}
\end{equation}
where $d_{0}$ will be chosen later, but small enough so
that $d_{0}\sum_{m\ge1}me^{-m/2}\le\frac{1}{8}$.
In words, $\mathcal{A}_m$ is the good event in which all neighboring 
(at scale $l+m-1$) excursion counts for balls whose centers are in a fixed
ball at scale $l$ are not too distinct.
We claim that 
\bea
&&
\bigcap_{m=1}^{k-l}\mathcal{A}_{m}\cap \Big\{\sqrt{2T_{\wt r_{l}}^{0, t }}>\alpha(l)+\frac{j}{2}\Big\}\label{s1.2}\\
&&\subseteq \Big\{\sqrt{2T_{y,\wt r_{l}}^{0 ,t }}>\alpha(l),\,\forall y\in F_{k}\cap B_{d}\left(0,\tilde{c}h(r_{l})\right)\Big\}, \nn
\eea
that is, under ${\mathcal A}_m$, having excursion counts (centered at $0$)
larger than $\alpha(l)+j/2$ implies that all center counts for slightly off-center balls are larger than $\alpha(l)$.

Using that 
$ q=d\left(\tilde{y},y\right)/r_{l}\le r_{l+m}/r_{l}=e^{-m}$ for $y\in F_{l+m}$ we see that on the event
in the left hand side of (\ref{s1.2}), for  any
$y\in F_{k}\cap B_{d}\left(0,\tilde{c}h_{l}\right)$
one has
\[
T_{y,\wt r_{l}}^{0, t }\ge\left(\alpha(l)+\frac{j}{2}\right)^{2}/2-j\left(\alpha(l)+j\right)d_{0}\sum_{m\ge1}me^{-m/2},
\]
which, since $d_{0}\sum_{m\ge1}me^{-m/2}\le\frac{1}{8}$,
implies that
\[
T_{y,\wt r_{l}}^{0, t }  \ge 
\left(\alpha(l)+\frac{j}{2}\right)^{2}/2-\frac{1}{4}j\left(\alpha(l)+j\right)
  >  \al^{2}(l)/2.
\]
This establishes (\ref{s1.2}) and taking complements we see that
\begin{equation}
  \mathcal{B}^{\bb ,k}_{t,l}
  \subseteq \bigcup_{m=1}^{k-l}\mathcal{A}^{c}_{m}\cup 
\Big\{\sqrt{2T_{\wt r_{l}}^{0, t }}\leq \alpha(l)+\frac{j}{2}\Big\}.
\label{s1.3}
\end{equation}
It follows that
\begin{eqnarray}
&& \mathcal{B}^{\bb ,k}_{t,l}\cap \Big\{\sqrt{2T_{\wt r_{l}}^{0, t }}>\alpha(l)+\frac{j}{2}\Big\}
\subseteq \bigcup_{m=1}^{k-l}\mathcal{A}^{c}_{m}. \label{gamlab}
\end{eqnarray}

For $y\in F_{l+m}\cap B_{d}\left(0,\tilde{c}h_{l}\right)$,
write 
$$\mathcal{A}_m^{y,c}=
\left\{\left|T_{y,\wt r_{l}}^{0, t }-T_{\tilde{y},\wt r_{l}}^{0, t}\right|\ge d_{0}j m\left(\alpha(l)+j\right)\sqrt{ q}\right\}.$$
We thus obtain
that
\begin{eqnarray}
\mathfrak{P}_{32} &\leq &
\sum_{m=1}^{k-l}\Pbm\left(
\cup_{y\in F_{l+m}\cap B_{d}\left(0,\tilde{c}h_{l}\right)}
\mathcal{A}_m^{y,c}
\,\bigg|\,\mathcal{C}_{0,t,l}\cap\mathcal{D}_{0,t,l}(j)\right)\nonumber\\
&\leq&\sum_{m=1}^{k-l}\left|F_{l+m}\cap B_{d}\left(0,\tilde{c}h_{l}\right)\right|
\label{eq:chaining union bound}\underset{y\in F_{l+m}\cap B_{d}\left(0,\tilde{c}h_{l}\right)}{\sup}
\Pbm\left(\mathcal{A}_m^{y,c}
\,\bigg|\,\mathcal{C}_{0,t,l}\cap\mathcal{D}_{0,t,l}(j)\right)\nn\\
&\le& c\sum_{m=1}^{k-l}e^{2m}\underset{y\in F_{l+m}\cap B_{d}\left(0,\tilde{c}h_{l}\right)}{\sup}
\Pbm\left(\mathcal{A}_m^{y,c}
\,\bigg|\,\mathcal{C}_{0,t,l}\cap\mathcal{D}_{0,t,l}(j)\right).\label{17.1ax}
\eea

 \corJ{Note that  if $m=T_{l-2}^{0,t }$ then  $T_{y,\wt r_{l}}^{0, t }=\trav{t}{0}{r_0}{r_1}{y}{\wt r_{l-1}^+}{\wt r_l}=\trav{m}{0}{r_{l-3}}{r_{l-2}}{y}{\wt r_{l-1}^+}{\wt r_l}$.
We can thus write the last probability as
\begin{eqnarray}
  \label{eq-thanks4}
  &&\\
&&
\!\!\!\!\!
\!\!\!\!\!
\Pbm\left(
\left|
\trav{T_{l-2}^{0,t}}{0}{r_{l-3}}{r_{l-2}}{y}{\wt r_{l-1}^+}{\wt r_l}
-\trav{T_{l-2}^{0,t}}{0}{r_{l-3}}{r_{l-2}}{\tilde y}{\wt r_{l-1}^+}{\wt r_l}
\right|\ge d_{0}j m\left(\alpha(l)+j\right)\sqrt{ q}\,\bigg|\,\mathcal{C}_{0,t,l}\cap\mathcal{D}_{0,t,l}(j)\right)\nn\\
&&
\!\!\!\!\!
\!\!\!\!\!\!
\leq  
\!\!\!
\sup_{u\in I_{\alpha(l)+j}}
\!\!\Pbm\left(
\left|
\trav{u^2/2}{0}{r_{l-3}}{r_{l-2}}{y}{\wt r_{l-1}^+}{\wt r_l}
-\trav{u^2/2}{0}{r_{l-3}}{r_{l-2}}{\tilde y}{\wt r_{l-1}^+}{\wt r_l}
\right|\ge d_{0}j mu
\sqrt{ q}/2\,\bigg|\,\mathcal{C}_{0,t,l}\cap\mathcal{D}_{0,t,l}(j)\right)\!.\nn
\end{eqnarray}}
We now have the following lemma, whose proof, based on the continuity
  estimates derived in Section \ref{sec-continest},
  appears at the end of
that section.

   \bl
  \label{lem:4.58m} There exist constants $c>0$ and $j_{0}<\ff$ such that for all $j\geq j_{0}$, $y\in F_{l+m}\cap B_{d}\left(0,\tilde{c}h_{l}\right)$, 
  $ q=d\left(\tilde{y},y\right)/r_{l} $ and $m\leq k-l$ as above,
\begin{eqnarray}
  \label{eq-thanks4lem}
  &&
  \!\!  \!\!  \!\!
  \!\!  \!\!  \!\!
  \sup_{u\in I_{\alpha(l)+j}}
\Pbm\left(
\left| \corJ{
\trav{u^2/2}{0}{r_{l-3}}{r_{l-2}}{y}{\wt r_{l-1}^+}{\wt r_l}
-\trav{u^2/2}{0}{r_{l-3}}{r_{l-2}}{\tilde y}{\wt r_{l-1}^+}{\wt r_l}}
\right|\ge d_{0}j m
u\sqrt{ q}/2\,\bigg|\,\mathcal{C}_{0,t,l}\cap\mathcal{D}_{0,t,l}(j)\right)\nn\\
&&\quad \quad
\quad \quad 
\leq  ce^{-4(j+ m)} .
\end{eqnarray}
  Here $\tilde y$ is the ``parent" of $y$ defined in the paragraph following (\ref{bu2}).
\el

Using \eqref{eq-thanks4lem} and substituting in \eqref{eq-P3} we see
that  for all $j\geq j_{0}$
\be
\mathfrak{P}_{32}\leq c\sum_{m=1}^{k-l}e^{2m}e^{-4(j+ m)}\le Ce^{-4j},\label{dec1a}
\ee 
provided $C$ is large enough.
Combining \eqref{bu2} and 
\eqref{dec1a} yields \eqref{eq:to show} and completes the proof of 
\eqref{cb.15y} and therefore of Lemma \ref{lem-sec3.3a}.
\end{proof}
 
\begin{proof}[Proof of Lemma \ref{lem-sec3.3}.]
  The proof is almost identical to that of Lemma \ref{lem-sec3.3a}.
  Replace $\mathfrak{P}_3$
by the same quantities with the extra conditioning on $\mathcal{C}_{0,t,l}$
omitted. For the analogue of $\mathfrak{P}_{31}$, because of \eqref{eq-thanks3}
we obtain exactly the same estimate, i.e. \eqref{bu2}. For the analogue of
$\mathfrak{P}_{32}$, we follow the proof up to
\eqref{eq-thanks4}, and note that the application of Lemma \ref{lem-cont2} used in the proof of Lemma \ref{lem:4.58m}
still works, since the conditioning on 
$\mathcal{D}_{0,t,l}(j)$ already specifies
$T^{0,t}_{l-2}$. 
This leads to \eqref{dec1a} and completes the proof of \eqref{17.1sv} and
hence of the lemma.
\end{proof}

\subsection{From excursion counts to cover time}
  \label{sec-fromex-to-ct}
  
This short section is devoted to the proof of 
  \eqref{goal.1asto}.
  \begin{proof}[Proof of \eqref{goal.1asto}.]
By the definitions of $t_z$ and $s_{L}(z)$, see
\eqref{eq:defoftsr} and \eqref{clt.2}, 
the estimate (\ref{goal.1})
that we have already proved 
is equivalent to the existence of a constant  $C$ so that
\begin{equation}
 \limsup_{L\to\ff} \Pbm\left( T_{L}^{y,s_{L}\left(z\right)}=0\mbox{ for some }y\in F_{L}\right)\leq C(1+z)  e^{ -z }.\label{ngoal.1}
\end{equation}
With
\[
  \tilde{\mathcal{C}}_{L}=\sup_{x\in F_{L}}\hit_{\partial B_{d}\left(x,h_{L}\right)},
\]
it follows from Theorem \ref{theo-etrt} that
\begin{equation}
  \limsup_{L\to\ff}\Pbm\left(\tilde{\mathcal{C}}_{L}\ge 4s_{L}\left(z\right)\right)\leq C e^{ -z+d\sqrt{z}}.\label{neq: ub real time}
\end{equation}
On the other hand, 
for some $d_{0}<\ff$
\[
  \tilde{\mathcal{C}}_{[\log\epsilon^{-1}]-d_{0}}\le \mathcal{C}_{\epsilon}^*\le
  \tilde{\mathcal{C}}_{[\log\epsilon^{-1}]+d_{0}},
\]
so that (\ref{neq: ub real time}) implies that
\begin{equation}
  \limsup_{\epsilon\to 0}\Pbm\left(\mathcal{C}_{\epsilon}^*
  \ge 4s_{[\log\epsilon^{-1}]}\left(z\right)\right)\leq Ce^{ -z+d \sqrt{z}},\label{neq: ubc real time}
\end{equation}
for a possibly larger $C$ and $d$. This is equivalent to \eqref{goal.1asto}.
\end{proof}

\section{Lower bound} 
\label{sec-lowerboundex}

In this section we complete the proof of tightness of the cover time by 
proving the following. 
\begin{proposition}
\label{lem-LBcover}
  For any $\delta>0$ there exists $-\infty<z<0$ such that
  \begin{equation}\label{eq: cov lb goal1}
\liminf_{\epsilon \to 0} \Pbm\(
\sqrt{\CC^{\ast}_\ep} -2\sqrt{2}\left(\log \epsilon^{-1}-\frac14
\log\log \epsilon^{-1}\right)
\geq z\) \ge 1-\delta.
\end{equation}

In addition, for any   $B_{d}\(   a,r\)\subseteq \S^2$, the same result holds if $\CC^{\ast}_{\eps,\S^2}$ is replaced by $\CC^{\ast}_{\eps,\S^2, B_{d}\(   a,r\)}$, the 
$\ep$-cover time of $B_{d}\(   a,r\)\subseteq \S^2$ by  Brownian motion on $\S^2$.
\end{proposition}
Indeed,  \eqref{eq: cov lb goal1} together with \eqref{goal.1asto}
yield \eqref{eq-tight}.  Along the way, we will  also obtain the
estimates on the right tail of the cover time contained in 
\eqref{goal.1astow}.

As discussed in the introduction, the main technical step is the control on 
the right tail of the $h_L$-cover local excursion time $t^*_L$, see
\eqref{eq-tsar}, in the form of \eqref{goal.1L}. We will in fact
prove a more quantitative version of the latter, which is an analogue 
of \cite[Lemma 5.3]{BRZ}.
Recall the notation $r_0$ and $r_l$, see \eqref{dr.1}. 
\begin{proposition}\label{prop:lb-main}
There exists a constant $c$ so that
for all $0 < \rzero < 1$ and for all $ z\geq 1$,
\begin{equation}
  \liminf_{L\to\ff} \Pbm\left(\inf_{y\in F_{L}} T_{L}^{y, t_z }=0 \right)\geq 
  \frac{(1+z)e^{-2z}}{{(1+z)e^{-2z}+cr^{2}_{0}}}.\label{goal.lb}
\end{equation}

In addition, for any  $B_{d}\(   a,r\)\subseteq \S^2$, the same result holds if $F_{L}$ is replaced by $F_{L}\cap B_{d}\(   a,r\)$.
\end{proposition}
The proof of Proposition \ref{prop:lb-main} uses a modified 
second moment method and occupies most of this section.
Before giving the proof, we show quickly how all the announced statements 
follow from Proposition \ref{prop:lb-main} and Theorem 
\ref{theo-etrt}.

\begin{proof}[Proof of Proposition \ref{lem-LBcover} (assuming 
  Proposition \ref{prop:lb-main}).]
Note that the inequality 
\eqref{eq: cov lb goal1} is equivalent to the statement that for any $\delta>0$ there exists $-\infty<z<0$ such that
\begin{equation}\label{eq: cov lb goal}
\liminf_{\epsilon \to 0} \Pbm\(
\frac{\CC^{\ast}_\ep}{ 4 \log \epsilon^{-1} } - \( 2 \log \ep^{-1} - \log\log \ep^{-1}\)
\geq z\) \ge 1-\delta.
\end{equation}
For $z=0$ one has that $t_{L,0}=s_{L}(0) $, see
  \eqref{eq:defoftsr} and \eqref{clt.2}.
Set $\rzero$ small enough so that 
\eqref{goal.lb} 
with $z=1$ guarantees
\begin{equation}\label{eq: exc bound}
\liminf_{L\to\ff} \Pbm\left(\inf_{y\in F_{L}} T_{L}^{y, s_{L}(0)}=0 \right)\geq 1 - {\delta}/{2},
\end{equation}
reducing $r_0$ further if necessary to ensure that \eqref{dr.1} holds.
Next define
\begin{equation}\label{eq: L dep on eps}
L = L(\epsilon) = \log\left( \frac{\rzero}{2 \epsilon } \right)
\end{equation}
With this choice of $L$ we have
$$ r_L = 2 \epsilon \mbox{ and } 1.8\,\epsilon \le h_L \le 2 \epsilon,$$
in addition to
$$ \sup_{ y \in \S^2}\inf_{ x \in F_L} d(x,y) \le \frac{h_L}{1000} \le \frac{r_L}{1000} \le \frac{\epsilon}{500}.$$
This implies that for all $y \in \S^2$ the ball $B_d(y,\epsilon)$ is contained in $B_d(x,h_L)$ for some $x \in F_L$.
Therefore on the event that $\inf_{y\in F_{L}} T_{L}^{y,s_{L}(0)}=0$ the cover time $\CC^{\ast}_\epsilon$ has not yet occured
at the time when the first $x \in F_L$ registers $s_{L}(0)$ excursions from $\partial B_d(x,h_0)$ to $\partial B_d(x,h_1)$, i.e. by time $\inf_{x \in F_L} \tau_x(s_{L}(0))$. Thus from \eqref{eq: exc bound} it follows that
\begin{equation}\label{eq: cov exc bound}
\liminf_{\epsilon \to 0} \Pbm\left( \CC^{\ast}_\epsilon \ge \inf_{x \in F_L} \tau_x( s_{L}(0))  \right)\geq 1 - {\delta}/{2},
\end{equation}
where we recall that $L$, and therefore also $\tau_x(s_{L}(0))$, depends on $\epsilon$ through \eqref{eq: L dep on eps}. 

Recall, see  \eqref{clt.2},  that $s(\cdot)$ is monotone increasing, so that $\tau_x(s_{L}(0))\geq \tau_x(s_{L}(-z_0))$  so that
Theorem \ref{theo-etrt} implies that 
\begin{equation}\label{eq: etrt bound}
\liminf_{\epsilon \to 0}\Pbm\left( \inf_{x \in F_L} \tau_x\left(s_{L}(0) \right) \ge 4s_{L}\left(-z_0-d\sqrt{z_0} \right) \right) \ge 1 - {\delta}/{2}.
\end{equation}
Combining  
\eqref{eq: cov exc bound} and \eqref{eq: etrt bound}  proves \eqref{eq: cov lb goal}.

For any  $B_{d}\(   a,r\)\subseteq \S^2$, the same proof shows that if  \eqref{goal.lb} holds
with $F_{L}$ replaced by $F_{L}\cap B_{d}\(   a,r\)$
then
 \eqref{eq: cov lb goal}   holds with $\CC^{\ast}_{\eps,\S^2}$  replaced by $\CC^{\ast}_{\eps,\S^2, B_{d}\(   a,r\)}$.
\end{proof}
\begin{proof}[Proof of \eqref{goal.1astow} (assuming
  Proposition \ref{prop:lb-main}).] 
  We use Proposition \ref{prop:lb-main} with $r_0 $ sufficiently small to obtain that
$$\liminf_{L\to\ff} \Pbm\left(\inf_{y\in F_{L}} T_{L}^{y, t_z }=0 \right) \geq c (1+z)  e^{ - 2 z } \mbox{ for all } z>0,$$
where we bounded above the denominator of the right-hand side 
of \eqref{goal.lb} by a constant. Set 
$L = L(\epsilon) = \log\left( \rzero/(2 \epsilon)  \right)$, as in
\eqref{eq: L dep on eps}, and argue as in
the proof of Proposition \ref{lem-LBcover} to obtain that
\begin{equation}\label{eq: cov exc bound right tail lb1}
\liminf_{\epsilon \to 0} \Pbm\left( \CC^{\ast}_\epsilon \ge \inf_{x \in F_L} \tau_x( t_z )  \right) \geq c  e^{ - 2 z }\mbox{ for all }z>0.
\end{equation}
Now $t_z = s_{L}(2 z + o(1) )$ for $z>0$ and large enough $L$, and by Theorem \ref{theo-etrt}
\begin{equation}\label{eq: etrt bound right tail lb2}
\liminf_{\epsilon \to 0}\Pbm\left( \inf_{x \in F_L} \tau_x\left( s_{L}(2 z) \right) \ge 4s_{L}\left( 2z - d\sqrt{z} \right) \right) \ge 1 - c' e^{-4z}.
\end{equation}
Also
$$ 4s_{L}(2 z-d\sqrt{z} ) = 4 \log \epsilon^{-1} \left( 2 \log \ep^{-1} - \log\log \ep^{-1} + 2 z -d\sqrt{z}+o(1) \right).$$
Together with 
\eqref{eq: cov exc bound right tail lb1} and
\eqref{eq: etrt bound right tail lb2}, we obtain
\begin{eqnarray}\label{eq: etrt bound right tail lb3}
&&\liminf_{\epsilon \to 0} \Pbm\left( \CC^{\ast}_\epsilon \ge 4 \log \epsilon^{-1} \left( 2 \log \ep^{-1} - \log\log \ep^{-1} + 2 z- d\sqrt{z}  \right) \right)\\
&&\quad \quad \quad \quad \geq c   e^{ -2z   } - c'e^{-4z}\geq  c   e^{ -2z   }\mbox{ for all }z\geq z_{0},
  \nonumber
\end{eqnarray}
which in turns implies that
\begin{equation}\label{eq: etrt bound right tail lb4}
\liminf_{\epsilon \to 0} \Pbm\left( \frac{\CC^{\ast}_\epsilon}{4 \log \epsilon^{-1}} - \left(2 \log \ep^{-1} - \log\log \ep^{-1} \right) \ge z \right) \geq c'e^{-z - c\sqrt{z} } \mbox{ for }z \ge z'_{0}.
\end{equation}
Due to the monotonicity in $z$ of the left side of
\eqref{eq: etrt bound right tail lb4}, 
this implies
\eqref{goal.1astow} for $z\geq 0$ (possibly reducing $c'$ as needed).
\end{proof}	

We turn to the main object of this section, which is the proof of
Proposition \ref{prop:lb-main}. We need to set up a second moment method, 
which means to attach to each $y\in F_L$ an event $\mathcal{I}_{y,z}$ 
so that, with $J_z=\sum_{y\in F_L} {\bf 1}_{\mathcal{I}_{y,z}}$,
the following properties hold:
\begin{eqnarray}
  \label{eq-desirada}
  &&\mathcal{I}_{y,z}\subset \left\{T_L^{y,t_z}=0\right\}.\\
  &&\frac{\left(\Ebm J_z\right)^2}{\Ebm J_z^2}
    \geq \frac{(1+z)e^{-2z}}{(1+z)e^{-2  z}+cr_0^2}.
    \label{eq-desirada2}
  \end{eqnarray}
  Indeed, \eqref{eq-desirada} and \eqref{eq-desirada2} together imply
  Proposition \ref{prop:lb-main} by an application of Cauchy-Schwarz.

  The major difficulty in constructing such events as $\mathcal{I}_{y,z}$
  is that  the computation of $\Ebm J_z^2$ involves probabilities of
  the form $\Pbm(\mathcal{I}_{y,z}\cap \mathcal{I}_{y',z})$ for
  $y\neq y'$, and one would like to have the events
  in the last probability decouple as much as possible. 
  A standard method is to have $\mathcal{I}_{y,z}$ include
  a barrier event, similar to but different from the one 
  used in the upper bound. Indeed, this was the approach of \cite{BK}, and
  also the approach taken by us in \cite{BRZ}. However, due to the difficulties
  in decoupling excursions around non-concentric centers, this turns out not
  to be enough to obtain the degree of precision in \eqref{eq-desirada2}.
  Our approach is to add to the barrier event information on
  the start and end points of excursions, in relatively large scales 
  (i.e., small $l$).

  We now begin with the construction, which culminates with
  \eqref{eq:LBTruncatedSum} below. 
  Recall the notation $\rho_L$ and $l_L$, see \eqref{eq:defofts}
  and \eqref{eq:AlphaBarrierDefd},
and set
\begin{equation}
\gamma\left(l\right)=\gamma\left(l,L\right)= \rho_{L}(L-l)+  l^{ 1/4}_{L}.\label{14.1}
\end{equation}
%and 
%\begin{equation}
%\delta_{z}\left(l\right)=\delta\left(l,L,z\right)= \rho_{L,z}(L-l) +6 l^{ 3/4}_{L},\label{14.3}
%\end{equation}
We introduce the events $\Iyz$, beginning with a barrier event. 
Set 
\begin{equation}
\Iyzh=\left\{ \gamma\left(l\right)\le\sqrt{2T_{l}^{y,t_{z}}} \mbox{ for }l=1,\ldots,L-1\mbox{ and }  T_{L}^{y,t_{z}}=0\right\}.\label{eq:TruncatedSummandLBz}
\end{equation}
As discussed above, we need to augment $\Iyzh$ by information on 
the angular increments of the excursions.
Instead of keeping track of individual 
excursions, we track the empirical measure of the increments, by comparing it
in Wasserstein distance to a reference measure. Recall that the Wasserstein
$L^1$-distance between probability measures on $\R$ is given by
\begin{equation}
  \dwa(\mu,\nu)=\inf _{\xi\in \PP^2(   \mu,\nu)}\Big \{   \int |x-y|\,d\xi(   x,y)\Big \},\label{was0}
\end{equation}
where $\PP^2(\mu,\nu)  $ denotes the set of probability measures on $\R\times\R$
with marginals $\mu,\nu$. If $\mu$
is a probability measure on $\R$ with finite support and if 
$\theta_{i},\, 1\leq i\leq n$ denote a sequence of i.i.d 
$\mu$-distributed random variables 
then it follows from \cite[Theorem 2]{FG}  that for some $c_{0}=c_0(\mu)$
\begin{equation}
  \mbox{\rm Prob}\left\{   
  \dwa\(   \frac{1}{n} \sum_{i=1}^{ n} \de_{\theta_{i}},\mu  \)> \frac{c_{0}x}{\sqrt{n}}     \right\}\leq 2 e^{ -x^{ 2}}.\label{was1}
\end{equation}

Let $\BM_{t}$ be Brownian motion in the plane.
For each $k$ let  $\nu_{k}$ be the probability measure on $[0,2\pi]$ defined by 
 \be
\nu_{k}(   dx)=P^{   (0,r_{k})}\(\mbox{arg }\BM_{H_{\partial B(  0, r_{k-1})}}\in \,dx\),\label{meas45}
\ee  
where $\mbox{arg }x$ for $x\in \R^{2}$ is the argument of $x$ measured from the positive $x$-axis and $P^w$ is the law of $\BM_\cdot$ started from $w$.

Returning to $X_{t}$, our Brownian motion on the sphere, and using isothermal coordinates, see Section \ref{sec-isot},
let   $0\leq \th_{k,i}\leq 2\pi$, $i=1,2,\ldots$ be  the angular increments centered at $y$,   mod $2\pi$,  from  $X_{H^{ i}_{\partial B_{d}(   y, h_{k})}}$ 
to $X_{H^{ i}_{\partial B_{d}(   y, h_{k-1})}}$,   the endpoints of the
$i$'th  excursion between $\partial B_{d}(   y, h_{k})$ and $\partial B_{d}(   y, h_{k-1})$. By the Markov property the $  \th_{k,i} $, $i=1,2,\ldots$ are independent,  and using  Section \ref{sec-isot} we see that   each $  \th_{k,i} $ has distribution $\nu_{k}$. We set, for $n$ a positive integer,
\begin{equation}
\WW_{y,k}(n)=\left\{   \dwa
  \(   \frac{1}{n} \sum_{i=1}^{ n} \de_{\th_{k,i}},\nu_{k}  \) \leq  \frac{c_{0}\log (  k)}{2\sqrt{n}}     \right\}.\label{eq:LBTruncatedSumz}
\end{equation}

We are ready to define the good events $\Iyz$. For $a\in \Z_{+}$ let 
\begin{equation}
N_{k,a}=[(   \rho_{L}(L-k)+a+1)^{ 2}/2].\label{nkdef.2}
\end{equation}
We set
\begin{equation}
N_{k}=N_{k,a}\,\, \mbox{  if  }\,\,\sqrt{2T_{k}^{y,t_{z} }}\in I_{\rho_{L}(L-k)+a}.\label{nkdef.1}
\end{equation}
With $L_{-}= 16 ( \log L )^{ 4}$ and $d^*$ a constant to be determined below,
let 
\begin{equation}
\Iyz=\Iyzh\cap_{k=d^{\ast}}^{L_{-}} \WW_{y,k}\(N_{k}\),
\label{eq:TruncatedSummandLBzz}
\end{equation}
and define the count
\begin{equation}
  J_z=\sum_{y\in F_{L}}{\bf 1}_{\Iyz}.\label{eq:LBTruncatedSum}
\end{equation}
To obtain \eqref{eq-desirada2}, we need a control on the
first and second moments of $J_z$, 
which is provided by the next two lemmas. 
Most of this section is devoted to their proof.
We emphasize that in the statements of the lemmas, the implied constants are
uniform in $r_0$ smaller than a fixed small threshold.
\bl [First moment estimate]
\label{prop:LowerBoundOneProfilez} There is a large enough $d^{\ast}$, such that for all $L$ sufficiently large, all   $1\leq z\leq  \(\log L\)^{1/4}$,
and all $y\in F_{L}$,
\begin{equation}
\Pbm  \left(\Iyz \right)\asymp (1+z)e^{-2L} e^{ -2z-z^{2}/4L}.\label{eq:LowerBoundOneProfilez}
\end{equation}
\el

Let 
\begin{eqnarray}
  \label{eq-G0L} &&\\
  G_{0} & = & \left\{ \left(y,y'\right):y,y'\in F_{L}\mbox{ s.t. }d\left(y,y'\right)>2h_{0}\right\} ,\nn\\
G_{k} & = & \left\{ \left(y,y'\right):y,y'\in F_{L}\mbox{ s.t. }2h_{k}<d\left(y,y'\right)\le2h_{k-1}\right\} \mbox{ for }1\le k<L,\nn\\
G_{L} & = & \left\{ \left(y,y'\right):y,y'\in F_{L}\mbox{ s.t. }0<d\left(y,y'\right)\le2h_{L-1}\right\} \nn .
\end{eqnarray}
 
\bl     [Second moment estimate]
\label{prop:UpperBoundz}  There are large enough $d^*,c'$,
such that for all $L $ sufficiently  large,   
all $1\leq    z\leq \(\log L\)^{1/4}$ 
%$1 \leq k\leq L-1$
and  all $(y, y')\in G_k $, $1\leq k\leq L$,
%with $2h_{k}<d(   y,y')\leq 2h_{k-1}$, or $0<d(   y,y')\leq 2h_{L-1}$,
\begin{equation}
\Pbm\(\Iyz \cap \Iyzp\)\leq  c'  (1+z)e^{-4L+2k}e^{ -2z  }e^{-ck_{L}^{1/4}}.\label{eq:UpperBoundz}
\end{equation}
\el 
Before providing the proofs of Lemmas \ref{prop:LowerBoundOneProfilez} and \ref{prop:UpperBoundz}, we show how Proposition \ref{prop:lb-main} follows from them.
\begin{proof}[Proof of Proposition \ref{prop:lb-main} (assuming Lemmas \ref{prop:LowerBoundOneProfilez} and \ref{prop:UpperBoundz}.)]
It follows immediately from (\ref{dr.2})  and  Lemma  \ref{prop:LowerBoundOneProfilez}   that
\begin{equation}
\Ebm(J_z)\asymp r^{-2}_{0}(1+z) e^{ -2z}.\label{ez.1}
\end{equation}
Since $J_z$ is a sum of indicators and hence at least $1$ if not $0$, an application of 
the Cauchy-Schwarz inequality to $J_z{\bf 1}_{J_z\geq 1}$ gives that
\begin{equation}
\(\Ebm(J_z)\)^{2}\leq \Pbm\(\inf_{y\in F_{L}} T_{L}^{y,t_{z}}=0\) \Ebm(J_z^{2}).\label{glb.3}
\end{equation}
Hence (\ref{goal.lb}) will follow from (\ref{ez.1}) and \eqref{glb.3}
once we show that 
\begin{equation}
\Ebm(J_z^{2})\leq \(\Ebm(J_z)\)^{2}+c  \Ebm(J_z).\label{glb.4}
\end{equation}
Since 
$\Ebm(J_z^{2})= \Ebm(J_z)+\sum_{y\neq y'\in F_{L}}
\Pbm\left(\Iyz \cap \Iyzp\right),$
 it suffices to show that
\begin{equation}
\sum_{y\neq y'\in F_{L}}\Pbm\left(\Iyz\cap \Iyzp\right) 
\leq \(\Ebm(J_z)\)^{2}+c  E(J_z).\label{glb.6}
\end{equation}
%Recall the sets $G_k$, see \eqref{eq-G0L}.
Recall the sets $G_k$, see \eqref{eq-G0L}.
We have that $\bigcup_{k=0}^{L}G_{k}= \{ \left( y,y'  \right) \in
  F_{L}\times F_{L}\,:\, y\neq y'\}$
and therefore 
\begin{equation}
\sum_{y\neq y'\in F_{L}}\Pbm\left(\Iyz \cap \Iyzp\right)
=\sum_{( y,y') \in G_{0}}\Pbm\left(\Iyz \cap \Iyzp\right)
 +\sum_{k=1}^{L}\sum_{(y,y') \in G_{k}}\Pbm\left(
 \Iyz \cap \Iyzp\right) .\label{eq:SecondMomentDecomp}
\end{equation}
By the Markov property, see  \cite[Lemma 5.3]{BK} for details,   we have that
\be
\sum_{( y,y') \in G_{0}}\Pbm\left(\Iyz \cap \Iyzp\right)
= \sum_{( y,y') \in G_{0}}
\Pbm\left(\Iyz
\right)\Pbm\left(\Iyzp\right)\leq \left(\Ebm\left(J_z\right)\right)^{2}. 
\label{eq:FirstTermInDecomp}
\ee
To handle $(y,y')\in G_k$, $k=1,\ldots, L$, 
we write
\begin{equation}
\sum_{(y,y') \in G_{k}}
\Pbm\left(\Iyz\cap \Iyzp\right)
\le cr^{- 2}_{0}e^{4L-2k}\sup_{( y,y') \in G_{k}}
\Pbm\left(\Iyz\cap \Iyzp\right),\label{eq:SumWithSizeOfGk}
\end{equation}
where we have used 
that for $1\le k\le L$,
\[\left|G_{k}\right|  \le 
  \left| F_{L} \right|{\displaystyle \sup_{v\in F_{L}}}
  \left| F_{L}\cap B_{d}\left(v,2h_{k-1}\right)\right| \le
  c\left| F_{L}\right|^2r_{k-1}^{2}
=  c\left| F_{L}\right|^{2} r^{ 2}_{0}e^{-2k},
\]
and therefore $|G_k|\leq cr^{ -2}_{0} e^{4L-2k}$.
Using this and \eqref{eq:UpperBoundz},
%(for $k\leq L-1$) and
%\eqref{eq:LowerBoundOneProfilez},
the right hand side of \eqref{eq:SumWithSizeOfGk}   %\corJ{for $k=1,\ldots,L-1$}  
is bounded above by  
$cr^{- 2}_{0}(1+z)e^{ -2z  }e^{-ck_{L}^{1/4}}.$ 
%\corJ{For $k= L$ we obtain the same result by bounding $\Pbm\left(\Iyz\cap \Iyzp\right)$ by $%\Pbm\left(\Iyz \right)$ and using (\ref{eq:LowerBoundOneProfilez}).} 
Summing over $k=1,\ldots,L$,
%and using \eqref{eq:FirstTermInDecomp},
we obtain that 
\[
 \sum_{k=1}^{L}\sum_{(y,y') \in G_{k}}\Pbm\left(
 \Iyz \cap \Iyzp\right)\leq 
 cr^{- 2}_{0}(1+z)e^{ -2z  }\sum_{k=1}^{L}e^{-ck_{L}^{1/4}}\leq c \Ebm
 J_z,\]
 where the second inequality used
\eqref{ez.1}. Combined with \eqref{eq:SecondMomentDecomp} and 
\eqref{eq:FirstTermInDecomp}, we obtain 
\eqref{glb.6} and thus complete the proof of  \eqref{goal.lb}.

For any  $B_{d}\(   a,r\)\subseteq \S^2$, the same proof shows that    \eqref{goal.lb} holds
with $F_{L}$ replaced by $F_{L}\cap B_{d}\(   a,r\)$.
\end{proof}
So, it remains to prove 
Lemmas 
\ref{prop:LowerBoundOneProfilez}.
and \ref{prop:UpperBoundz}.
The proof of Lemma \ref{prop:LowerBoundOneProfilez}
 is relatively straight forward,
given the barrier estimates of \cite{BRZ}, and is provided in subsection
\ref{subsec-UB}. The proof of Lemma \ref{prop:UpperBoundz}
is much more intricate, and is divided to cases according to the 
distance between $y$ and $y'$, see subsections \ref{subsec-bulk}-\ref{subsec-veryearly}.
The last of these sections  uses a decoupling argument which is 
described in detail in subsection \ref{sec-Decoupling}.

Before turning to these proofs, we introduce some notation and
record a simple computation that will be
useful in calculations.
Recall (\ref{nkdef.2})-(\ref{nkdef.1}), 
and for $a\in \Z_{+}$ introduce the level-$k$ event 
\begin{equation}
  \HH_{k,a} =\lc   \sqrt{2T_{k}^{y,t_{z} }}\in I_{\rho_{L}(L-k)+a}\rc.\label{hka.1}
\end{equation}
Note that on $\HH_{k,a} $ we have 
$N_{k}=N_{k,a}$ where 
\begin{equation}
N_{k,a}=[(   \rho_{L}(L-k)+a+1)^{ 2}/2].\label{nkdef.2u}
\end{equation}

The next lemma is the computation alluded to above.

\bl\label{lem-int}
For all $L $ sufficiently  large  and   $1\leq k< L$,
\begin{equation}
e^{-{\gamma^{2}\left( k \right)}/{2(L-k)}}\leq       cL^{{(L-k) }/{L}}e^{-2(L-k)-2k_{L}^{1/4}}.\label{int.30}
\end{equation}
In particular, if $k\geq \log^{4}L$ then
\begin{equation}
e^{-{\gamma^{2}\left( k \right)}/{2(L-k)}}\leq       ce^{-2(L-k) - k_{L}^{1/4}}.\label{int.31}
\end{equation}
Furthermore, for some $d^{\ast}$ sufficiently large, the same bounds hold for\\ $e^{-{ \gamma^{2}\left( k \right)}\(1-{\frac{2}{L-k}}\)/2(L-k)}$ if
$L-k\geq d^{\ast}$.
\el

\begin{proof} 
	Recalling  (\ref{eq:defofts}) and \eqref{14.1} we have
	\bea\gamma^{2}\left( k \right)&=&
	4(L-k)^{2} -2\frac{(L-k)^{2}}{L}\log L+4(L-k) k_{L}^{1/4}\nn\\
	&&\hspace{1 in}+O\left( \frac{(L-k)\log L}{L^{ 3/4}}  \right) + k_{L}^{1/2}.\label{w4.36}
      \eea
      Consequently 
	\begin{equation}
	  \label{eq:TsDivdedByL2}
	  \frac{\gamma^{2}\left( k \right)}{2(L-k)}
	\geq   2(L-k)  -\frac{(L-k) }{L}\log L+2 k_{L}^{1/4}+o_{L}(1)+  {k_{L}^{1/2} /2( L-k)}.
      \end{equation}
	This yields (\ref{int.30}).
	
	To see (\ref{int.31}),
	note that if $ \log^{4}L\leq k\leq L-\log^{4}L$ then $k_{L}^{1/4}\geq \log L$, while if $   k\geq L-\log^{4}L$ then $L^{(L-k)/L}\leq c$.
	
	For the last statement of the lemma, note that by  
	(\ref{w4.36}),  
	\[\frac{\gamma^{2}\left( k \right)}{(L-k)^{2}}=
	\frac{4 k_{L}^{1/4}}{L-k}+\frac{k_{L}^{1/2}}
      { (L-k)^{2}} +O(1).\]
Since $L-k\geq d^{\ast}$, by taking 	$d^{\ast}$ sufficiently large, up to an error which is $O(1)$ this is dominated by the `extra term' ${k_{L}^{1/2}/2
%\textcolor{blue}{2}
(L-k)}$ in (\ref{eq:TsDivdedByL2}).
\end{proof}

\subsection{The first moment estimate}
\label{subsec-UB}
We prove in this subsection Lemma 
\ref{prop:LowerBoundOneProfilez}. 
In the proof we will use barrier estimates from Section \ref{sec:BoundaryCrossing} and the decoupling lemma (Lemma 
\ref{lem-decoup})  from section \ref{sec-Decoupling}. Recall that our goal is to evaluate $\Pbm(\Iyz)$ up to multiplicative 
constants.

\begin{proof}[Proof of Lemma \ref{prop:LowerBoundOneProfilez}. ]
For the upper bound, we note that 
since $\Iyz \subseteq \Iyzh$,
the upper bound in (\ref{eq:LowerBoundOneProfilez}) 
is immediate from the barrier estimate contained in
Lemma \ref{prop:BarrierSecGWProp} of Appendix I. 

For the lower bound we have for all $0\leq z\leq \(\log L\)^{1/4}$
\begin{eqnarray}
\Pbm  \left(\Iyz\right)&\geq& \Pbm \left(\Iyzh\right)-
\sum_{k=d^{\ast}}^{L_{-}}   \Pbm
\left(\Iyzh\bigcap \WW^{c}_{y,k}\(N_{k}\)  \right)
\label{ty1}\\
&\geq& c\,(1+z)e^{-2L} e^{ -2z}-\sum_{k=d^{\ast}}^{L_{-}}  
\Pbm  \left(\Iyzh\bigcap \WW^{c}_{y,k}\(N_{k}\)\right), \nn
\end{eqnarray}
see \eqref{eq:LBTruncatedSumz} and \eqref{nkdef.1} for notation,
and
where  for $\Pbm (\Iyzh )$ we have 
used the barrier estimate contained in
Lemma \ref{prop:BarrierSecGWProp} of Appendix I. We note that
\begin{equation}
  \Pbm  \!\left(\Iyzh\bigcap \WW^{c}_{y,k}(N_{k})\right)
  \!\!
\leq 
  \!\!\!\!\!
\sum_{a\geq [k^{1/4}]} \!\!\! \Pbm\!
  \left(\Iyzh^{k,a}\right)  \mbox{\rm where } \Iyzh^{k,a}:=\Iyzh\bigcap \HH_{k,a}\bigcap
  \WW^{c}_{y,k}(N_{k,a}),
  \label{nkdef.3}
\end{equation}
see (\ref{hka.1}) and \eqref{nkdef.2} for notation and we have used (\ref{eq:TruncatedSummandLBz}) and 
(\ref{14.1}) to restrict the sum to $a\geq [k^{1/4}]$.
We show below that
for all $d^{\ast}\leq k\leq L_{-}$, and all $0\leq z\leq \(\log L\)^{1/4}$,
\begin{equation}
 \sum_{a\geq [k^{1/4}]} \Pbm  \left(\Iyzh^{k,a}  
\right)\leq c(1+z)e^{-2L} e^{ -2z}e^{ -c'(   \log k)^{ 2}}.\label{was1pl}
\end{equation}
Furthermore, it is easily seen using Lemma \ref{lem: GW proc LD} that the sum over $a>L^{3/4}$ is much smaller than the right hand side of (\ref{was1pl})
so it suffices to show that
\begin{equation}
 \sum_{a= [k^{1/4}]}^{L^{3/4}} \Pbm  \left(\Iyzh^{k,a}  
\right)\leq c(1+z)e^{-2L} e^{ -2z}e^{ -c'(   \log k)^{ 2}}=: \mathcal{E}(z,k).\label{was1p}
\end{equation}
Combining \eqref{was1p} with \eqref{ty1} and \eqref{nkdef.3} yields the lower 
bound in (\ref{eq:LowerBoundOneProfilez}) (when $d^*$ is taken sufficiently
large) and completes the proof of Lemma \ref{prop:LowerBoundOneProfilez}.

We turn to the proof of \eqref{was1p}. We introduce the notation
\begin{equation}
\BB_{y,k,L}=\left\{\gamma \left(l\right)\le\sqrt{2T_{l}^{y,t_{z}}} \mbox{ for }l=k,\ldots,L-1;\;  T_{L}^{y,t_{z}} =0\right\},
\label{eq:JDownnew}
\end{equation}
for the barrier condition from $l\geq k$ to $L$. Then, with 
\begin{equation}
  \label{eq-kpa}\KK_{k,p,a}=\HH_{k-3,p}\bigcap 
\HH_{k,a}\bigcap \WW^{c}_{y,k}(N_{k,a})\bigcap 
\BB_{y,k+1,L},
\end{equation}
see \eqref{hka.1} and \eqref{nkdef.2u},
we have 
\be 
   \Pbm  \left(\Iyzh^{k,a}  
\right)
\leq   \sum_{p=[(   k-3)^{1/4}]}^{L^{3/4}}\Pbm 
\left(\KK_{k,p,a}\right)\label{nkdef.10}
\ee
plus a term which is much smaller than the right hand side of (\ref{was1p}).
Hence to prove \eqref{was1p} it suffices to show that for all $d^{\ast}\leq k\leq L_{-}$, and all $0\leq z\leq \(\log L\)^{1/4}$,
\begin{equation}
\sum_{a=[k^{1/4}]}^{L^{3/4}}\sum_{p=[(   k-3)^{1/4}]}^{L^{3/4}}\Pbm 
\left(\KK_{k,p,a}\right)\leq c(1+z)e^{-2L} e^{ -2z}e^{ -c'(   \log k)^{ 2}}.\label{mwas7p}
\end{equation}
 
Let 
\begin{equation}
  \WW^{\in x}_{y,k}(   n)=\left\{  \dwa
   \(\frac{1}{n} \sum_{i=1}^{ n} \de_{\th_{k,i}},\nu_{k}\) 
   \in \frac{c_{0}}{2\sqrt{n}} I_{ x  }     \right\},\label{eq:LBTruncatedSumz4}
\end{equation}
so that
\begin{equation}
\WW^{c}_{y,k}(N_{k,a})\subseteq 
\cup_{m=\log k}^{\ff} \WW^{\in m}_{y,k}( N_{k,a}),\label{compde.1}
\end{equation}
and consequently, setting
\begin{equation}
  \LL_{k,m,p,a}=\KK_{k,p,a}\cap \WW_{y,k}^{\in m}(N_{k,a}),  \label{extra.1}
\end{equation} 
we have
\begin{equation}
\Pbm \left(\KK_{k,p,a}\right)\leq 
\sum_{m=\log k}^{\ff}   \Pbm \left(\LL_{k,m, p,a}\right).\label{wasd.1}
\end{equation}

Thus to prove (\ref{mwas7p}) it suffices to show that   for all $m\geq \log k $,  all $d^{\ast}\leq k\leq L_{-}$, and all $0\leq z\leq \(\log L\)^{1/4}$,
\begin{equation}
\sum_{a=[k^{1/4}]}^{L^{3/4}}\sum_{p=[(   k-3)^{1/4}]}^{L^{3/4}}\Pbm 
\left(  \LL_{k,m,p,a}\right)\leq c(1+z)e^{-2L} e^{ -2z}e^{ -c'\,m^{ 2}}.\label{firstmomentgoal}
\end{equation}

Write
\begin{equation}
  \label{eq-abbr}
  \LL'_{k,m,p,a}=\HH_{k-3,p}\bigcap \HH_{k,a}\bigcap \WW_{y,k}^{\in m}(N_{k,a}).
\end{equation}
 
Since  the $\th_{k,i}$ are  i.i.d. 
$\nu_{k}$-distributed random variables, as explained in the paragraph 
before (\ref{eq:LBTruncatedSumz}), it follows from (\ref{was1})  that
\begin{equation}
\Pbm \left(\WW_{y,k}^{\in m}( N_{k,a})\,\Big |\, T_{k-3}^{y,t_{z}}\right)
\leq 2e^{ -m^{ 2}/4}.\label{was1u}
\end{equation}

\bl\label{lem- kmpa}
There exist constants $c,c'>0$ so that
\begin{equation}
 \Pbm \left(\LL'_{k, m, p,a}\)\leq c\,  e^{ -2k-2(z-p)-(z-p)^{2}/4k}e^{ -m^{ 2}/8}  e^{-c' (a-p)^{ 2}}.\label{17.jjf}
\end{equation}
\el
\begin{proof}
%{\bf  Proof: }
From Lemma \ref{lem: GW proc LD}, as in \eqref{17.1} and
\eqref{17.1g}, we obtain
\begin{equation}
\Pbm  \left(\HH_{k-3,p}\right)=
\Pbm \left(\sqrt{2T_{k-3}^{y,t_{z}}}\in  I_{[\rho_{L}\left(L-k+3\right)]+p}
\right)\leq c\,e^{ -2k-2(z-p)-(z-p)^{2}/4k},\label{used.1}
\end{equation}
and  a slight variation of the same argument shows that for some $c_1,c>0$
\[\Pbm  \left(\HH_{k,a}\,|\,\HH_{k-3,p} \right)=
\Pbm \left( \sqrt{2T_{3}^{y,(   [\rho_{L}\left(L-k+3\right)]+p)^{ 2}/2}}\in  
I_{[\rho_{L}\left(L-k\right)]+a} \right)\leq 
ce^{-c_{1} (a-p)^{ 2}},\]
and hence 
\begin{equation}
\Pbm  \left(\HH_{k-3,p}\cap \HH_{k,a} \right) 
 \leq c\,e^{ -2k-2(z-p)-(z-p)^{2}/4k} e^{-c_{1} (a-p)^{ 2}}. \label{used.3}
\end{equation}

Using    (\ref{was1u}) and  (\ref{used.1}) we obtain that 
\begin{equation}
\Pbm \left(\HH_{k-3,p}\bigcap \WW_{y,k}^{\in m}(N_{k,a})\right)\leq   c\,e^{ -2k-2(z-p)-(z-p)^{2}/4k}e^{ -m^{ 2}/4}.\label{used.4}
\end{equation}
Using the  Cauchy-Schwarz inequality and \eqref{eq-abbr}, we have that
  $\Pbm (\LL'_{k, m, p,a})
\leq \Big (\Pbm ( \HH_{k-3,p}\bigcap  
\WW_{y,k}^{\in m}(N_{k,a})   \Big )^{ 1/2}\Big (\Pbm (\HH_{k-3,p}
\bigcap \HH_{k,a}  )\Big )^{ 1/2}$. The lemma follows by substituting \eqref{used.3} and \eqref{used.4}.
\end{proof}
%From the definition \eqref{eq-abbr} and the 
%Cauchy-Schwarz inequality we obtain that 
%\begin{eqnarray}
% \Pbm \left(\LL'_{k, m, p,a}\)
%&\leq& \Big (\Pbm \left( \HH_{k-3,p}\bigcap  
%\WW_{y,k}^{c}(N_{k,a})   \right)\Big )^{ 1/2}\Big (\Pbm \left(\HH_{k-3,p}
%\bigcap \HH_{k,a}  \right)\Big )^{ 1/2}\nn\\
%&
%\leq& c\,  e^{ -2k-2(z-p)-(z-p)^{2}/4k}e^{ -m^{ 2}/8}  e^{-c (a-p)^{ 2}}.\nn
%\end{eqnarray}
%\qed
%\end{proof}

Let
\begin{equation}
k^{+}=k+\lceil10^{10}  \log L\rceil,\hspace{.2 in}k^{++}=k+2\lceil10^{10}  \log L\rceil,\label{eq:DefOfKPlus}
\end{equation}
 \bea
  \label{167.1}
&&  \Bbb{  B}_{y,k+3,k^{ +} }^{j' }
=\left\{\gamma \left(l\right)\le\sqrt{2T_{l}^{y,t_{z}}} \mbox{ for }
l=k+3,\ldots,k^{ +}, \right.\nn\\
&&\left.\hspace{2 in} \,\sqrt{2T_{k^{ +}}^{y,t_{z}}} \in 
I_{\rho_{L}(L-k^{ +})+j'}\right\},
\eea
and
\bea
&&  \label{167.2}
  \widehat{\mathfrak{B}}_{y,k^{ ++}+1,L}
=\left\{\rho_{L}(L-l) \le\sqrt{2T_{l}^{y,t_{z}}} \mbox{ for }
l=k^{ ++}+1,\ldots,L-1,\right.\nn\\
&&\left.\hspace{3 in} \,\sqrt{2T_{L}^{y,t_{z}}}=0\right\}.
\eea
(Compare \eqref{167.2} to \eqref{eq:JDownnew}. The only difference is that a straight barrier is used.)
We have that 
\begin{eqnarray}
&& \Pbm \left(\LL_{k, m, p,a}\right)=\Pbm \left(
\LL'_{k, m, p,a}
\bigcap \BB_{y,k+1,L} \right)\leq \sum _{j' = \left(k^{ +}\right)^{1/4}}^{L^{3/4}} \sum _{j'' = \left(k^{ ++}\right)^{1/4}}^{L^{3/4}} 
\nn\\
&& \hspace{.1 in}
\Pbm\(\LL'_{k, m, p,a}\cap    \Bbb{  B}_{y,k+3,k^{ +} }^{j' }\cap 
  \widehat{\mathfrak{  B}}_{y,k^{+ +}+1,L};\,\sqrt{2T_{k^{ ++}}^{y,t_{z}}} \in 
I_{\rho_{L} (L-k^{ ++})+j''}\)\nn\\
&& \hspace{3in}+o(\mathcal{E}(z,k)), \label{161.2}
\end{eqnarray}
see \eqref{was1p} for the definition of $\mathcal{E}(z,k)$,
and the error is due to the restriction
%where, once again, up to an error which is much smaller than the right hand side of (\ref{eq:UpperBoundz2}),  we may assume that 
$j', j''\leq L^{3/4}$.
Let 
	\begin{align}
	\label{g-sigma}
	&{\mathcal G}_{k^{ +}}^y = \\
	&\nn
	%denote
%the 
\mbox{\rm $\sigma$-algebra generated by the excursions
from $\partial B_{d}(y,h_{k^{ +}-1})$ to $\partial B_{d}(y,h_{k^{ +}})$.}
\end{align}
 (In the definition of ${\mathcal G}_{k^{ +}}^y$,
if we start outside $\partial B_{d}(y,h_{k^{ +}-1})$ we include the initial excursion to $\partial B_{d}(y,h_{k^{ +}-1})$. Do not confuse with \eqref{defG}.)
 Note that
\begin{equation}
\mathcal{A}_{j'}:= \LL'_{k, m, p,a} \cap    \Bbb{  B}_{y,k+3,k^{ +} }^{j' }\in  {\mathcal G}_{k^{ +}}^y. \label{167.2g}
\end{equation}
 Using (\ref{167.2g})   we have 
\begin{eqnarray}
&& \Pbm\( \mathcal{A}_{j'};\,\sqrt{2T_{k^{ ++}}^{y,t_{z}}} \in 
I_{\rho_{L} (L-k^{ ++})+j''};\,   \widehat{\mathfrak{  B}}_{y,k^{+ +}+1,L} \)
\nn\\
&& \leq \sup_{s\in I_{\rho (L-k^{ +})+j'},\,\, v\in I_{\rho (L-k^{ ++})+j''}}
\Pbm\( \mathcal{A}_{j'};\,\sqrt{2T_{k^{ ++}}^{y,t_{z}}} \in 
I_{\rho (L-k^{ ++})+j''};\, \right.\nn\\
&&\left.\hspace {.8 in}\Pbm\(  \widehat{\mathfrak{  B}}_{y,k^{+ +}+1,L}\,
| \, \trav{s^{2}/2}{y}{k^{+}-1}{k^{+}}{y}{k^{+ +}-1}{k^{+ +}} =v^{2}/2, \, {\mathcal G}_{k^{+}}^y\) \).  \label{161.2m}
\end{eqnarray}

By  Lemma \ref{recursionends} with the $k,k'$ there replaced by 
$k^{ +}, k^{+ +}$ and using that
$(1+10(k^{++}-k^{ +})h_{k^{++}}/h_{k^{ +}})^{4L^2}$ is bounded above uniformly,
we have that for some universal $c<\ff$
\begin{eqnarray}
\label{old4.65}
&&\Pbm\(  \widehat{\mathfrak{  B}}_{y,k^{+ +}+1,L}\,
| \, \trav{{s^{2}}/{2}}{y}{k^{+}-1}{k^{+}}{y}{k^{+ +}-1}{k^{+ +}} =v^{2}/2, \, {\mathcal G}_{k^{+}}^y\)  
\\
&& \leq c\Pbm\(  \widehat{\mathfrak{  B}}_{y,k^{+ +}+1,L}\,
| \, \trav{{s^{2}}/{2}}{y}{k^{+}-1}{k^{+}}{y}{k^{+ +}-1}{k^{+ +}} =v^{2}/2 \).    \nonumber
\end{eqnarray}
%(We used here that \corJ{$(1+10(k_{++}-k^{ +})h_{k_{++}}/h_{k^{ +}})^{4L^2}$ }
%is bounded above uniformly.)

By the barrier estimate  
  of  Lemma \ref{prop:BarrierSecGWPropk} in Appendix I, we see  
that for $j''$   in the range of summation in \eqref{161.2}, uniformly in $s\in I_{\rho (L-k^{ +})+j'}$ and $ v\in I_{\rho (L-k^{ ++})+j''}$,
\be 
  \Pbm\(  \widehat{\mathfrak{  B}}_{y,k^{+ +}+1,L}\,
| \, \trav{{s^{2}}/{2}}{y}{k^{+}-1}{k^{+}}{y}{k^{+ +}-1}{k^{+ +}} =v^{2}/2 \)  \leq c( 1+  j'' )e^{ -2(L-k^{+ +})-2 j'' }. 
\label{161.4}
\ee
 Thus,   uniformly in $[k^{1/4}]\le a \leq L^{3/4}$,  
\begin{eqnarray}
&&  \Pbm \left(\LL_{k, m, p,a}\right)=\Pbm\(\LL'_{k, m, p,a}\cap \BB_{y,k+3,L} \)  \leq c  e^{ -2(L-k^{+ +})  } 
\nn\\
  &&\hspace{.5 in} \times \sum _{j' =  \left(k^{ +}\right)^{1/4}}^{L^{3/4}} \sum _{j'' = \left(k^{ ++}\right)^{1/4}}^{L^{3/4}}   j'' e^{  -2 j'' } \Pbm\(\LL'_{k, m, p,a}\cap \Bbb{  B}_{y,k+3,k^{ +}, k^{ ++} }^{j', j'' }\),\label{161.3f}
\end{eqnarray}
 where 
  \bea
  \label{167.1q}
&&  \Bbb{  B}_{y,k+3,k^{ +}, k^{ ++} }^{j', j'' }
=\left\{\gamma \left(l\right)\le\sqrt{2T_{l}^{y,t_{z}}} \mbox{ for }
l=k+3,\ldots,k^{ +},  k^{ ++},\right.\nn\\
&&\left.\hspace{1 in} \,\sqrt{2T_{k^{ +}}^{y,t_{z}}} \in 
I_{\rho_{L}(L-k^{ +})+j'},\,\sqrt{2T_{k^{ ++}}^{y,t_{z}}} \in 
I_{\rho_{L} (L-k^{ ++})+j''}\right\}.
\eea

We first consider the case that $m\ge (\log L)^{4}$. Using (\ref{161.3f}), dropping 
the event  $\Bbb{  B}_{y,k+3,k^{ +}, k^{ ++} }^{j', j'' }$ 
and then  using (\ref{17.jjf}), we obtain
\begin{eqnarray}
&&\Pbm \left(\LL_{k, m, p,a}\right)\leq c  e^{ -2(L-k^{+ +})  } L^{3/2} \Pbm\(\LL'_{k, m, p,a} \)
\label{161.3fm}\\
&& \leq  c\, e^{ -2L  } L^{3/4}e^{ 2( k^{+ +}-k)  }e^{
-2(z-p)-(z-p)^{2}/4k}e^{ -m^{ 2}/8} .\nn
%\nonumber\\
%&& \leq  c\, e^{ -2L  } L^{3/2}e^{ 2( k^{+ +}-k)  }e^{
%4k}e^{ -m^{ 2}/8}. \nonumber
\end{eqnarray}
Using  that $k\leq L_{-}= 16 ( \log L )^{ 4}$ and our assumption that $0\leq z\leq \(\log L\)^{1/4}$, it   follows that (\ref{firstmomentgoal}) holds for all $m\ge (\log L)^{4}$. 

We consider separately the case where $\log k\leq m<\log\log  L$ and the case where $\log\log  L\leq m< (\log L)^{4}$.
 
To handle the last probability  in \eqref{161.3f} for $\log k\leq m<\log\log  L$, we use the following
estimate, whose proof, given in  sub-section \ref{sec-Decoupling}, uses the decoupling lemma (Lemma 
\ref{lem-decoup}), barrier estimates and
the control on Wasserstein distance contained in (\ref{was1}).

\bl\label{lem-decx} Let  $\log k\leq m<\log\log  L$. For $a, p, j', j'',  k$ in the ranges specified above,
 \be
 \Pbm\(\LL'_{k, m, p,a}\cap \Bbb{  B}_{y,k+3,k^{ +}, k^{ ++} }^{j', j'' }\)\leq G_{k,a,j', j''} \Pbm \left(\LL'_{k, m, p,a}\)+
e^{-\sqrt{L}},\nn
\ee
where
 \begin{eqnarray}
 &&G_{k,a,j',j''}= c  (a+m^{5/2})e^{-2a}e^{-2 (k^{ ++}-k  )     }  \sum_{\stackrel{\{\bar j': |\bar j'-j'|\leq 2M_{0}m\}}{\{\bar j'': |\bar j''-j''|\leq 2M_{0}m\}}}
 \label{was22pfa}\\
 &&   \hspace{.5 in} \bar j'\,\frac{e^{  -(\bar j'-a)^{ 2}/(4( k^{ +}-k )   )}}{(k^{ +}-k   ) }     e^{2 \bar j'' } \, \frac{e^{-(\bar j'-\bar j'')^{2}/2(k^{ +}-k  )}}{(k^{ +}-k   ) }.\nonumber
 \end{eqnarray}
\el

Combining  Lemma \ref{lem-decx}
with  \eqref{161.3f} and  (\ref{17.jjf}) we see that
\bea 
&& \hspace{.5 in}\Pbm \left(\LL_{k, m, p,a}\right) \leq  ce^{ -2L -2 z }e^{ -m^{ 2}/8}e^{4M_{0}m}\label{clearlower.1m}\\
&&\times \sum _{j' = \left(k^{ +}\right)^{1/4}}^{L^{3/4}} \sum _{j'' =  \left(k^{ ++}\right)^{1/4}}^{L^{3/4}}   e^{ -(z-p)^{2}/4k}(a+m^{5/2})e^{ -c'( p-a)^{ 2} }\nn\\
&&
\sum_{\bar j': |\bar j'-j'|\leq 2M_{0}m } \bar j'\,\frac{e^{  -(\bar j'-a)^{ 2}/(4( k^{ +}-k )   )}}{(k^{ +}-k   ) }
\sum_{\bar j'': |\bar j''-j''|\leq 2M_{0}m}  \bar j''   \, \frac{e^{-(\bar j'-\bar j'')^{2}/2(k^{ +}-k  )}}{(k^{ +}-k   ) }.  \nn 
\eea
which leads to
\bea 
&&\hspace{.3 in}\Pbm \left(\LL_{k, m, p,a}\right)\leq  ce^{ -2L -2 z }e^{ -m^{ 2}/16}   e^{ -(z-p)^{2}/4k}ae^{ -c'( p-a)^{ 2} }\label{clearlower.1ms}\\
&&
\times \sum_{j' \geq  \left(k^{ +}\right)^{1/4}/2}  \frac{j'}{(k^{ +}-k   ) }  e^{-(j'-a)^{ 2}/(4( k^{ +}-k )  )}
\sum_{j'' }  \frac{j''}{(k^{ +}-k   ) }  e^{-(j''-j')^{ 2}/( 2(k^{ +}-k ) )} \nn
\\
&&\leq ce^{ -2L -2 z }e^{ -m^{ 2}/16}  e^{ -(z-p)^{2}/4k}ae^{ -c'( p-a)^{ 2} }\nn\\
%\label{clearlower.1}\\
&&
\hspace{1in}\times \sum_{j'\geq   \left(k^{ +}\right)^{1/4}/2}  \frac{\(j'^{2}+j'\sqrt{k^{ +}-k}\)}{(k^{ +}-k   )^{3/2} }  e^{-(j'-a)^{ 2}/(4( k^{ +}-k )  )}.\nn
%=:\sum_{j'\geq \gamma \left(k^{ +}\right)/2}\RR_{k, p,a,j'}
\eea
We obtain \eqref{firstmomentgoal} for $\log k\leq m\leq \log \log L$
by an elementary, if tedious, calculation.

To handle the case when  $ \log \log L\leq m\leq (\log L)^{4}$, we use, instead of Lemma \ref{lem-decx}, the following a priori weaker
estimate, whose proof, also given in  sub-section \ref{sec-Decoupling}, is a variation of the proof of Lemma \ref{lem-decx}.

\bl\label{lem-decxb} Let  $ \log k\leq m\leq (\log L)^{4}$. For $a, p, j', j'',  k$ in the ranges specified above,
 \be
 \Pbm\(\LL'_{k, m, p,a}\cap \Bbb{  B}_{y,k+3,k^{ +}, k^{ ++} }^{j', j'' }\)\leq H_{k,a,j', j''} \Pbm \left(\LL'_{k, m, p,a}\)+
e^{ -\sqrt{L}},\nn
\ee
where
 \begin{eqnarray}
 &&H_{k,a,j',j''}= c   e^{-2a}e^{-2 (k^{ ++}-k  )     }  \sum_{\stackrel{\{\bar j': |\bar j'-j'|\leq 2M_{0}m\}}{\{\bar j'': |\bar j''-j''|\leq 2M_{0}m\}}}
 \label{was22pf}\\
 &&   \hspace{.5 in} \bar j'\,\frac{e^{  -(\bar j'-a)^{ 2}/(4( k^{ +}-k )   )}}{(k^{ +}-k   )^{1/2}  }     e^{2 \bar j'' } \, \frac{e^{-(\bar j'-\bar j'')^{2}/2(k^{ +}-k  )}}{(k^{ +}-k   )^{1/2} }.\nonumber
 \end{eqnarray}
\el
 
 Combining the last Lemma with  \eqref{161.3f} and  (\ref{17.jjf}) 
%and $\(j^{2}+j\sqrt{k^{ +}-k}\)=:\sum_{j}\RR_{k, p,a,j}$
 we obtain  
\bea 
&&\hspace{.3 in}\Pbm \left(\LL_{k, m, p,a}\right)\label{clearlower.1mb}\\
&&\leq  ce^{ -2L -2 z }e^{ -m^{ 2}/8}e^{4M_{0}m}\sum _{j' =  \left(k^{ +}\right)^{ 1/4}}^{L^{3/4}} \sum _{j'' =  \left(k^{ ++}\right)^{ 1/4}}^{L^{3/4}}   e^{ -(z-p)^{2}/4k} e^{ -c'( p-a)^{ 2} }\nn\\
&&
\times \sum_{\bar j': |\bar j'-j'|\leq 2M_{0}m } \bar j'\,\frac{e^{  -(\bar j'-a)^{ 2}/(4( k^{ +}-k )   )}}{(k^{ +}-k   )^{1/2} }
\sum_{\bar j'': |\bar j''-j''|\leq 2M_{0}m}  \bar j''   \, \frac{e^{-(\bar j'-\bar j'')^{2}/2(k^{ +}-k  )}}{(k^{ +}-k   )^{1/2} }\nn\\
&&\leq
%\eea
%which leads to
%\bea 
%&& \Pbm \left(\LL_{k, m, p,a}\right)\leq  
ce^{ -2L -2 z }e^{ -m^{ 2}/16}   e^{ -(z-p)^{2}/4k} e^{ -c'( p-a)^{ 2} }\nn\\
%\label{clearlower.1msb}\\
&&\times \sum_{j'  }  \frac{j'}{(k^{ +}-k   )^{1/2} }  e^{-(j'-a)^{ 2}/(4( k^{ +}-k )  )}
%\nn\\
%&& \hspace{1 in}
\sum_{j'' }  \frac{j''}{(k^{ +}-k   )^{1/2} }  e^{-(j''-j')^{ 2}/( 2(k^{ +}-k ) )},  \nn \\
&&\leq ce^{ -2L -2 z }e^{ -m^{ 2}/20}(k^{ +}-k   )^{-1}   e^{ -(z-p)^{2}/4k} e^{ -c'( p-a)^{ 2} }
\nn\\
%\label{clearlower.1msc}\\
&&\times \sum_{j'  }  \frac{j'}{(k^{ +}-k   )  }  e^{-(j'-a)^{ 2}/(4( k^{ +}-k )  )}
%\nn\\
%&& \hspace{1 in}
\sum_{j'' }  \frac{j''}{(k^{ +}-k   ) }  e^{-(j''-j')^{ 2}/( 2(k^{ +}-k ) )},  \nn 
\eea
where in the last inequality we used that
for $m\geq \log \log L$, we have $e^{ -m^{ 2}/16}\leq ce^{ -m^{ 2}/20}(k^{ +}-k   )^{-2}$. 
Note however
that we can no longer assert that $j' \geq   \left(k^{ +}\right)^{1/4}/2$.
Another tedious summation now yields  
 %But for $m\geq \log \log L$, we have $e^{ -m^{ 2}/16}\leq ce^{ -m^{ 2}/20}(k^{ +}-k   )^{-2}$ and 
%hence 
%\bea 
%&& \Pbm \left(\LL_{k, m, p,a}\right)\leq  ce^{ -2L -2 z }e^{ -m^{ 2}/20}(k^{ +}-k   )^{-1}   e^{ -(z-p)^{2}/4k} e^{ -c'( p-a)^{ 2} }\label{clearlower.1msc}\\
%&&\times \sum_{j'  }  \frac{j'}{(k^{ +}-k   )  }  e^{-(j-a)^{ 2}/(4( k^{ +}-k )  )}\nn\\
%&& \hspace{1 in}
%\sum_{j'' }  \frac{j''}{(k^{ +}-k   ) }  e^{-(j''-j')^{ 2}/( 2(k^{ +}-k ) )}.  \nn 
%\eea
%We proceed as before except in the case $j'\leq 2a$, after using  (\ref{clearlower.2}) we have 
%\be  
% \Pbm \left(\LL_{k, m, p,a}\right)\leq  ce^{ -2L -2 z }e^{ -m^{ 2}/20}(k^{ +}-k   )^{-1}   e^{ -(z-p)^{2}/4k} a^{3}e^{ -c'( p-a)^{ 2} }.\label{clearlower.1msd} 
%\ee 
%Summing over $a$ leads to the bound 
%\be  
 %\sum_{a}  \Pbm \left(\LL_{k, m, p,a}\right)\leq  ce^{ -2L -2 z }e^{ -m^{ 2}/20}(k^{ +}-k   )^{-1}  p^{3} e^{ -(z-p)^{2}/4k},\label{clearlower.1mse} 
%\ee 
%and then summing over $p$ we have 
%\be  
% \sum_{p} \sum_{a}  \Pbm \left(\LL_{k, m, p,a}\right)\leq  ce^{ -2L -2 z }e^{ -m^{ 2}/20}(k^{ +}-k   )^{-1}  z^{3} k^{4}.\label{clearlower.1msf} 
%\ee 
%Using the fact that $(k^{ +}-k   )^{-1}  z^{3}\leq c$ and $m\geq c\log k$
%completes  the proof of 
the inequality 
(\ref{firstmomentgoal}) for the case $ \log \log L$ $\leq m$ $\leq (\log L)^{4}$. This completes the proof
% and hence 
of Lemma \ref{prop:LowerBoundOneProfilez}. 
\end{proof}

 \subsection{Second moment estimate:   branching in the bulk}
 \label{subsec-bulk}
We begin our proof of the second moment estimate
contained in \eqref{eq:UpperBoundz}. The proof is divided to cases according
to the value of $k$.
In this subsection we prove Lemma \ref{prop:UpperBoundz} for  $2^{4}\log^{4} L\le k<L-2^{4}\log^{4}L$, that is, branching in the bulk. 
In this case, the curved boundary
$\gamma(\cdot)$ will play an important role and 
considerably
simplify the proof compared to the case of very early branching treated below
(where
sophisticated decoupling, and the full barrier,  needs to be used).
Indeed, we will content ourselves with dropping the barrier almost 
entirely, and just bound
the probability of the event
\begin{equation}
\left\{ T_{L}^{y,t_{z}}=0\right\} \cap\left\{T_{L}^{y',k,\gamma^{2}(k)/2}=0\right\} ,\label{eq:LargeKBiggerEvent}
\end{equation}
which  contains the event $\Iyz\cap \Iyzp$.
In this, we follow the argument in \cite{BK}. Specifically,
by the proof of \cite[Corollary 6.7]{BK} %(or by an application
%of Lemma \ref{recursionends}) 
it follows that 
\begin{eqnarray}
&&\Pbm\( T_{L}^{y,t_{z}}=0,\,T_{L}^{y',k,\gamma^{2}(k)/2}=0 \)
\label{lbj.1}\\ 
&&\leq \(1-\frac{ 1}{L-k}\(1-\frac{2}
{L-k}\)\)^{\gamma^{2}\left(k \right)/2}\Pbm\left(T_{L}^{y,t_{z}}=0\right)   \nonumber\\
&&\leq e^{-\frac{ \gamma^{2}\left(k \right)}{2(L-k)}\(1-\frac{2}{
L-k}\)}\Pbm\left(T_{L}^{y,t_{z}}=0\right).   \nonumber
\end{eqnarray}
The contribution in \eqref{lbj.1} 
referring to $y$ is easily bounded by (\ref{eq:NotHitByrL}),
giving
\begin{equation}
\Pbm\left(T_{L}^{y,t_{z}}=0\right)\le ce^{-2L}Le^{ -2z-z^{2}/4L}.\label{eq:HighYsAlone}
\end{equation}
Using Lemma \ref{lem-int} and then the fact that $k^{1/4}_{L}\geq 2 \log L$ we see that 
\begin{eqnarray}
\Pbm(\Iyz\cap \Iyzp)&\leq & c e^{-2L}Le^{ -2z-z^{2}/4L}e^{-2(L-k) -2k^{1/4}_{L}}
\nn\\
&\leq&  c e^{-(4L-2k)- k^{1/4}_{L}}e^{ -2z -z^{2}/4L}.\label{lbj.2}
\end{eqnarray}

\subsection{Second moment estimate: late  branching }
In this subsection we prove Lemma    \ref{prop:UpperBoundz} for  $L-16\log^{4}L\le k<L-1$.
Consider first  the case \[L-16\log^{4}L\le k<L-d^{\ast}.\]
The proof is very similar to the bulk branching case, except that we no longer have $k^{1/4}_{L}\geq 2\log L$ so we will need a barrier bound to control the factor  $L$ on the right hand side (\ref{eq:HighYsAlone}).  Thus we set
\begin{equation}
\BB^{\uparrow}_{y}=\left\{ \rho_{L}\left(L-l \right)\le\sqrt{2T_{l}^{y,t_{z}}}\mbox{ for }l=1,\ldots,k-3\right\}.\label{lbj.3}
\end{equation}
Once again, using the proof of \cite[Corollary 6.7]{BK} and the calculations in (\ref{lbj.1}) it follows that 
\begin{eqnarray}
&&\Pbm\( \BB^{\uparrow}_{y}\cap \left\{T_{L}^{y,t_{z}}=0\right\},\,T_{L}^{y',k,\gamma^{2}(k)/2}=0 \)
\label{lbj.4}\\
&&\leq \(1- \frac{1}{L-k}\(1-{\frac{2}{L-k}}\)\)^{\gamma^{2}(k)/2}\Pbm\left(\BB^{\uparrow}_{y}\cap \left\{T_{L}^{y,t_{z}}=0\right\}\right)   \nonumber\\
&&\leq e^{-{\frac{ \gamma^{2}\left(k \right)}{2(L-k)}}\(1-{\frac{2}{L-k}}\)}\Pbm\left(\BB^{\uparrow}_{y}\cap \left\{T_{L}^{y,t_{z}}=0\right\}\right).   \nonumber\\
&&\leq c   e^{-2(L-k)-2k^{1/4}_{L}}  \Pbm\left(\BB^{\uparrow}_{y}\cap \left\{T_{L}^{y,t_{z}}=0\right\}\right),   \nonumber
\end{eqnarray}
where the last line follows from  Lemma \ref{lem-int}.

The contribution 
referring to $y$ is then bounded by   the barrier estimate contained in
Lemma \ref{prop:BarrierSecGWPropLk} of Appendix I,
\begin{equation}
\Pbm\left(\BB^{\uparrow}_{y}\cap \left\{T_{L}^{y,t_{z}}=0\right\}\right)\le c(1+z)(L-k)^{1/2}e^{-2L}e^{ -2z -z^{2}/4L}.\label{eq:HighYsAlonem}
\end{equation}
Hence
\begin{eqnarray}
\Pbm(\Iyz\cap\Iyzp)&\leq & c(1+z)(L-k)^{1/2}e^{-2L}e^{ -2z-z^{2}/4L}e^{-2(L-k)- 2 k^{1/4}_{L}}
\nn\\
&\leq&  c (1+z)e^{-(4L-2k)-k^{1/4}_{L}}e^{ -2z-z^{2}/4L},\label{lbj.5}
\end{eqnarray}
where we have used part of the exponential in $k^{1/4}_{L}=(L-k)^{1/4}$ to control the factor $(L-k)$.

For $L-d^{\ast}\le k<L-1$ we simply bound the term $\Pbm\left(\Iyz\cap 
\Iyzp\right) $ by $\Pbm\left(\Iyz\right)$
and obtain from (\ref{eq:LowerBoundOneProfilez}) the following upper
bound
\begin{equation}
\Pbm\left(\Iyz\cap\Iyzp\right)\leq c(1+z)e^{-2L} e^{ -2z-z^{2}/4L}\leq c_{d^{\ast}} (1+z)e^{-(4L-2k)-k^{1/4}_{L}}e^{ -2z-z^{2}/4L}.\label{eq:LowK}
\end{equation}

\subsection{Second moment estimate: early  branching }
 \label{subsec-veryearly}
In this subsection we prove Lemma    \ref{prop:UpperBoundz}  for  $1\le k\leq L_{-}=16\log^{4}L$.  It is here that the distinction between $\Iyz$ and $\Iyzh$,
see  \eqref{eq:TruncatedSummandLBz}, plays an important role.
This difference will be controlled by
(\ref{was1}). We remark that for $k$ not too small we can avoid this by  applying  the methods of \cite[Proposition 6.16]{BK}, but for  $k=o\( \(\log\log L\)^{4}\)$, that approach fails.

Consider first the case where  $d^{\ast}\le k\leq L_{-}$. Recall $\BB_{y,k,L}$ 
from  (\ref{eq:JDownnew}).
Then  
\be
\Pbm( \Iyz\cap \Iyzp)\leq
\sum_{a\geq [k^{1/4}]} \Pbm( \WW_{y,k}(N_{k,a } )\cap 
\BB_{y,k+3,L} \cap \HH_{k,a}   \cap \Iyzp).\label{61.0jed}
\ee
Hence it suffices to show that if $d(y,y')\geq 2h_k$ and $k\geq d^*$ then 
\be 
\Pbm( \WW_{y,k}(N_{k,a} )\cap \BB_{y,k+3,L}\cap \HH_{k,a}   \cap 
\Iyzp)\leq  ce^{-a}  E(   k),\label{eq:UpperBoundz2}
\ee 
where we abbreviate
\begin{equation}
E(   k)=e^{-4L+2k}ze^{ -2z  }.\label{eofk}
\end{equation}
By replacing $\Iyzp$ by the larger set $\{T_{L}^{y',k^{+},\gamma^{2}(k^{+})/2}=0\}$ it is easy to see that the probability  in (\ref{61.0jed}) for $a>L^{3/4}$ is much smaller than the right hand side of (\ref{eq:UpperBoundz2}), so we may assume that $a\leq L^{3/4}$.

%Recall $\BB_{y,k+3, k^{ +}, b}$, $ \wh \BB_{y,k+3, k^{ +}, b}$, 
%$ \wt \BB_{y,k^{+},b,k^{++},L}$ from (\ref{eq:JDownzout}),  (\ref{eq:JDownz}) and (\ref   {eq:JDownzi}),   and set 
Set 
\begin{equation}
\Iyzhpk=\left\{ \gamma\left(l\right)\le\sqrt{2T_{l}^{y',t_{z}}} \mbox{ for }l=1,\ldots,k-2,  k+1,\ldots  L-1;  T_{L}^{y',t_{z}}=0\right\},\label{eq:TruncatedSummandLBztt}
\end{equation} 
that is, we drop the barrier in $\Iyzhp$ for $l=k-1,k$.

%\textcolor{red}   {Begin new material}
 Recall the events $\Bbb{  B}_{y,k+3,k^{ +} }^{j' }$ and $\widehat{\mathfrak{  B}}_{y,k^{ ++}+1,L}$,
see \eqref{167.1} and \eqref{167.2}, and the $\sigma$-algebra ${\mathcal G}_{k^{ +}}^y$,
see \eqref{g-sigma}.
%Let
% \bea
%  \label{67.1}
%&&  \Bbb{  B}_{y,k+3,k^{ +} }^{j' }
%=\left\{\gamma \left(l\right)\le\sqrt{2T_{l}^{y,t_{z}}} \mbox{ for }
%l=k+3,\ldots,k^{ +}, \right.\nn\\
%&&\left.\hspace{2 in} \,\sqrt{2T_{k^{ +}}^{y,t_{z}}} \in 
%I_{\rho_{L}(L-k^{ +})+j'}\right\},
%\eea
%where we have dropped the barrier from $k^{ +}$ to $k^{ ++}$,
 %\begin{equation}
 % \label{67.1v}
 % \Bbb{C}_{y,k+3,k^{ +}, v}
%=\left\{\gamma \left(l\right)\le\sqrt{2T_{l}^{y,t_{z}}} \mbox{ for }
%l=k+3,\ldots,k^{ +}, \,\sqrt{2T_{k^{ +}}^{y,t_{z}}}=v\right\},
%\end{equation}
%and
%\bea
%&&  \label{67.2}
%  \mathfrak{  B}_{y,k^{ ++}+1,L}
%=\left\{\rho_{L}(L-l) \le\sqrt{2T_{l}^{y,t_{z}}} \mbox{ for }
%l=k^{ ++}+1,\ldots,L,\right.\nn\\
%&&\left.\hspace{3 in} \,\sqrt{2T_{L}^{y,t_{z}}}=0\right\}.
%\eea
%  
%
%
 We have that
\begin{eqnarray}
&& \Pbm\(   \WW_{y,k}(N_{k,a} )\cap \BB_{y,k+3,L}\cap 
\HH_{k,a}   \cap \Iyzp\)\leq \sum _{j' = \left(k^{ +}\right)^{1/4}}^{L^{3/4}} \sum _{j'' =  \left(k^{ ++}\right)^{1/4}}^{L^{3/4}} 
\nn\\
&& \hspace{.5 in}
\Pbm\(   \WW_{y,k}(N_{k,a} )\cap    \Bbb{  B}_{y,k+3,k^{ +} }^{j' }\cap 
  \widehat{\mathfrak{  B}}_{y,k^{+ +}+1,L}\cap \HH_{k,a}   \cap \Iyzhpk,\right.\nn\\
  &&\left. \hspace{1.5 in}\,\sqrt{2T_{k^{ ++}}^{y,t_{z}}} \in 
I_{\rho_{L} (L-k^{ ++})+j''}\)+o(e^{-a}  E(   k)), \label{61.2}
\end{eqnarray}
where, once again, the error term (see \eqref{eq:UpperBoundz2})
is coming from the restriction
%up to an error which is much smaller than the right hand side of (\ref{eq:UpperBoundz2}),  we may assume that 
$j', j''\leq L^{3/4}$.

%
%  Let ${\mathcal G}_{k^{ +}}^y$ denote
%the $\sigma$-algebra generated by the excursions
%from $\partial B_{d}(y,h_{k^{ +}-1})$ to $\partial B_{d}(y,h_{k^{ +}})$ (and if we start outside $\partial B_{d}(y,h_{k+2})$ we include the initial excursion to $\partial B_{d}(y,h_{k^{ +}})$). 
 Next, note that
\begin{equation}
\mathcal{A}_{j'}:= \WW_{y,k}(N_{k,a} ) \bigcap   \HH_{k,a}   \bigcap \Iyzhpk \cap    \Bbb{  B}_{y,k+3,k^{ +} }^{j' }\in  {\mathcal G}_{k^{ +}}^y. \label{67.2g}
\end{equation}
This is the reason we introduced $\Iyzhpk$.
Using (\ref{67.2g})   we have 
\begin{eqnarray}
&& \Pbm\( \mathcal{A}_{j'};\,\sqrt{2T_{k^{ ++}}^{y,t_{z}}} \in 
I_{\rho_{L} (L-k^{ ++})+j''};\,   \widehat{\mathfrak{  B}}_{y,k^{+ +}+1,L} \)
\nn\\
&& \leq \sup_{s\in I_{\rho (L-k^{ +})+j'},\,\, v\in I_{\rho (L-k^{ ++})+j''}}
\Pbm\( \mathcal{A}_{j'};\,\sqrt{2T_{k^{ ++}}^{y,t_{z}}} \in 
I_{\rho (L-k^{ ++})+j''};\, \right.\nn\\
&&\left.\hspace {.8 in}\Pbm\(  \widehat{\mathfrak{  B}}_{y,k^{+ +}+1,L}\,
| \, \trav{s^{2}/2}{y}{k^{+}-1}{k^{+}}{y}{k^{+ +}-1}{k^{+ +}} =v^{2}/2, \, {\mathcal G}_{k^{+}}^y\) \).  \label{61.2m}
\end{eqnarray}
 Recall the estimates \eqref{old4.65} and \eqref{161.4}. We then have that
%
%By  Lemma \ref{recursionends} with the $k,k'$ there replaced by 
%$k^{ +}, k^{+ +}$ we have that for some universal $c<\ff$
%\begin{eqnarray}
%&&\Pbm\(  \frak{  B}_{y,k^{+ +}+1,L}\,
%| \, \trav{\frac{m^{2}}{2}}{y}{k^{+}-1}{k^{+}}{y}{k^{+ +}-1}{k^{+ +}} =v^{2}/2, \, {\mathcal G}_{k^{+}}^y\)  \\
%&& \leq c\Pbm\(  \frak{  B}_{y,k^{+ +}+1,L}\,
%| \, \trav{\frac{m^{2}}{2}}{y}{k^{+}-1}{k^{+}}{y}{k^{+ +}-1}{k^{+ +}} =v^{2}/2 \).    \nonumber
%\end{eqnarray}
%(We used here that \corJ{$(1+10(k_{++}-k^{ +})h_{k_{++}}/h_{k^{ +}})^{4L^2}$ }
%is bounded above uniformly.)
%
%By the barrier estimate  
%  of  Lemma \ref{prop:BarrierSecGWPropk} in Appendix I, we see  
%that for $j''$   in the range of summation in \eqref{61.2}, uniformly in $m\in I_{\rho (L-k^{ +})+j'}$ and $ v\in I_{\rho (L-k^{ ++})+j''}$,
%\begin{eqnarray}
%&&\hspace{.2 in} \Pbm\(  \frak{  B}_{y,k^{+ +}+1,L}\,
%| \, \trav{\frac{m^{2}}{2}}{y}{k^{+}-1}{k^{+}}{y}{k^{+ +}-1}{k^{+ +}} =v^{2}/2 \)  
%\label{61.4} \leq c( 1+  j'' )e^{ -2(L-k^{+ +})-2 j'' }.  
%\end{eqnarray}
 %Thus   
%we see that  
uniformly in $[k^{1/4}]\le a \leq L^{3/4}$, 
\begin{eqnarray}
&& \Pbm\(   \WW_{y,k}(N_{k,a} )\cap \BB_{y,k+3,L}\cap 
\HH_{k,a}   \cap \Iyzp\)  
\nn\\
&& \leq c   e^{ -2(L-k^{+ +})  }  \sum _{j' = \left(k^{ +}\right)^{1/4}}^{L^{3/4}} \sum _{j'' =  \left(k^{ ++}\right)^{1/4}}^{L^{3/4}}( 1+  j'' )e^{ -2 j'' }
\nn\\
  &&\hspace{1 in}\times \Pbm\(   \WW_{y,k}(N_{k,a} )\cap \Bbb{  B}_{y,k+3,k^{ +}, k^{ ++} }^{j', j'' }\cap \HH_{k,a}   \cap \Iyzhpk\),\label{61.3}
\end{eqnarray}
 where $\Bbb{  B}_{y,k+3,k^{ +}, k^{ ++} }^{j', j'' }$ is as in \eqref{167.1q}.
 % \bea
 % \label{67.1q}
%&&  \Bbb{  B}_{y,k+3,k^{ +}, k^{ ++} }^{j', j'' }
%=\left\{\gamma \left(l\right)\le\sqrt{2T_{l}^{y,t_{z}}} \mbox{ for }
%l=k+3,\ldots,k^{ +},  k^{ ++}\right.\nn\\
%&&\left.\hspace{1 in} \,\sqrt{2T_{k^{ +}}^{y,t_{z}}} \in 
%I_{\rho_{L}(L-k^{ +})+j'},\,\sqrt{2T_{k^{ ++}}^{y,t_{z}}} \in 
%I_{\rho_{L} (L-k^{ ++})+j''}\right\}.
%\eea

To handle the last term in \eqref{61.3}, we use the following
estimate, whose proof, given in the following sub-section, uses the decoupling lemma (Lemma 
\ref{lem-decoup}), barrier estimates and
the control on Wasserstein distance contained in (\ref{was1}).

\bl\label{lem-basicveb}
For some $M_{0}<\ff$ and $k,a,j'$ in the ranges specified above,
\bea
&&
\Pbm(   \WW_{y,k}(N_{k,a} )\bigcap   \Bbb{  B}_{y,k+3,k^{ +}, k^{ ++} }^{j', j'' }
\bigcap \HH_{k,a}   \bigcap \corJ{\Iyzhpk} ) \label{basicveb.1}\\
 && \leq F_{k,a,j', j''}\,\,\Pbm(\corJ{\Iyzhpk})+
e^{-2L-\sqrt{L}},
\nn
\eea
where
 \begin{eqnarray}
 &&F_{k,a,j',j''}= c ae^{-2a}e^{-2 (k^{ ++}-k  )     }  \sum_{\stackrel{\{\bar j': |\bar j'-j'|\leq 2M_{0}\log    k\}}{\{\bar j'': |\bar j''-j''|\leq 2M_{0}\log    k\}}}
 \label{was22p}\\
 &&   \hspace{.5 in} \bar j'\,\frac{e^{  -(\bar j'-a)^{ 2}/(4( k^{ +}-k )   )}}{(k^{ +}-k   ) }     e^{2 \bar j'' } \, \frac{e^{-(\bar j'-\bar j'')^{2}/2(k^{ +}-k  )}}{(k^{ +}-k   ) }.\nonumber
 \end{eqnarray}
%\[F_{k,a,j'}= cae^{-2a} \!\!\!\!\!  \sum_{\{j: |j-j'|\leq 2M_{0}\log    k\}}  
%\!\!
%\frac{ j}{(k^{ +}-k   )^{ 3/2}}e^{ -2(k^{ +}-k   )   +2j-(j-a)^{ 2}/( 2(k^{ +}-k )   )} ze^{ -2L-2z }.\]
\el

Assuming Lemma \ref{lem-basicveb}, 
we can now complete the proof of Lemma \ref{prop:UpperBoundz}
for $d^{\ast}\le k\leq L_{-}=16\log^{4}L$.
In view of  (\ref{61.2})-(\ref{61.2m}),\eqref{old4.65},\eqref{161.4}, and  using the barrier estimate
 Lemma \ref{prop:BarrierSecGWProp} for $\Pbm(\corJ{\Iyzhpk})$, we see that the contribution of 
$F_{k,a,j',j''}$ to (\ref{eq:UpperBoundz2}) is bounded by 
 \begin{eqnarray}
 &&  cz e^{  - 4L+2k -2z } ae^{-2a}  e^{8 M_{0}\log    k} \sum_{\stackrel{\{\bar j': |\bar j'-j'|\leq 2M_{0}\log    k\}}{\{\bar j'': |\bar j''-j''|\leq 2M_{0}\log    k\}}}
 \label{514.5}\\
 &&   \hspace{.5 in} \bar j'\,\frac{e^{  -(\bar j'-a)^{ 2}/(4( k^{ +}-k )   )}}{(k^{ +}-k   ) } (\bar j''+2M_{0}\log    k)     \, \frac{e^{-(\bar j'-\bar j'')^{2}/2(k^{ +}-k  )}}{(k^{ +}-k   ) }.\nonumber
 \end{eqnarray}
 Using the fact that $a\geq k^{1/4}$, and recalling that $E(   k)=e^{-4L+2k}ze^{ -2z  }$, see
\eqref{eofk}, this is bounded by
 \begin{equation}
 cE(   k)\,\, e^{-3a/2}  \sum_{j' }j'\,\frac{e^{  -(  j'-a)^{ 2}/(4( k^{ +}-k )   )}}{(k^{ +}-k   ) } \sum_{j''}   j''     \, \frac{e^{-(  j'-  j'')^{2}/2(k^{ +}-k  )}}{(k^{ +}-k   ) }.\label{514.60}
 \end{equation} 
 The sum over $j''$ can be bounded by $c(1+j'/(k^{ +}-k   )^{1/2})$. Thus we can bound (\ref{514.60}) by 
  \bea
  &&
 cE(   k)\,\, e^{-3a/2}  \sum_{j' }\,\( j' (k^{ +}-k   )^{1/2}+ j'^{2} \)\,\frac{e^{  -(  j'-a)^{ 2}/(4( k^{ +}-k )   )}}{(k^{ +}-k   )^{3/2} }.\label{514.60a}\\
&& \leq  ca^{2}e^{-3a/2}E(   k)\leq  ce^{-a}E(   k).\nn
 \eea
 
% \textcolor{red}   {End new material }

Similarly, the contribution of the
last term of \eqref{basicveb.1}
to (\ref{eq:UpperBoundz2}) is bounded by\footnote{Terms depending on $z$ are absorbed in the factor $\sqrt{L}$.}
\begin{equation}
cL^{3/2}e^{ -2(L-k^{+ +}) }e^{-2L-\sqrt{L}},\label{extra.3}
\end{equation}
which, after summation in $k^{1/4}\leq a\leq L^{3/4}$, is easily seen to be  bounded by $ce^{-k^{1/4}}E(   k)$.
Together with the previous displays, this  
completes the proof of Lemma \ref{prop:UpperBoundz} for 
$d^{\ast}\le k\leq L_{-}$.

    \corJ{It thus remains to consider the case where 
 $1\le k< d^{\ast}$. We show below that  for some $c,c'<\ff$, 
 \begin{equation}
\Pbm\(\Iyz \cap \Iyzp\)\leq  c'  (1+z)e^{-4L+2d^{\ast}}e^{ -2z  }e^{-c(d^{\ast}_{L})^{1/4}},\label{dast}
\end{equation}
whenever  $2h_{d^{\ast}}<d(y,y')$, that is, without the condition that $d(y,y')\leq 2h_{d^{\ast}-1}$.
Then,  since for  $1\le k< d^{\ast}$,   $2h_{k}<d(y,y')$ implies $2h_{d^{\ast}}<d(y,y')$, and    $d^{\ast}$ is fixed, by increasing the constant $c'$ in (\ref{dast}) we obtain (\ref{eq:UpperBoundz}) for $1\le k< d^{\ast}$.
 }
%  We have 
%  \begin{eqnarray}
% \Pbm( \Iyz\cap \Iyzp)&\leq&
% \sum_{a\geq [k^{1/4}]} \Pbm( \WW_{y,k}(N_{k,a } )\cap 
% \BB_{y,k+3,L} \cap \HH_{k,a}   \cap \Iyzp)\nonumber\\
% &\leq& 
% \sum_{a\geq [(d^*)^{1/4}]} \Pbm( \WW_{y,d^*}(N_{d^*,a } )\cap 
% \BB_{y,d^*+3,L} \cap \HH_{d^*,a}   
% \cap \Iyzp)
% .\label{61.0jedadd}
% \end{eqnarray}
% Since $d(y,y')\geq 2h_{k}\geq 2h_{d^*}$, we can apply
% For this we can take (\ref{61.0jed})
% with $k$ replaced by $d^{\ast}$ and use
% (\ref{eq:UpperBoundz2}) and (\ref{eofk}) for $k=d^*$
% (together with the fact that 
% $d^{\ast}$ is fixed) to obtain the claimed bound. 

    \corJ{To prove (\ref{dast}) solely under the condition that $2h_{d^{\ast}}<d(y,y')$, we return to the proof of (\ref{eq:UpperBoundz}) for $k=d^{\ast}$.       The condition  $2h_{k}<d(y,y')\leq 2h_{ k-1}$ is needed to guarantee that  $\Iyzhpk \in  {\mathcal G}_{k^{+}}^y.$}       \corJ{If instead of just dropping the barrier in $\Iyzhpk$ for $l=k-1, k$ we were to drop it for all $l\leq k$ and replace $\Iyzhpk$ by 
    \begin{equation}
\wh\JJ_{y',k,z}= \left\{ \gamma\left(l\right)\le\sqrt{2T_{l}^{y,t_{z}}} \mbox{ for }l=  k+1,\ldots  L-1;  T_{L}^{y,t_{z}}=0\right\},\label{eq:TruncatedSummandLBzqq}
\end{equation}  
then $\wh\JJ_{y',k,z} \in  {\mathcal G}_{k^{+}}^y.$   We did not use this for general $k$ since dropping the barrier for  a unbounded number of $l$'s can effect our estimates. However, $ d^{\ast}$ is fixed so we can replace $\wh\II_{y',d^{\ast},z}$ by $\wh\JJ_{y',d^{\ast},z}$ and following our proof of (\ref{eq:UpperBoundz}) for $k=d^{\ast}$ we will 
obtain (\ref{dast}) whenever  $2h_{d^{\ast}}<d(y,y')$.}

  This  completes the proof of  Lemma \ref{prop:UpperBoundz}. 
\qed

\subsection{Decoupling}\label{sec-Decoupling}
In this section we assume that  $d^{\ast}\leq k\leq L^{1/2}$.

 Let  $\Psi=\{\psi_{k,i},\,\,    i=1,2,\ldots   \}$ be a collection  of independent $\nu_{k}$-distributed random variables, independent of the
 Brownian motion X. We set
\begin{equation}
\VV_{y,k}(   n)=\left\{   \dwa\(   \frac{1}{n} \sum_{i=1}^{ n} \de_{\psi_{k,i}},\nu_{k}  \)\leq \frac{c_{0}\log k}{2\sqrt{n}}     \right\},\label{eq:LBTruncatedSumz8}
\end{equation}
and define $\WW_{y,k}(   n)$ similarly with the $\psi_{k,i}$ replaced by $\th_{k,i}$.
We set
\begin{equation}
  \label{short}
  \gykm= \VV_{y,k}(N_{k,a } )\cap \WW_{y,k}(N_{k,a }) \cap \HH_{k,a}.
\end{equation}

We will often use the notation $P_{X, \Psi}$ instead of
$\Pbm\times \nu_k^{\Z_+}$, while emphasizing  
that we are dealing with two independent processes. 
It follows from (\ref{was1}) that for any event 
$\mathcal{E}$ which is measurable with respect to $X$,
\begin{equation}
\Pbm \( \WW_{y,k}(N_{k,a } ) \cap \HH_{k,a}   \cap \mathcal{E}\)
\label{cupdef.1azk}
 \leq 2 P_{X, \Psi}\(\GG_{y,k}(N_{k,a})
\cap  \mathcal{E}\).
\end{equation}
Set $\bar \th_{k,N } =(   \th_{k,1},   \th_{k,2},\ldots,   \th_{k,N })$ where the $\th_{k,i }$, as above, are the angular increments of the excursions at level $k$. Similarly we set 
$\bar \psi_{k,N } =$ \\$(   \psi_{k,1},   \psi_{k,2},\ldots,   \psi_{k,N })$.

On $\gykm$,
by the triangle inequality we have that
\begin{equation}
 \dwa\(   \frac{1}{ N_{k,a } } \sum_{i=1}^{ N_{k,a } } \de_{\psi_{k,i}}, 
 \frac{1}{N_{k,a } } \sum_{i=1}^{ N_{k,a } } \de_{\th_{k,i}} \)\leq\frac{c_{0}\log k}{\sqrt{N_{k,a} }}.\label{was4}
\end{equation}

If $\Qb\in \PP^2(\mu,\nu)$ with
marginals $ \mu=\frac{1}{ N_{k,a } } \sum_{i=1}^{ N_{k,a } } \de_{\th_{k,i}}$ 
and $\nu =\frac{1}{N_{k,a } } \sum_{i=1}^{ N_{k,a } } \de_{\psi_{k,i}}$ and 
distinct $\th_{k,i}$ and $\psi_{k,i}$, then we can write 
\[\Qb =\frac{1}{N_{k,a } }\sum_{i,j=1}^{ N_{k,a } }q_{i,j}\,\,\de_{\th_{k,i}}\times \de_{\psi_{k,j}},\] and
\[  \frac{1}{N_{k,a } }=\mu  (\th_{k,i} )=
\frac{1}{N_{k,a } } \sum_{j=1}^{ N_{k,a } }q_{i,j},
\quad 
\frac{1}{N_{k,a } }=\nu  (\psi_{k,j} )=\frac{1}{N_{k,a } } \sum_{i=1}^{ N_{k,a } }q_{i,j}. \]
Thus, 
$\{q_{i,j}\}$ is a doubly stochastic matrix. 
Hence, recall (\ref{was0}), (\ref{was4}) says that
\begin{equation}
\inf _{A\in DS( N_{k,a }   )} \sum_{i,j=1}^{ N_{k,a } }A_{i,j} |\th_{k,i}-\psi_{k,j} |  \leq c_{0}\log k \sqrt{N_{k,a } },\label{was4m}
\end{equation}
where the infimum is over the set of  doubly stochastic $N_{k,a } \times N_{k,a } $ matrices, denoted $DS(N_{k,a})$.

Using G. Birkhoff's theorem that any doubly stochastic  matrix is a convex 
combination of  permutation matrices,  \cite{BvN, M},  (\ref{was4m}) implies that there is a permutation $\pi=\pi_{\bar \th_{k,N_{k,a }},\bar \psi_{k,N_{k,a }}}$ of $[1,N_{k,a } ]$ such that, on $\gykm$,
\begin{equation}
\sum_{i=1}^{ N_{k,a } }  |\th_{k,i}-\psi_{k,\pi(   i)}|\leq c_{0}\log k\sqrt{N_{k,a } }.\label{was5}
\end{equation}

Let $\mathcal{C}(   y, k+2) $ denote the space of finite sequences of continuous paths from 
$\partial B_{d}(   y, h_{k+2})$ to $\partial B_{d}(   y, h_{k+1})$. We define an equivalence relation $\mathcal{R}$ on $\mathcal{C}(   y, k+2) $ by saying that two sequences in $\mathcal{C}(   y, k+2) $ are equivalent if they differ by some rotation around $y$. Let $\wt{\mathcal{C}}(   y, k+2)= \mathcal{C}(   y, k+2)/ \mathcal{R}$, and let $\Delta$ denote an auxiliary point, the `cemetery state'.
To each $\al\in [0,2\pi)$,
  we associate a random variable $\bar Y^{\al }$ with values in 
  $\wt{\mathcal{C}}(   y, k+2)\cup \{\Delta\}$, as follows. 
  Consider the  Brownian excursion started at some point in $\partial B_{d}(   y, h_{k})$ until exiting $\partial B_{d}(   y, h_{k-1})$, conditioned so that $\al$ is the angular increment between its initial and terminal point. If this excursion reaches $\partial B_{d}(   y, h_{k+2})$, we let  $\bar Y^{\al }$ be the element of $\wt{\mathcal{C}}(   y, k+2) $ generated by our excursion. (There may be several  excursions from 
$\partial B_{d}(   y, h_{k+2})$ to $\partial B_{d}(   y, h_{k+1})$  until exiting $\partial B_{d}(   y, h_{k-1})$, which is why $\mathcal{C}(   y, k+2) $ involves sequences of excursions). If our excursion exits $  B_{d}(   y, h_{k-1})$ before reaching $\partial B_{d}(   y, h_{k+2})$, we set 
 $\bar Y^{\al }=\Delta$.  In this manner, for each $\psi_{k,j}$, $j=1,2,\ldots$ we define $\bar Y^{\psi_{k,j} } \in \wt{\mathcal{C}}(   y, k+2)\cup \{\Delta\}$. Recalling the definition of  $\th_{k,i}$,  
 it follows from the Markov property that the 
 excursions $Y^{\th_{k,i} } \in \wt{\mathcal{C}}(   y, k+2)
 \cup\{\Delta\}$, $i=1,2,\ldots$, generated by the
 Brownian motion $X$, have the same law as  $\bar Y^{\th_{k,i} }$, 
 $i=1,2,\ldots$.
 Note that  on the event $\HH_{k,a}  $ we have that 
 $T_{k}^{y,t_{z}}\leq N_{k,a}$.

  By a direct computation using the Poisson kernel in
  \eqref{poisson}, we see that there exists a universal  constant $c_1$
  so that if $x,x'\in  \partial B_d(0,h_{k-1})$ then 
  \bea
    \label{eq-poissondif}
  &&  \max_{u,z\in \partial B_{d}(0,h_k)}
    \left|
    \frac{p_{B_{d}(0,h_{k-1})}(z,x)} {p_{B_{d}(0,h_{k-1})}(u,x)}
    -\frac{p_{B_{d}(0,h_{k-1})}(z,x')} {p_{B_{d}(0,h_{k-1})}(u,x')}\right|\\
    &&=\max_{u,z\in \partial B_{d}(0,h_k)}
    \left|\frac{ \sin^{ 2} ( d(  u,x)/2)}
{ \sin^{ 2} ( d( z,x)/2)}-\frac{ \sin^{ 2} ( d(  u,x')/2)}
{ \sin^{ 2} ( d( z,x')/2)}\right|
    \leq c_1 d_a(x,x'),
    \nonumber
  \eea
  where $d_a(x,x')$ denotes the difference of arguments of $x$ and $x'$, which
  can be taken by definition (since $h_{k-1}$ is small and the ratio
  $h_{k-1}/h_k$ is fixed) as $d(x,x')/h_{k-1}$, 
  and we used the fact that $\frac{ \sin^{ 2} ( d(  u,x)/2)}
{ \sin^{ 2} ( d( z,x)/2)}$ is Lipschitz in $x$, uniformly in $u,z\in \partial B_{d}(0,h_k)$.

  With $c_1$ as in \eqref{eq-poissondif}, 
 let $p_{i}= (c_{1}|\th_{k,i}-\psi_{k,\pi(   i)}|)\wedge 1$.
 Note that
 by (\ref{was5}),
 with $c_2=c_{0}c_{1}$,
 we have that, on $\gykm$, 
\begin{equation}
\sum_{i=1}^{ N_{k,a } }p_{i}\leq c_2\log k \sqrt{ N_{k,a } }.\label{was5d}
\end{equation}
Let $B_{i}, i=1,\ldots, N_{k,a } $  be 
independent Bernoulli random variables with mean $p_{i}$, independent of 
the Brownian motions.   
We can use $B_i$ to create a coupling between
$Y^{\th_{k,i} } $ and $\bar Y^{\psi_{k,\pi(   i)} } $, as follows.

\begin{lemma}
  \label{lem-coupling}
For each $1\leq i\leq N_{k,a }$ there exist random variables $\wh Y_{i}, Z_{i}, \bar Z_{i}$  in $\wt{\mathcal{C}}(   y, k+2)\cup \{\Delta\}$
so that 
%$Z_i$ and $\bar Z_i$ are independent and 
\begin{eqnarray}
&&Y^{\th_{k,i} } =(   1-B_{i})\wh Y_{i}+B_{i}Z_{i}
\label{coupl.0}\\
&&\bar Y^{\psi_{k,\pi(   i)}} =(   1-B_{i})\wh Y_{i}+B_{i}\bar Z_{i}.\nn
\end{eqnarray}
The random variables $B_i$ and $\wh Y_i$ can be taken independent of each other.
\end{lemma}
We refer to the joint law of $B_{i}, \wh Y_{i}, Z_{i}, \bar Z_{i}$ in the lemma as
$\wt P^{\th_{k,i}, \psi_{k,\pi(   i)} }$.
\begin{proof}
Without loss of generality we take $y=0$. (\ref{coupl.0}) will follow from \cite[Theorem 5.2]{L}\footnote{The formulation in \cite[Theorem 5.2]{L} allows for the  independence of 
$Z_i$ and $\bar Z_i$ of each other, with the change  that the parameter of the 
Bernoulli $B_i$ equals the variation distance. By splitting $B_i$ in that formulation, we arrive at the current statement.}   
once we show that the variational distance between the distributions $Y^{\th_{k,i} } $ and $\bar Y^{\psi_{k,\pi(   i)} } $  is less than 
 $c_{1}|\th_{k,i}-\psi_{k,\pi(   i)}|$. The latter follows from an application Poisson kernel estimates, 
 similarly to the proof of Lemma \ref{recursionends}. We provide the details below.
 %, with $n=1$ and $k'=k+2$.

Let $\bar{\tau} = \inf \{ t:\, X_t \in \partial B_{d}(0, h_{k-1})\}$,
$\tau_0=0$ and for $i=0,1,\ldots$ define
\begin{eqnarray*}
\tau_{2i+1} & = & \inf \{ t \geq \tau_{2i}  :\; X_t \in \partial
B_{d}(0, h_{k-1}) \cup \partial B_{d}(0, h_{k+2})\} \\
\tau_{2i+2} & = &
 \inf \{ t \ge \tau_{2i+1} :\; X_t \in \partial B_{d}(0, h_{k+1})\}\,.
\end{eqnarray*}

If $\tau_{2i+1}<\bar{\tau}$, let $\mathcal{E}_{i}$ denote the 
excursion  from time $\tau_{2i+1}$ until $\tau_{2i+2}$, while if  $\tau_{2i+1}\geq\bar{\tau}$, let $\mathcal{E}_{i}=\emptyset$.
Then let 
$\mathcal{I}_{j}$ denote the $\si$-algebra of  events measurable with respect to   $\mathcal{E}_{i}$, $0\leq i\leq j$, which are 
rotationally invariant around $0$. We set $\mathcal{I}=\cup_{j=0}^{ \ff} \mathcal{I}_{j}$.

If $\AA\subseteq \mathcal{I}_{j}-\mathcal{I}_{j-1}$, using the Markov property we have, for any $u\in \partial B_{d}(0, h_{k})$, 
%\textcolor{red}{removed a line}
\begin{eqnarray}
\Pbm^{u}\(\AA \,\,\big |\, X_{\bar{\tau}}=x \)
&=&\Pbm^{u }\(\AA,\,\, \tau_{2j+2}<\bar{\tau} \,\,\big |\, X_{\bar{\tau}}=x \)
\nn\\
&=&\frac{\Ebm^{u}\(\AA,\,\,\tau_{2j+2}<\bar{\tau};\,\,p_{B_d(   0, h_{k-1 } )}( X_{\tau_{2j+2}},x)  \)}{ p_{B_{d}(   0, h_{k-1})}(u,x)},\nn
\end{eqnarray}
see \eqref{poisson}.
Using \eqref{eq-poissondif}, we obtain that  for
$\AA\subseteq \mathcal{I}_{j}-\mathcal{I}_{j-1}$,
\[  |\Pbm^{u}\(\AA\,\,\big |\, X_{\bar{\tau}}=x \)-\Pbm^{u }\(\AA\,\,\big |\, X_{\bar{\tau}}=x' \)|
\leq c_{1}
d_a(x,x')P^{(u) }\(\AA  \) ,  \]
    uniformly in $\AA$ and $j$.
  Hence if $\BB=\cup_{j=0}^{ \ff}\AA_{j}$ with $\AA_{j}\subseteq \mathcal{I}_{j}-\mathcal{I}_{j-1}$, then
  \begin{eqnarray}
  &&|\Pbm^{u }\(\BB\,\,\big |\, X_{\bar{\tau}}=x \)-
  \Pbm^{u }\(\BB\,\,\big |\, X_{\bar{\tau}}=x' \)|\nonumber \\
  &&\quad \leq   \sum_{j=0}^{ \ff}   |\Pbm^{u}\(\AA_{j}\,\,\big |\, 
  X_{\bar{\tau}}=x \)-\Pbm^{u}\(\AA_{j}\,\,\big |\, 
  X_{\bar{\tau}}=x' \)|\nn\\
  &&\quad \leq   c_{1}   \,d_a(x,x')\,\, 
  \sum_{j=0}^{ \ff}\Pbm^{u}\(\AA _{j} \)= c_{1}\,d_a(x,x')\,\,
  \Pbm^{u }\(\BB\) 
  \leq   c_{1}  \,d_a(x,x'),\nn
  \end{eqnarray}
  uniformly in $\BB$.
  
  Finally we need to consider the event that $\{\tau_1=\bar\tau\}$.
  Since
  \[\Pbm^{u}\(\tau_1=\bar\tau \,\,\big |\, X_{\bar{\tau}}=x \)=
  1-  \Pbm^{u}\(\tau_1\neq \bar \tau \,\,\big |\, X_{\bar{\tau}}=x \),
\]
  and since 
    \[\{\tau_{1}\neq  \bar{\tau}\}= \cup_{j=0}^{ \ff}\{  \tau_{2j+2}<\bar{\tau}\}\in \mathcal{I},\] 
the desired result follows from the previous paragraph.
 \end{proof}

We define 
\begin{equation}
\wt P=\wt P^{\bar \th_{k,N_{k,a }},\bar \psi_{k,N_{k,a }} }=\bigotimes_{i=1}^{ N_{k,a }}\wt P^{\th_{k,i}, \psi_{k,\pi(   i)} },\label{cupdef.1}
\end{equation}
where $\pi=\pi_{\bar \th_{k,N_{k,a }},\bar \psi_{k,N_{k,a }}}$,
see \eqref{was5} and $\wt P^{\th_{k,i}, \psi_{k,\pi(   i)} }$ is as in
Lemma \ref{lem-coupling}.
If $\UU$ is an event measurable on
$Y^{\th_{k,i} } $, $i=1,\ldots,  T_{k}^{y,t_{z}}$, 
then for any event $\mathcal{V}$ 
which depends only on excursions outside of $  B_{d}(y, h_{k})$
we have 
\[  P_{X, \Psi}\( \gykm 
\cap \UU  \cap \mathcal{V}\)
=E_{X, \Psi}\( \wt P^{\bar \th_{k,N_{k,a }},\bar \psi_{k,N_{k,a }} }\(\UU\), 
\gykm \cap \mathcal{V}\).\]

We use $\wt P^{ \psi}$ to denote the law of $\sum_{i=1}^{ N_{k,a }}\de_{\bar Y^{\psi_{k,i}} }$ under $\wt P$. We note that
\begin{equation}
\wt P^{ \psi}\mbox{ does not depend on  }\,\bar \th_{k,N_{k,a }}.\label{cupdef.2}
\end{equation}

 $Y^{\th_{k,i} } $ and $\bar Y^{\psi_{k,\pi(   i)}} $ will coincide  with  probability at least  $1-p_{i}$ under $\wt P$. We will call this a success. Let $\Ga_{k,a }=\Ga_{k,a }(   \bar \th_{k,N_{k,a } },\, \bar \psi_{k,N_{k,a } }  )$ be the number of     successes among the $N_{k,a }$ excursions. 
 We have the following.
 
\bl\label{lem-numberofsucceses} There exists $d^*$ sufficiently large
so that, for $d^*\leq k\leq L^{1/2}$,
  and all $L$ large, 
  \[ \wt{P}\(\Ga_{k,a } <  N_{k,a } -2c_2L\log k \)\leq    e^{- 10L}.\]
\el

\begin{proof}
The number of failures is at most $\sum_{i=1}^{N_{k,a } }B_{i}$, 
and using $\lambda>0$ 
for the first inequality,
$1+px\leq e^{px}$ for the third and
(\ref{was5d}) for last inequality,
we obtain
\begin{eqnarray*}
&&\wt P(\sum_{i=1}^{N_{k,a } }B_{i}\geq 2c_2\log k\sqrt{N_{k,a } })
\leq   e^{-2\lambda c_2 \log k\sqrt{N_{k,a } }}
\prod_{i=1}^{N_{k,a } }\wt E(e^{\la B_{i}})\\
&&\quad \quad =        \prod_{i=1}^{N_{k,a }}\(1+p_{i}(e^{\la}-1)\)
e^{-2\la c_2 \log k\sqrt{N_{k,a } }}\nonumber\\
&&\quad \quad
\leq e^{\sum_{i=1}^{N_{k,a } }p_{i}(e^{\la}-1)} \,\,e^{-2\la c_2 \log k \sqrt{N_{k,a }}}
 \leq  e^{ c_2\log k\sqrt{N_{k,a }}\(e^{\la}-1-2\la\)}.
\end{eqnarray*}   
Hence taking $\la>0$ small, recalling that 
$N_{k,a} \geq L^{2}$ see
\eqref{nkdef.2}, we have that for some $c_{3}>0$
\[\wt P\(\sum_{n=1}^{N_{k,a } }B_{n}\geq 2c_2\log k\sqrt{N_{k,a }  }\)\leq e^{- c_3
L\log k}\leq e^{- 10L}\]
 by choosing $d^{\ast}$ sufficiently large so that $ \log k$ is sufficiently large.  
 \end{proof}

Thus, if  we set 
\begin{equation}
\Phi_{k,a }=\Phi_{k,a }(   \bar \th_{k,N_{k,a } },\, \bar \psi_{k,N_{k,a } }  )=\Big\{\Ga_{k,a } \geq  N_{k,a } -2c_2L \log k  \Big\},\label{was6a}
\end{equation}
we have  
\begin{equation}
  P_{X,\Psi}(\gykm
\cap \UU  \cap \mathcal{V})
\leq E_{X, \Psi}(\wt  P( \UU  \cap \Phi_{k,a }),\gykm
\cap \mathcal{V})+e^{- 10L}.\label{was10}
\end{equation}

Recall (\ref{hka.1}) and assume that $a\leq L^{3/4}$. Let  $N=T_{k}^{y,t_{z}}$. On the event $\HH_{k,a}  $ we have $N_{k,a-1}\leq N\leq N_{k,a}$ 
and 
$N_{k,a }-N_{k,a-1}\leq  4L$. 
 Let 
 \begin{equation}
   \label{eq-uncoupled}
   A=A_\theta=\{i \,|\,\th_{k,i}\in  \bar\th_{k,N } 
 \mbox{ that did not successively couple } \}
 \end{equation}
 and with $A^{c}_{N}=\{1,\ldots, N\}\cap A^{c}$ set
  \begin{equation}
   \label{eq-uncoupledB}
   B=B_\psi=\{i \,|\,\psi_{k,i}\in  \bar\psi_{k,N_{k,a } } 
 \mbox{ that did not successively couple with } A^{c}_{N}\}.
 \end{equation}
 Then under $\Phi_{k,a }$, recalling that $N=T_{k}^{y,t_{z}}$,
 \[\bar \tau_{k,N }=\{   \th_{k,i},\,i\in A \}\cup
 \{   \psi_{k,j},\,j\in B \}\] satisfies that
 $| \bar\tau_{k, N  }|\leq 4c_2 L\log k$.   
 Recall the notation \eqref{eq-traver2} and
 use $T_{l}^{y,k,  \bar\tau_{k,N }}$ to denote 
 the number of excursions at level $l $ that occur during the 
 $| \bar\tau_{k,N }|$ excursions 
 $\{ Y^{\th_{k,i} },\,i\in A \}\cup 
 \{\bar  Y^{\psi_{k,j} },\,j\in B \}$.  
 With $M_0$ to be defined below, see Lemma \ref{lem-IGW}, let 
\begin{equation}
\mathcal{A}_{N,k}=\bigcup_{l=k+3}^{k^{+ +}}   \Big\{ T_{l}^{y,k,  \bar\tau_{k,N } } \ge M_{0}\,\rho_{L}(L-l)\log k \Big\}, 
\mbox{\rm with $k^{+ +}=k+200 [\log L]$}. \label{12was.0}
\end{equation}
In words, $\mathcal{A}_{N,k}$ is the event that for some
$l\in [k+3,k^{+ +}]$, the uncoupled excursions generated an excessive number
of excursions at level $l$.
We have 
\begin{equation}
\UU\subseteq  \(\UU\cap \mathcal{A}^{c}_{N,k}\)\cup \mathcal{A}_{N,k}.\label{was10big}
\end{equation}

We can now state our main decoupling lemma.
It will be under the following assumption.
\begin{assumption}
   $\UU$ is measurable with respect to
    $Y^{\th_{k,i} } $, $i=1,\ldots,  T_{k}^{y,t_{z}}$, and there
  exists an event $\wt \UU$ measurable on $\sum_{i=1}^{ N_{k,a }}\de_{\bar Y^{\psi_{k,i}} }$ so that
  \be \UU\cap \mathcal{A}^{c}_{N,k} \cap \Phi_{k,a }\subseteq \wt \UU(\sum_{i=1}^{ N_{k,a }}\de_{\bar Y^{\psi_{k,i}} }).\label{assump/P}
  \ee
\label{bigassump.1}
\end{assumption}
In words, the assumption allows for events that, whenever
there are not too many uncoupled excursions at level $k$ 
and the latter do not generate too many excursions at a level in $[k,k^+]$, 
can be dominated by an event using only the (empirical measure of the)
coupled excursions. In a typical application, $\mathcal{U}$ will be
a barrier event for the original excursions, and $\wt \UU$ will be 
a barrier event for the coupled excursions, with a slightly modified barrier.

\bl[decoupling lemma]\label{lem-decoup}
Let $\UU$ 
satisfy Assumption \ref{bigassump.1}, 
and let $\mathcal{V}$ be an event  
which depends only on excursions outside of $  B_d(y, h_{k})$. Then for all $d^*\leq k\leq L^{1/2}$ and $a\leq L^{3/4}$
\begin{eqnarray*}
&&\Pbm \( \WW_{y,k}(N_{k,a } ) \cap \HH_{k,a} \cap \UU   \cap \mathcal{V}\)
\label{decoup.basic}\\
&& \leq c \Pbm  \(\wt \UU\(\sum_{i=1}^{ N_{k,a }}\de_{ Y^{\th_{k,i}}}\)\) 
\times \Pbm\(\WW_{y,k}(N_{k,a } )\cap \HH_{k,a}   \cap\mathcal{V}\)\nn\\
&&+ c  e^{-5L\log k/(\log L)}       \Pbm\(\WW_{y,k}(N_{k,a } )\cap 
\HH_{k,a}   \cap \mathcal{V}\)+e^{- 10L}.\nonumber
\end{eqnarray*}
\el

\begin{proof}
By (\ref{cupdef.1azk})
\begin{equation}
\Pbm \( \WW_{y,k}(N_{k,a } ) \cap \HH_{k,a}   \cap  \UU   \cap \mathcal{V}\)
\label{cupdef.1azk7}
 \leq 2 P_{X, \Psi}\(\GG_{y,k}(N_{k,a})
\cap   \UU   \cap \mathcal{V}\),
\end{equation}
and by (\ref{was10})
\begin{equation}
  P_{X,\Psi}(\gykm
\cap \UU  \cap \mathcal{V})
\leq E_{X, \Psi}(\wt  P( \UU  \cap \Phi_{k,a }),\gykm
\cap \mathcal{V})+e^{- 10L}.\label{was107}
\end{equation}

Using (\ref{was10big}) we can bound
\begin{eqnarray}
&&   E_{X, \Psi}(\wt  P( \UU  \cap \Phi_{k,a }),\gykm
\cap \mathcal{V})
\label{was10bigb}\\
&& \leq  E_{X, \Psi}\(\wt  P\( \UU\cap \mathcal{A}^{c}_{N,k} \cap \Phi_{k,a }\) \gykm 
  \cap \mathcal{V}\)\nonumber\\
  &&+E_{X, \Psi}\(\wt  P\( \mathcal{A}_{N,k} \cap \Phi_{k,a } \),
\gykm \cap \VV\).\nonumber
\end{eqnarray}

Using our assumption \eqref{assump/P} we have 
\begin{eqnarray}
&&E_{X, \Psi}\(\wt  P\( \UU\cap \mathcal{A}^{c}_{N,k} \cap \Phi_{k,a }\), 
\label{extra.4}
\gykm 
  \cap \mathcal{V}\)
\\
&&\leq \wt P^{\psi}\(\wt \UU\(\sum_{i=1}^{ N_{k,a }}\de_{\bar Y^{\psi_{k,i}} }\)\) \times \Pbm\(\WW_{y,k}(N_{k,a } )\cap \HH_{k,a}   \cap  \mathcal{V}\)    \nonumber\\
&&=  P \(\wt \UU\(\sum_{i=1}^{ N_{k,a }}\de_{ Y^{\th_{k,i}} }\)\) \times 
\Pbm\(\WW_{y,k}(N_{k,a } )\cap \HH_{k,a} 
\cap \mathcal{V}\),    \nonumber
\end{eqnarray}
where the decoupling in the inequality came from Assumption
\ref{bigassump.1} and (\ref{cupdef.2}), 
and the final equality comes from the fact that 
$\sum_{i=1}^{ N_{k,a }}\de_{\bar Y^{\psi_{k,i}} }$ and 
$\sum_{i=1}^{ N_{k,a }}\de_{ Y^{\th_{k,i}} }$ have the same distribution.

We continue by  bounding the probabilities of the last term 
on the right side of
\eqref{was10bigb}. Toward this end we need the following lemma.

\bl \label{lem-IGW}
There exists
$M_{0}<\ff$ independent of $k$ and $L$ such that for all 
$l=k+3,\ldots,k^{+ +}=k+10^{10} [\log L]$, under $ \Phi_{k,a }$
\begin{equation}
\wt P\( \Big \{T_{l}^{y,k, \bar\tau_{k,N }  } \ge  M_{0}\,\rho_{L}(L-l) \log k \Big\}\,\big |\,   \bar\tau_{k,N }\)\leq 
 e^{-10L\log k/(\log L)}.\label{cup.9a}
\end{equation}
 \el

 \begin{proof}[Proof of Lemma \ref{lem-IGW}]
   We note that conditional on $\bar\tau_{k,N}$, the 
   $|\bar \tau_{k,N}|$ excursions
   that are used to construct 
   $T_{l}^{y,k, \bar\tau_{k,N }}$ are independent. 
  We enumerate these excursions by
   $i=1,\ldots,|\bar \tau_{k,N}|$.
   Write $T_l^{y,k,{(i)}}$ for the
   number of excursions at level $l $ that occur during
   the $i$'th excursion from $\bar\tau_{k,N}$. By Lemma \ref{recursionends} with $n=1$ (which is nothing but a standard use of   Poisson kernel estimates
   as in \cite{DPRZthick}),
   % , see
   there exists a constant $c$ (independent of $i,l,k$) so that
  \begin{equation}
    \label{eq-flightmiddle}
    \wt P(T_l^{y,k,{(i)}}=j)\leq \frac{c}{(l-k)^2}
    \left(1-\frac{1}{l-k}\right)^{j-1},\; j\geq 1.
  \end{equation}
%  To see this note that the right hand side bounds the probability for an excursion without conditioning on the endpoint, and we can use the 
  %Poisson kernel 
%   $p_{B_{d}(0,h_{k-1})}(z,x)$, see
 % \eqref{poisson}, to control the conditioning.
 % Indeed, due to radial symmetry, when considering 
 % two excursions from $|\bar \tau_{k,N}|$
 % with differing angular increments, we may consider
   %them ending at the same location on $\partial B_d(y,h_{k-1})$
  %while starting at different locations on $\partial B_d(y,h_{k})$. 
  %The Radon-Nykodim derivative between the (sub-probability) measures
  %describing the first 
  %hitting location on $\partial B_d(y,h_{k+1})$ is then easily
  %seen to be bounded by the maximum of
  %ratios of Poisson kernels at different locations
  %on $\partial B_d(y,h_{k+1})$, 
  %which is uniformly bounded. See \cite{DPRZ} for  similar arguments.
  %as soon as $
  % since denoting by $\wt P^y$
  %the law of an excursion  started at 
  %\begin{equation}
  %  \label{eq-ttt}
  %  \frac{\wt P.

    Let $T_{l-k}^{(j)}$, $j\geq 1$ denote independent
    copies of $T_{l-k}$ under $\Pgw_1$. Equation \eqref{eq-flightmiddle} implies
    that $T_l^{y,k,{(i)}}$ is stochastically dominated by a (finite) sum
    of independent copies of $T_{l-k}$ plus possibly a constant, in 
    the sense that
    there exist integers
    $n_1,n_2$ so that
    if $l-k>n_1$ then for all $s\geq 0$ integer, 
    \begin{equation}
      \label{eq-katz1}
      \wt P(T_l^{y,k,{(i)}}\geq s)\leq \Pgw_1{ }^{{}^{\otimes n_2}} (\sum_{j=1}^{n_2}
      T_{l-k}^{(j)}\geq s).
    \end{equation}
    (We use inclusion-exclusion and the fact that $l-k\geq n_1$ so that
    the probabilities in \eqref{eq-katz1} are small even for $s=1$, see
    \eqref{eq-flightmiddle}.) Similarly,
    if $l-k\leq n_1$ then
    \begin{equation}
      \label{eq-katz2}
      \wt P(T_l^{y,k,{(i)}}\geq s)\leq \Pgw_1{}^{ {}^{\otimes n_2}}
      (\sum_{j=1}^{n_2}
      T_{l-k}^{(j)}\geq s-n_2).
    \end{equation}
    (Here, we use the additive integer $n_2$ in order to make sure that
    even for small $s\geq 1$ when the left side in 
    \eqref{eq-katz2} is large, the inequality 
    remains true.)
    All in all, \eqref{eq-katz2} holds for all $l-k$.
    Therefore, applying 
    \eqref{eq-katz2},
    the Markov property, 
    and Lemma \ref{lem: GW proc LD}, (or more directly, \cite[Lemma 4.6]{BK}), we obtain that
\begin{eqnarray*}
    &&\wt P(  T_{l}^{y,k, \bar\tau_{k,N }  } \ge  M_{0}\,\rho_{L}(L-l) \log k\,\big |\,   \bar\tau_{k,N }) \label{multdef.10}\\
    &&\leq 
   \Pgw_1{}^{ {}^{\otimes n_2|\bar\tau_{k,N }|}}
      \Big(\sum_{j=1}^{n_2|\bar\tau_{k,N }|}
      T_{l-k}^{(j)}\geq  M_{0}\,\rho_{L}(L-l) \log k-n_2|\bar\tau_{k,N }|\Big)   \\
    &&\leq e^{-(\sqrt{M_{0}\,\rho_{L}(L-l) \log k-n_2|\bar\tau_{k,N }|}-\sqrt{n_2|\bar\tau_{k,N}|})^{2}/(l-k)}.
    \end{eqnarray*}
  Since $\rho_{L}(L-l)\sim 2L$, $| \bar\tau_{k,N }|\leq 4c_2 L\log k$ under $ \Phi_{k,a }$, and
  $l-k\leq 10^{10} \log L$, the lemma follows by choosing 
  $M_0$ large enough.
 \end{proof}
 We continue with the proof of Lemma \ref{lem-decoup}. 
Summing in \eqref{cup.9a} over 
$l=k+3,\ldots,k^{ ++}$ while considering 
$N$ with $\sqrt{N}\in I_{\rho_{L}(L-k)+a}$,
 and using the fact that $L/(\log L)\gg L^{3/4}\gg k^{+ +}$,  
we obtain that
\begin{equation}
\wt P\(   \mathcal{A}_{N,k} \cap \Phi_{k,a  } \)\leq 
 e^{-5L\log k/(\log L)}.\label{cup.9ar}
\end{equation}
% \el 
Using \eqref{cup.9ar} we see that
 \bea
 \label{easybad.1}
&&E_{X, \Psi}\(\wt  P\( \mathcal{A}_{N,k} \cap \Phi_{k,a } \),
\gykm \cap \VV\)\\
&&\quad 
\leq  c  e^{-5L\log k/(\log L)}       \Pbm\(\WW_{y,k}(N_{k,a } )\cap \HH_{k,a}   \cap\mathcal{V}\).% \leq  c  e^{-5mL/(\log L)}e^{ -2L-2z },
\nn
\eea
Combining \eqref{cupdef.1azk7}-\eqref{extra.4} and \eqref{easybad.1} completes the proof  
of Lemma \ref{lem-decoup}.
\end{proof}

As in the notation above, let  $T_{l}^{y,k, \bar \psi_{k,N_{k,a }} }   $ denote the number of excursions at level $l $ that occur during the $N_{k,a }$ excursions $ \bar  Y^{\psi_{k,j} },\,j=1,2,\ldots, N_{k,a }$.
In our applications, the functional on the right hand side of (\ref{assump/P}) will be a functional of the $T_{l}^{y,k, \bar \psi_{k,N_{k,a }} }$, for $l=k+3,\ldots, $ which  are functionals of $\sum_{i=1}^{ N_{k,a }}\de_{\bar Y^{\psi_{k,i}} }$.

%\textcolor{red}{Changes until end of section}

\begin{proof}[Proof of Lemma \ref{lem-basicveb}.]
  We begin the proof with a use of the decoupling lemma, Lemma \ref{lem-decoup}.
  Recall (\ref{167.1q}) and \eqref{eq:TruncatedSummandLBztt}. Let 
  $\UU= \Bbb{  B}_{y,k+3,k^{ +}, k^{ ++} }^{j', j'' }$ and $\mathcal{V}=
 \corJ{\Iyzhpk}$.
  %Recall the notation \eqref{12was.0} and \eqref{was6a}.
  We first claim, see \eqref{12was.0} and \eqref{was6a} for notation, that
  \bea
&&
\quad \quad
\UU\cap \mathcal{A}^{c}_{N,k } \cap \Phi_{k,a  }\label{bigassump.12}\\
&&\subseteq 
\bigcup_{\{s'\in \Z_{+}:  |s'-(\rho_{L}(L-k^{ +})+j')|\leq 2M_{0}\log k\}}\bigcup_{\{s''\in \Z_{+}:  |s''-(\rho_{L}(L-k^{ ++})+j'')|\leq 2M_{0}\log k\}}
\nn\\
&&\hspace{.5 in}\Big\{\rho_{L}(L-l)\le \sqrt{ 2 T_{l}^{y,k, \bar \psi_{k,N_{k,a }}  }   }\mbox{ for }l=k+3,\ldots, k^{ +},  k^{ ++},\nn\\
&&\hspace{1.5 in} \sqrt{2T_{k^{ +}}^{y,k, \bar \psi_{k,N_{k,a }} }}\in I_{s'},\, \sqrt{2T_{k^{+ +}}^{y,k, \bar \psi_{k,N_{k,a }} }}\in I_{s''}\Big\}.
\nn
\eea

To see (\ref{bigassump.12}), first note that  $|T_{l}^{y,t_{z}}-T_{l}^{y,k, \bar \psi_{k,N_{k,a }} }| \leq T_{l}^{y,k, \bar\tau_{k,N }  }$ by (\ref{coupl.0}) and the definition of $\bar\tau_{k,N }$. Hence  by  
(\ref{12was.0}), on  $ \Bbb{  B}_{y,k+3,k^{ +}, k^{ ++} }^{j', j'' }\bigcap
\mathcal{A}^{ c}_{N,k }$, for all $k+3\leq l\leq k^{ +},k^{+ +}$   
  we have 
\begin{equation}
|T_{l}^{y,t_{z}}-T_{l}^{y,k, \bar \psi_{k,N_{k,a }} } |\leq  M_{0}
\rho_{L}(L-l)\log  k\leq .01 l^{1/4}\rho_{L}(L-l)\label{descrep}
\end{equation}
for $d^*$ sufficiently large. Since on $ \Bbb{  B}_{y,k+3,k^{ +}, k^{ ++} }^{j', j'' }$ we have 
that $\gamma \left(l\right)\le\sqrt{2T_{l}^{y,t_{z}}} $ for 
$l=k+3,\ldots,k^{ +},k^{+ +}$,
it follows that
\be 
 \sqrt{ 2 T_{l}^{y,k, \bar \psi_{k,N_{k,a }} }   }\geq \sqrt{ 2T_{l}^{y,t_{z}}- .02\,\, l^{1/4}\rho_{L}(L-l) }\geq  \rho_{L}(L-l)+.9 l^{1/4},
\label{was21a}
\ee
which completes the proof of 
\eqref{bigassump.12}.
  
Note next that 
the $T_{l}^{y,k, \bar \psi_{k,N_{k,a }} }$ in  
\eqref{bigassump.12}
are measurable on 
$\sum_{i=1}^{ N_{k,a }}\de_{\bar Y^{\psi_{k,i}} }$. Therefore,
Assumption \ref{bigassump.1} is satisfied and we can apply
Lemma \ref{lem-decoup} (together with the fact that
for $l=k+3,\ldots, k^{ +}$ we
have $T_{l}^{y,k, \bar \th_{k,N_{k,a }}}=T_{l}^{y,k,N_{k,a }}$) 
to deduce that
\begin{eqnarray}
&&\hspace{.3 in}  \Pbm(   \WW_{y,k}(N_{k,a} )\bigcap  
 \Bbb{  B}_{y,k+3,k^{ +}, k^{ ++} }^{j', j'' }\bigcap \HH_{k,a}   \bigcap 
\corJ{\Iyzhpk} ) \label{was10q}\\
&&\leq c\,\sum_{\{s'\in \Z_{+}:  |s'-(\rho_{L}(L-k^{ +})+j')|\leq 2M_{0}\log  k\}} \,\sum_{\{s''\in \Z_{+}:  |s''-(\rho_{L}(L-k^{+ +})+j'')|\leq 2M_{0}\log  k\}} 
\nn\\
&&  \Pbm \(\rho_{L}(L-l)\le \sqrt{ 2 T_{l}^{y,k,N_{k,a }}   }\mbox{ for }l=k+3,\ldots, k^{ +};\right.\nn\\
&&\hspace{.2 in}\left. \sqrt{2T_{k^{ +}}^{y,k, N_{k,a }}}\in I_{s'};\, \sqrt{2T_{k^{ ++}}^{y,k, N_{k,a }}}\in I_{s''}\)  \times \Pbm( \corJ{\Iyzhpk})+e^{-\sqrt{L}}\Pbm(\corJ{\Iyzhpk}).\nn
\end{eqnarray}
Using    Lemma \ref{prop:BarrierSecGWPropWas},
%in Appendix I,  
we obtain  
 \begin{eqnarray}
 &&  \Pbm \(\rho_{L}(L-l)\le \sqrt{ 2 T_{l}^{y,k, N_{k,a } }   }\mbox{ for }
 l=k+3,\ldots, k^{ +};\right. \label{was21.f}\\
&&\hspace{1.5 in}\left.  \sqrt{2T_{k^{ +}}^{y,k, N_{k,a }}}\in
 I_{\rho_{L}(L-k^{ +})+\bar j'}\)
\nn\\
% &&\leq        \frac{c\,a}{(k^{ +}-k   )^{ 3/2}}\bar j'e^{ -2(k^{ +}-k   )-2 a  +2\bar j'-(\bar j'-a)^{ 2}/( 4(k^{ +}-k)    )}  \nn\\
 &&\leq        \frac{cae^{-2a}}{(k^{ +}-k   )^{ 3/2}}\bar j'e^{ -2(k^{ +}-k   )  +2\bar j'-(\bar j'-a)^{ 2}/( 4(k^{ +}-k)    )}.  \nn
 \end{eqnarray}
 It follows from \cite[Proposition 1.4]{BRZ} that 
 \begin{eqnarray}
 && \sup_{v\in  I_{\rho_{L}(L-k^{ +})+\bar j'}}\Pbm \(\sqrt{2T_{k^{ ++}}^{y,k, N_{k,a }}}\in
 I_{\rho_{L}(L-k^{ ++})+\bar j''}\,\Big |\,  \sqrt{2T_{k^{ +}}^{y,k, N_{k,a }}}=
v\)
 \label{was21}\\
 && \leq c  \frac{1}{(k^{ ++}-k^{+}   )^{ 1/2}}e^{-2 (k^{ ++}-k^{+}  )     }e^{-2( \bar j'-\bar j'') }e^{-(\bar j'-\bar j'')^{2}/2(k^{ ++}-k^{+}  )}. \nonumber
 \end{eqnarray}

Noting that $k^{ ++}-k^{+}=k^{ +}-k $
 we can bound the sum in (\ref{was10q}) by
 \begin{eqnarray}
 && c  ae^{-2a}e^{-2 (k^{ ++}-k  )     }  \sum_{\stackrel{\{\bar j': |\bar j'-j'|\leq 2M_{0}\log    k\}}{\{\bar j'': |\bar j''-j''|\leq 2M_{0}\log    k\}}}
 \label{was22}\\
 &&   \hspace{.5 in} \bar j'\,\frac{e^{  -(\bar j'-a)^{ 2}/(4( k^{ +}-k )   )}}{(k^{ +}-k   ) }     e^{2 \bar j'' } \, \frac{e^{-(\bar j'-\bar j'')^{2}/2(k^{ +}-k  )}}{(k^{ +}-k   ) }.\nonumber
 \end{eqnarray}
  This is the term $F_{k,a,j'}$ in \eqref{basicveb.1}.
  The last term in (\ref{was10q}) gives the second term
  in the right side of \eqref{basicveb.1}. 
\end{proof}

\begin{proof}[Proof of Lemma \ref{lem-decx}.] The proof is very similar to that of  Lemma \ref{lem-basicveb}. We again take   $\UU= \Bbb{  B}_{y,k+3,k^{ +}, k^{ ++} }^{j', j'' }$ but now take $\mathcal{V}=
\LL'_{k, m, p,a}$ and replace $\log k$ by $m$ in (\ref{eq:LBTruncatedSumz8}) and all corresponding expressions. The important consequence for us is that now in (\ref{descrep}) we will only have $M_{0}
\rho_{L}(L-l)m\leq .01 l^{1/4}\rho_{L}(L-l) $ for $l\geq cm^{4}$, so that we have to drop the barrier from $k+3 $ to $cm^{4}<<k^{+}$.  In consequence, in (\ref{was21.f}) we need to replace the factor $a$ by $a+m^{2}$, compare the proofs of Lemmas \ref{prop:BarrierSecGWPropWas} and 
\ref{prop:BarrierSecGWPropk}.
 \end{proof}

\begin{proof}[Proof of Lemma \ref{lem-decxb}.] The proof is very similar to that of  Lemmas
\ref{lem-basicveb} and
 \ref{lem-decx}, except that now we drop the barrier in (\ref{was10q}) and instead of (\ref{was21.f}) we use the analog of (\ref{was21}):
\begin{eqnarray}
 && \Pbm \(  \sqrt{2T_{k^{ +}}^{y,k, N_{k,a }}}\in
 I_{\rho_{L}(L-k^{ +})+\bar j'} \)
 \label{was21z}\\
 && \leq c  \frac{1}{(k^{ +}-k   )^{ 1/2}}e^{-2 (k^{ +}-k  )     }e^{-2(a-\bar j') }e^{-(\bar j'-a)^{2}/2(k^{ +}-k  )}. \nonumber
 \end{eqnarray}

 \end{proof}

\section{Continuity estimates}\label{sec-continest}

This section is devoted to the statement and proof of a general continuity
result for excursion counts, Lemma \ref{lem-cont1}.
We begin by introducing notation.

  Throughout the section we let
  $0<a<b<1$ be fixed constants and let
$0<r<R<\tilde{r}<\tilde{R}$ with
\begin{equation}
  \label{eq-ratiobounds}
  {h(r)}/{h(R)},{h(\tilde{r})}/{h(\tilde{R})}, {h(R)}/{h(\tilde r)}\in [a,b].
\end{equation}

Let $\npp$ denote a fixed point on the sphere, for instance the ``south pole'', which
is identified with $0\in R^2$ when using isothermal coordinates, see
Section \ref{sec-isot}. Let $\mu_{0,h(\tilde{r})}$ denote 
the uniform probability measure on $\partial B_{d}(\npp,h(\tilde{r}))$. 
Let  $Y_{1},Y_{2},\ldots$ denote a collection of
excursions $\partial B_{d}(\npp,h(\tilde{R}))\to\partial B_{d}(\npp,h(\tilde{r}))$, and let
$\mathcal{I}_{1},\ldots,\mathcal{I}_{n}$ be a collection of
  rotationally invariant subsets of excursions $\partial B_{d}(\npp,h(\tilde{R}))\to\partial B_{d}(\npp,h(\tilde{r}))$
 for which the densities
\begin{equation}
\inf_{i=1,\ldots, n}\frac{\mathbb{P}^{w}\left(\BMS_{\hit_{\partial B_{d}\left(\npp,h(\tilde{r})\right)}}\in du\,|\, \BMS_{\cdot\wedge \hit_{\partial B_{d}\left(\npp,h(\tilde{r})\right)}}\in \mathcal{I}_{i}\right)}{\mu_{0,h(\tilde{r})}\left(du\right)}\ge a_0>0,\label{eq: conditioned density bounded below}
\end{equation}
for any (hence all, by rotational invariance) $w\in \partial B_{d}(\npp,h(\tilde{R}))$. Note that $a_0$ is determined by $a,b$ and hence fixed throughout this section.
With $a,b$ fixed, the results in the remainder of this
  section are uniform in choices of $r,R,\tilde{r},\tilde{R}$ satisfying
  \eqref{eq-ratiobounds}
  and
  \eqref{eq: conditioned density bounded below}. 
Recall the notation $\trav{n}{0}{\tilde R}{\tilde r}{x}{R}{r}$, see Definition
\ref{def-trav}, for  the number of
	traversals  
	$\partial B_{d}(x,h(R))\to\partial B_{d}(x,h(r))$
	during $n$ excursions $\partial B_{d}(0,h(\tilde{r}))\to
	\partial B_{d}(0,h(\tilde{R}))$. The following
lemma is the main result of this section.
\bl\label{lem-cont1}
Fix $0<a<b<1$.
  Let $r<R<\tilde r<\tilde R$ so that
  \eqref{eq: conditioned density bounded below}
  and \eqref{eq-ratiobounds} hold.
Let $x,y$ be such that $B_{d}(x,h(R)),B_{d}(y,h(R))\subset B_{d}(\npp,h(\tilde{r}))$.
Then for any $C_{0}<\ff$ there exist small $c_{0},q_{0} >0$,
depending on $a,b,C_0$ only,
such that if $q=d\left(x,y\right)/R\leq q_{0}$,
and $\theta\leq c_{0}\sqrt{(n-1)}$, then
\begin{eqnarray}
&&\Pbm\left(\left| \trav{n}{\npp}{\tilde R}{\tilde r}{x}{R}{r}
        - \trav{n}{\npp}{\tilde R}{\tilde r}{y}{R}{r}
        \right|\ge\theta\sqrt{  q(n-1)}
                \,\,\,\Bigg |
        \,\,Y_{m}\in\mathcal{I}_{m},\,m=1,\ldots,n
\right)\nonumber \\&&\quad \le  \exp\left(-C_{0}\theta^{2}\sqrt{q}\right).\label{cb.1}
\end{eqnarray}
\el
\begin{proof}
The proof  of Lemma \ref{lem-cont1} will involve several steps.
We begin by restating the lemma in terms of
certain traversals counts.
Let $\wt D_{0,i}$ and $\wt R_{0,i}$ be the successive arrivals to 
$\partial B_{d}(\npp,h(\tilde{R}))$
and $\partial B_{d}(\npp,h(\tilde{r}))$. That is, 
$\wt R_{0,1}=\hit_{\partial B_{d}(\npp,h(\tilde{r}))}$, and for $i\geq 1$,
\begin{equation}
\wt D_{0,i}=\hit_{\partial B_{d}(\npp,h(\tilde{R}))}
\circ\th_{\wt R_{0,i}}+\wt R_{0,i}\label{not.1}
\end{equation}
and
\begin{equation}
\wt R_{0,i+1}=\hit_{\partial B_{d}(\npp,h(\tilde{r}))}\circ\th_{\wt D_{0,i}}+\wt D_{0,i},\label{not.2}
\end{equation}
where $\theta_t$ denotes time shift by $t$.
Thus
\[
\wt R_{0,1}<\wt D_{0,1}<\wt R_{0,2}<\wt D_{0,2}<\ldots.
\]
 Let
$A_{x,i}$ and $B_{y,i}$ denote the number of traversals from 
$\partial B_{d}(x,h(R))$
to $\partial B_{d}(x,h(r))$,  respectively
$\partial B_{d}(y,h(R))$ to $\partial B_{d}(y,h(r))$,
during the $i$'th excursion from $\partial B_{d}(\npp,h(\tilde{r}))$ to
$\partial B_{d}(\npp,h(\tilde{R}))$ (i.e. between time $\wt R_{0,i}$ and $\wt D_{0,i}$).
We have that
\begin{equation}
\travshort{n}{\npp}{\tilde{R}}{\tilde{r}}{x}{R}{r} = \sum_{i=1}^{n} A_{x,i}, \hspace{.2 in}   \travshort{n}{\npp}{\tilde{R}}{\tilde{r}}{y}{R}{r}= \sum_{i=1}^{n}B_{y,i}. \label{1.5tot}
\end{equation}
Hence Lemma \ref{lem-cont1} is equivalent to the statement that
there exist small $c_{0},q_{0}>0$, depending on $a,b,C_0$ only,
such that if $q=d\left(x,y\right)/R\leq q_{0}$,
and $\theta\leq c_{0}\sqrt{(n-1)}$,
\bea\label{1.5t}
&&\mathbb{P}\left(\left|\sum_{i=1}^{n}\left(A_{x,i}-B_{y,i}\right)\right|\ge\theta\sqrt{(n-1)q}\,\,\,\Bigg |\,\,Y_{m}\in\mathcal{I}_{m},\,m=1,\ldots,n\right)\\
&&\quad \le\exp\left(-C_{0}\theta^{2}\sqrt{q}\right). \nonumber
\eea

Before proving \eqref{1.5t},
we develop some necessary material.
We assume that
$q_{1}=q_1(a,b)$  is
sufficiently small so that  with $q=d(x,y)/R\leq q_1$,
\begin{equation}
  \label{eq-inclusion1}
  B_{d}\left(x,h(r)\right)\subset B_{d}\left(y,h(R)\right),
  \quad
B_{d}\left(y,h(r)\right)\subset B_{d}\left(x,h(R)\right).
\end{equation}

\bl
\label{lem:c3}
There exists $q_1=q_1(a,b)<1/2$ and $c_{2}<\ff$ so that,
for all $q<q_1$ and all $k\geq 1$,
\[
  \sup_{u\in\partial B_{d}\left(\npp,h(\tilde{r})\right)}\mathbb{P}^{u}\left(A_{x,1}-B_{y,1}\ge k\right)\le c_{2}qe^{-2k}.
\]
\el

\begin{proof}
Let $R_{x,1}, R_{x,2},\ldots$ be the successive hitting
times of
$\partial B_{d}\left(x,h(r)\right)$ after departure times $ D_{x,1},D_{x,2},\ldots$
from $\partial B_{d}\left(x,h(R)\right)$. That is,
set
$D_{x,1}=\hit_{\partial B_{d}\left(x,h(R)\right)}$, and for $i\geq 1$,
\begin{equation}
R_{x,i}=\hit_{\partial B_{d}\left(x,h(r)\right)}\circ\th_{D_{x,i}}+D_{x,i}\label{not.1a}
\end{equation}
and
\begin{equation}
D_{x,i+1}=\hit_{\partial B_{d}\left(x,h(R)\right)}\circ\th_{R_{x,i}}+R_{x,i}.\label{not.2a}
\end{equation}
Thus
\[
D_{x,1}<R_{x,1}<D_{x,2}<R_{x,2}<\ldots.
\]
Let $B_{y,1}\left(j\right)$ denote the number
of traversals from $\partial B_{d}\left(y,h(R)\right)$ to $\partial B_{d}\left(y,h(r)\right)$  up till $\hit_{\partial B_{d}\left(\npp,h(\tilde{R})\right)}$ that
take place between $D_{x,j}$ and $D_{x,j+1}$. Then
\[
B_{y,1}\ge\sum_{j\ge1}B_{y,1}\left(j\right),
\]
so
  \begin{eqnarray}
    \label{eq:bla}
A_{x,1}-B_{y,1}&\le& A_{x,1}-\sum_{j\ge1}B_{y,1}\left(j\right)=\sum_{ j: D_{x,j+1}<\hit_{\partial B_{d}\left( \npp,h(\tilde{R}) \right)}}
\left(1-B_{y,1}\left(j\right)\right) \nonumber\\
&\leq &\sum_{ j: D_{x,j+1}<\hit_{\partial B_{d}\left(\npp,h(\tilde{R}) \right)}}     1_{\left\{ B_{y,1}\left(j\right)=0\right\} }.
\end{eqnarray}
Since
\[\{ \hit_{\partial B_{d}\left(y,h(R)\right)}\circ \th_{D_{x,j}}< R_{x,j}\}\cap
  \{ \hit_{\partial B_{d}\left(y,h(r)\right)}\circ \th_{R_{x,j}}< D_{x,j+1}\}\subseteq \{ B_{y,1}\left(j\right)\neq 0  \}\,\]
we have by taking complements that
\[\{ B_{y,1}\left(j\right)= 0  \}\subseteq \{R_{x,j}< \hit_{\partial B_{d}\left(y,h(R)\right)}\circ  \th_{D_{x,j}}\}\cup
  \{ D_{x,j+1}<\hit_{\partial B_{d}\left(y,h(r)\right)}\circ \th_{R_{x,j}}\}.\]
Hence,
\begin{eqnarray}
\label{eq:preblax}
\sup_{u\in\partial B_{d}\left(x,h(R)\right)}
\!\!\!\!\!\!
\mathbb{P}^{u}\left(B_{y,1}\left(j\right)=0\right)
&\le&
\!\!\!\!\!\!
\!\!\!\!\!\!
\sup_{u\in\partial B_{d}\left(x,h(R)\right)}\mathbb{P}^{u}\left(\hit_{\partial B_{d}\left(x,h(r)\right)}<\hit_{\partial B_{d}\left(y,h(R)\right)}\right) \nn\\
&&
\hspace{-0.5in}+ \sup_{v\in\partial B_{d}\left(x,h(r)\right)}\mathbb{P}^{v}\left(\hit_{\partial B_{d}\left(x,h(R)\right)}<\hit_{\partial B_{d}\left(y,h(r)\right)}\right).
\end{eqnarray}
Further,
with $d=d(x,y)$,
we have  that
$B_{d}\left(y,h(R)\right)
\subseteq B_{d}\left(x,h(R)+d\right)$
and therefore,
$\hit_{\partial B_{d}\left(y,h(R)\right)}<
\hit_{\partial B_{d}\left(x,h(R)+d\right)}$ for any
path starting at 
$u\in \partial B_d\left(x,h(R)\right) \cap  B_d\left(y,h(R)\right)$.
On the other hand, for
$u\in \partial
B_d\left(x,h(R)\right) \cap B_d\left(y,h(R)\right)^c$ we have that
$\hit_{\partial B_{d}\left(y,h(R)\right)}<
\hit_{\partial B_{d}\left(x,h(r)\right)}$ by \eqref{eq-inclusion1}.
All in all, we obtain
\begin{eqnarray}
\label{eq-inclusion3}
&&
\sup_{u\in\partial B_{d}\left(x,h(R)\right)}
\mathbb{P}^{u}\left(\hit_{\partial B_{d}\left(x,h(r)\right)}
<\hit_{\partial B_{d}\left(y,h(R)\right)}\right)
\nonumber
\\
&&\quad \leq
\sup_{u\in\partial B_{d}\left(x,h(R)\right)}
\mathbb{P}^{u}\left(\hit_{\partial B_{d}\left(x,h(r)\right)}<
\hit_{\partial B_{d}\left(x,h(R)+d\right)}\right)\,.
\end{eqnarray}
Similarly, since
$B_{d}\left(x,h(r)-d\right)\subseteq B_{d}\left(y,h(r)\right)$,
we have
  \begin{eqnarray}
\label{eq-inclusion4}
&&
\sup_{v\in\partial B_{d}\left(x,h(r)\right)}
\mathbb{P}^{v}\left(\hit_{\partial B_{d}\left(x,h(R)\right)}
<\hit_{\partial B_{d}\left(y,h(r)\right)}\right)
\nonumber
\\
&&\quad \leq
\sup_{u\in\partial B_{d}\left(x,h(r)\right)}\mathbb{P}^{u}
\left(\hit_{\partial B_{d}\left(x,h(R)\right)}<\hit_{\partial B_{d}\left(x,h(r)-d
\right)}\right)\,.
\end{eqnarray}
By (\ref{isc.5}) 
and \eqref{eq-ratiobounds}, we have that the right
sides of \eqref{eq-inclusion3} and
\eqref{eq-inclusion4} are bounded by $c_1q$, for some $c_1=c_1(a,b)$
(which is fixed in what follows). Combining
this with \eqref{eq:preblax}, it follows that
\be
\label{eq:blax}
\sup_{u\in\partial B_{d}\left(x,h(R)\right)}
\mathbb{P}^{u}\left(B_{y,1}\left(j\right)=0\right)\leq c_1q\,.
\ee

Therefore, if we set
\[
  1-p_{1}:=\sup_{\stackrel{r<R<\tilde r<\tilde R\;
    \text{satisfying \eqref{eq-ratiobounds}}}
  {x,y: q(x,y)\leq q_1}}\;
  \sup_{u\in\partial B_{d}\left(x,h(R)\right)}
  \mathbb{P}^{u }\left(R_{x,1}<\hit_{\partial B_{d}
  \left( \npp,h(\tilde{R})\right)}\right) <1,
\]
then the sum (\ref{eq:bla}) can be stochastically dominated by
$\chi:=\sum_{j=1}^{G-1}I_{j},$
where $G,I_{i}$ are independent, $G$ is geometric with success probability
$p_{1}=p_1(a,b)>0$, and $I_{i}$ are
Bernoulli with success probability $c_{1}q$.
Using that $\Pbm(G-1=k)=(1-p_{1})^{k}p_{1}\leq e^{-kp_{1}}$ we have
that 
\[E(e^{\lambda \chi})=
\frac{p_1}{1-(1-p_1)(1-c_1q+c_1qe^\lambda)}\leq 4\]
if one assumes that $\lambda>0$ and $c_1qe^\lambda= p_1/4$ 
(which requires that $q_1\leq p_1/(4c_1)$). Using Chebycheff's inequality
one then obtains that 
\[P(\sum_{j=1}^{G-1} I_j\geq k)\leq 4 e^{-\lambda k}= 4\left(\frac{4c_1 q}{p_1}\right)^k\leq c_{2}qe^{-2k},\]
with $c_{2}=16c_1 e^{2}/p_1$ if one assumes that $q<q_1(a,b)$ with $4c_1 q_1/p_1<e^{-2}$.
\end{proof}

\begin{corollary}
\label{cor:HigherMomentBound}
For some $c_{4}<\ff$, with quantifiers as in Lemma
\ref{lem:c3},
if $J$ is a geometric random variable with success parameter $p_{2}>0$,
independent of $\{A_{x,i}-B_{y,i}\}$, then for
 $\lambda\le  p_{2}/2  $,
\bea
&&\sup_{u\in\partial B_{d}\left( \npp,h(\tilde{r})\right)}\mathbb{E}^{u}\left(\exp\left(\lambda\sum_{i=1}^{J-1}\left|\left(A_{x,i}-B_{y,i}\right)\right|\right)\right)\leq e^{c_{4}q\lambda /p_2  }.\label{basicexp1}
\eea
\end{corollary}

Note the linear in $\lambda$ behavior of the right
  side in \eqref{basicexp1}.
This is essentially due to mean of $|A_{x,i}-B_{y,i}|$. We will later
need to improve on this and obtain a quadratic in $\lambda$ behavior
when considering the same variables without absolute values, see Lemma
\ref{basicexp2}.
\begin{proof}
  We assume throughout that $q<1/2$.
  We begin by a moment computation. Let $m\geq 1$. Then,
by Lemma \ref{lem:c3},
  \begin{eqnarray}
\label{eq:m-th moment}
    \mathbb{E}^{u}\left(\left|A_{x,1}-B_{y,1}\right|^{m}\right)
& \le &\!\!\!\!
\sum_{k\ge1}k^{m}\left( \mathbb{P}^{u}\left(A_{x,1}-B_{y,1}\ge k\right)+\mathbb{P}^u\left(B_{y,1}-A_{x,1}\ge k\right)\right) \nonumber\\
& \le& 2c_{2} \sum_{k\ge1}k^{m}qe^{ -2k}
 \leq 2c_{2}qm!\sum_{k\ge1}e^{ -k}
= c_{3}q m!
 \end{eqnarray}
 For $m=0$ we trivially
  bound the left hand side of \eqref{eq:m-th moment} by $1$.

To prove (\ref{basicexp1}) we adapt the proof of \cite[(8.27)]{BK}\footnote{As pointed out by the referee, there is a typo in the latter;  in the bottom of page 538 and top of page 539, all sums of the form $\sum_{i_1,i_2,\ldots: \sum i_j=k}$ are missing the multiplicative factor $k!/\prod_j i_j!$. With these extra factors, the derivation in \cite{BK} gives the result claimed there.}.
For any $m\ge1$,
\bea
&&\\
&&
\mathbb{E}^{u}\left(\left(\sum_{i=1}^{J -1}\left|\left(A_{x,i}-B_{y,i}\right)\right|\right)^{m}\right)\label{eq:moment of absolute value}\nonumber\\
&&=\sum_{m_{1},m_{2},\ldots:\sum m_{j}=m}\frac{m!}{\prod_j m_j!}\mathbb{E}^{u}\left(1_{\left\{   J -1\geq \sup\left\{ j:m_{j}\ne0\right\}\right\} }\prod_{j=1}^{\infty}
\left|\left(A_{x,j}-B_{y,j}\right)\right|^{m_{j}}\right),\nn
\eea
  so that using the Markov property and  (\ref{eq:m-th moment}) we get that
  the last expression is bounded above by
\bea
\label{jmom.1}
&& m! \sum_{m_{1},m_{2},\ldots:\sum m_{j}=m}\left(1-p_{2}\right)^{\sup\left\{ j:m_{j}\ne0\right\} -1}\prod_{j=1}^{\infty}\left(c_{3} q \right)^{1_{\{m_{j}>0\}}}\\
&&\le c_{3} q m!
 \sum_{m_{1},m_{2},\ldots:\sum m_{j}=m}\left(1-p_{2}\right)^{\sup\left\{ j:m_{j}\ne0\right\} -1}
 %\prod_{j=1}^{\infty}m_{j}!
=c_{3} q   \frac{m!}{p_{2}^{m}}.\nn
\eea
where the last equality follows because
\begin{eqnarray*}
&&m!\sum_{m_{1},m_{2},\ldots:\sum m_{j}=m}
\left(1-p_{2}\right)^{\sup\left\{ j:m_{j}\ne0\right\} -1}\\
  &&=\sum_{m_{1},m_{2},\ldots:\sum m_{j}=m}\frac{m!}{\prod_{j=1}^\infty m_j!}\left(1-p_{2}\right)^{\sup\left\{ j:m_{j}\ne0\right\} -1}
\prod_{j=1}^{\infty}m_{j}!
\end{eqnarray*}
is the $m-th$ moment of a  geometric sum of standard exponentials,
which is itself exponential with mean $p_{2}^{-1}$, and for the inequality we assumed that $q<1/c_{3}$. It follows from (\ref{jmom.1}) that
\be\mathbb{E}^{u}\left(\exp\left(\lambda\sum_{i=1}^{J-1}\left|\left(A_{x,i}-B_{y,i}\right)\right|\right)\right)  \le  
1+c_{3}q\sum_{m\ge 1}\left(\frac{\lambda}{p_{2}}\right)^{m}
  \le  1+c_{4}q\left(\frac{\lambda}{p_{2}}\right),\label{jmom.1r}
\ee
for $\lambda\le  p_{2}/2 $, which proves (\ref{basicexp1}).
 \end{proof}

We can now return to the proof of \eqref{1.5t}.
Consider the excursions $\wt X^{i}=\BMS_{\left(\wt R_{0,i}+\cdot\right)\wedge \wt D_{0,i}}$
of the Brownian motion $\BMS_{\cdot}$ from $\partial B_{d}\left( \npp,h(\tilde{r})\right)$
to $\partial B_{d}\left( \npp,h(\tilde{R})\right)$. Note that $Y_{i}=\BMS_{\left(\wt D_{0,i}+\cdot\right)\wedge \wt R_{0,i+1}}$. The variables $A_{x,i},B_{y,i}$
are measurable with respect to $\BMS_{\left(\wt R_{0,i}+\cdot\right)\wedge \wt D_{0,i}}$.
The excursions $\BMS_{\left(\wt R_{0,i}+\cdot\right)\wedge \wt D_{0,i}}$ are dependent
through the starting and ending points of successive excursions. As
in \cite[Section 8]{BK}  we will construct renewal times that
give some independence.

Let $\mu_{x,r} $ denote the uniform measure on $\partial B_{d}\left(x,r\right)$.
For Brownian motion on $\mathbb{S}^{2}$ it is clear from symmetry that  for any $x$ and $ r< R$,
\begin{equation}
\Pbm^{\mu_{x,R} }\left(\BMS_{\hit_{\partial B_{d}\left(x,r\right)}}\in dw\right)=\mu_{x,r} \left(dw\right), \label{4.1}
\end{equation}
\[
\Pbm^{\mu_{x,r} }\left(\BMS_{\hit_{\partial B_{d}\left(x,R\right)}}\in dw\right)=\mu_{x,R} \left(dw\right).
\]
 In addition, it follows from the rotation invariance of $ \mathcal{I}_{i}$ that
 \begin{equation}
 \mathbb{P}^{\mu_{0,h(\tilde{R})} }\left(\BMS_{\hit_{\partial B_{d}\left( \npp,h(\tilde{r})\right)}}\in dw\,|\, \BMS_{\cdot\wedge \hit_{\partial B_{d}\left( \npp,h(\tilde{r})\right)}}\in \mathcal{I}_{i}\right)=\mu_{0,h(\tilde{r})}. \label{4.1i}
 \end{equation}
 \corJ{This reflects the symmetry of the sphere,
and is the main reason why we work with $M=\S^2$.}

 \corJ{ 
We postpone the proof of the following  lemma to later in this section.} 
\bl\label{lem-equalmean}
On $\S^2$  we have
\begin{equation}
\mathbb{E}^{\mu_{0,h(\tilde{r})}}\left(\left(A_{ x,1}-B_{y,1 }\right)\ \right)=0.\label{4.2}
\end{equation}
\el

Let
\[
p_{3}=\inf_{\stackrel{u\in\partial B_{d}\left( \npp,h(\tilde{R})\right),w\in\partial B_{d}\left( \npp,h(\tilde{r})\right)}{i=1,\ldots,n}}\frac{\mathbb{P}^{u}\left(\BMS_{\hit_{\partial B_{d}\left( \npp,h(\tilde{r})\right)}}\in dw\,|\, \BMS_{\cdot\wedge \hit_{\partial B_{d}\left( \npp,h(\tilde{r})\right)}}\in \mathcal{I}_{i}\right)}{\mu_{0,h(\tilde{r})}\left(dw\right)}.
\]
By our assumption (\ref{eq: conditioned density bounded below}) we have
that
\be
p_{3}\geq a_0>0.\label{check1}
\ee
 For $u\in\partial B_{d}\left( \npp,h(\tilde{R})\right)$ define the measure,
\[
\nu^{i}_{u}\left(dw\right)=\frac{\mathbb{P}^{u}\left(\BMS_{\hit_{\partial B_{d}
\left( \npp,h(\tilde{r})\right)}}\in dw\,|\, \BMS_{\cdot\wedge
  \hit_{\partial B_{d}\left( \npp,h(\tilde{r})\right)}}\in \mathcal{I}_{i}\right)-p_{3}\mu_{0,h(\tilde{r})}\left(dw\right)}{1-p_{3}}.
\]
(We assume that $p_3<1$, otherwise we simply take $\nu^i_u=\mu_{0,h(\tilde{r})}$.)
By the definition of $p_{3}$ we have that $\nu^{i}_{u}\ge0$ and by construction
$\nu^{i}_{u}$ is a probability measure on $\partial B_{d}\left( \npp,h(\tilde{r})\right)$. Furthermore, by   (\ref{4.1i}), when $\nu^{i}_{u}$ is
averaged over $u$ distributed as $\mu_{0,h(\tilde{R})}$ we
recover $\mu_{0,h(\tilde{r})}$:
\bea
&&
\nu^{i}_{\mu_{0,h(\tilde{R})}}\left(dw\right)\label{eq:StayStationary}\\&&=\frac{\mathbb{P}^{\mu_{0,h(\tilde{R})}}\left(\BMS_{\hit_{\partial B_{d}\left( \npp,h(\tilde{r})\right)}}\in dw\,|\, \BMS_{\cdot\wedge \hit_{\partial B_{d}\left( \npp,h(\tilde{r})\right)}}\in \mathcal{I}_{i}\right)-p_{3}\mu_{0,h(\tilde{r})}\left(dw\right)}{1-p_{3}}\nn\\
&&\overset{(\ref{4.1i})}{=}\frac{\mu_{0,h(\tilde{r})}\left(dw\right)-p_{3}\mu_{0,h(\tilde{r})}\left(dw\right)}{1-p_{3}}=\mu_{0,h(\tilde{r})}\left(dw\right).\nn
\eea
Now construct a sequence $X_{\cdot}^{1},X_{\cdot}^{2},\ldots$ of excursions
from $\partial B_{d}\left( \npp,h(\tilde{r})\right)$ to $\partial B_{d}\left( \npp,h(\tilde{R})\right)$
as follows: Let $X_{\cdot}^{1}=\BMS_{\cdot\wedge \hit_{\partial B_{d}\left( \npp,h(\tilde{R})\right)}}\mbox{ under }\mathbb{P}$, and
let $I_{2}, I_{3},\ldots,$ be i.i.d. Bernoulli random variables with success
probability $p_{3}$, independent of the Brownian motion $X$. Then, 
\bea
\label{arrz}
&&\mbox{If }I_{2}=1,\mbox{ let }X_{\cdot}^{2}=\BMS_{\cdot\wedge \hit_{\partial B_{d}\left( \npp,h(\tilde{R})\right)}}\mbox{ under }\mathbb{P}^{\mu_{0,h(\tilde{r})}}.\nonumber\\
&&\\
&&\mbox{If }I_{2}=0,\mbox{ let }X_{\cdot}^{2}=\BMS_{\cdot\wedge \hit_{\partial B_{d}\left( \npp,h(\tilde{R})\right)}}
\mbox{ under }\mathbb{P}^{\nu^{1}_{X_{\infty}^{1}}}.
\nonumber
\eea
Here we have used the abbreviation $X_{\infty}^{1}=\BMS_{ \hit_{\partial B_{d}\left( \npp,h(\tilde{R})\right)}}$, which comes  from the definition of $X_{\cdot}^{1}$.
We iterate this construction to get $X_{\cdot}^{3},X_{\cdot}^{4},\ldots.$
It follows as in \cite[Lemma 8.5]{BK} that
\be
\left(X_{\cdot}^{i}\right)_{i\ge1}\overset{\mbox{law}}{=}\left(\BMS_{\left(\wt R_{0,i}+\cdot\right)\wedge \wt D_{0,i}}\right)_{i\ge1}\mbox{ under }\mathbb{P}\(\cdot \,\Big |\,Y_{i}\in \mathcal{I}_{i},\,i=1,\ldots, n\).\label{ren1}
\ee
Hence,  to bound
\[\mathbb{P}\(\sum_{i=1}^{n}\left(A_{x,i}-B_{y,i}\right)\ge\theta\sqrt{(n-1)q} \,\Bigg|\,Y_{i}\in \mathcal{I}_{i},\,i=1,\ldots, n\)
\]
we may instead bound
\[
\mathbb{P}\left(\sum_{i=1}^{n}\left(A_{x,1}-B_{y,1}\right)\left(X_{\cdot}^{i}\right)\ge\theta\sqrt{(n-1)q}\right).
\]

Consider the renewal times
\[
J_{0}=1\mbox{ and }J_{i+1}=\inf\left\{ j>J_{i}:I_{j}=1\right\} .
\]
We have that
\be
\left(X_{\cdot}^{J_{i}},X_{\cdot}^{J_{i}+1},\ldots,X_{\cdot}^{J_{i+1}-1}\right),i\ge 0,\label{ren1p}
\ee
are an independent sequence of vectors of excursions, whose lengths are distributed
as a geometric random variable on $\left\{ 1,2,\ldots\right\} $ with
parameter $p_{3}$. Furthermore, the sequences (\ref{ren1p}) are identically distributed for $i\geq 1$.

Note that $X_{\cdot}^{J_{1}}=\BMS_{\cdot\wedge \hit_{\partial B_{d}\left( \npp,h(\tilde{R})\right)}}\mbox{ under }\mathbb{P}^{\mu_{0,h(\tilde{r})}}$, so  that $X_{\infty}^{J_{1}}$ has distribution $\mu_{0,h(\tilde{R})}$.
Hence by (\ref{eq:StayStationary}) we have $\mathbb{P}^{\nu^{1}_{X_{\infty}^{J_{1}}}}=\mathbb{P}^{\mu_{0,h(\tilde{r})}}$. Thus, whether $I_{J_{1}+1}=0$ or $I_{J_{1}+1}=1$, we 
 have that $X_{\cdot}^{J_{1}+1}=\BMS_{\cdot\wedge \hit_{\partial B_{d}\left( \npp,h(\tilde{R})\right)}}\mbox{ under }\mathbb{P}^{\mu_{0,h(\tilde{r})}}$, and this will continue for all $X_{\cdot}^{i}$, $i\geq J_{1}$. In particular, it follows from (\ref{4.2}) that
\begin{equation}
\mathbb{E} \left(\sum_{i=J_{1}}^{J_{2}-1}\left(A_{x,1}-B_{y,1}\right)\left(X_{\cdot}^{i}\right)\right)=0.\label{4.2f}
\end{equation}

This leads to the following improvement on  (\ref{basicexp1}).
\bl
 If $\lambda\le  p_{3}/2 $ then
\begin{equation}
\mathbb{E} \left(\exp\left(\lambda\sum_{i=J_{1}}^{J_{2}-1} 
\left(A_{x,1}-B_{y,1}\right)\left(X_{\cdot}^{i}\right) \right)\right) \leq 
e^{c_{4}q\left(\frac{\lambda}{p_{3} }\right)^{2} }.\label{basicexp2}
\end{equation}
\el
\begin{proof}
Using (\ref{4.2f}) for the first moment and bounding the other moments by their absolute values we see that
\[
\begin{array}{l}
\mathbb{E} \left(\exp\left(\lambda  \sum_{i=J_{1}}^{J_{2}-1}\left(A_{x,1}-B_{y,1}\right)\left(X_{\cdot}^{i}\right)  \right)\right)\\
\le1+\sum_{m\ge2}\frac{\lambda^{m}}{m!}\mathbb{E} \left( \left(\sum_{i=J_{1}}^{J_{2}-1}\left|\left(A_{x,1}-B_{y,1}\right)\left(X_{\cdot}^{i}\right)\right| \right)^{m}\right).
\end{array}
\]
As in (\ref{jmom.1r}) this is bounded by
\[  1+c_{3}q\sum_{m\ge2}\left( \lambda/p_{3} \right)^{m} \le 1+c_{4}q\left( \lambda/p_{3} \right)^{2},
\]
for $\lambda\le  p_{3}/2 $ and (\ref{basicexp2}) follows. \end{proof}

Let $U_{n}=\sup\left\{ i>1:J_{i}\le n\right\} $, the number of renewals up till time $n$. We have the upper bound
\bea
&&
\sum_{i=1}^{n}\left(A_{x,1}-B_{y,1}\right)\left(X_{\cdot}^{i}\right)\le \sum_{i=1}^{J_{1}-1}|\left(A_{x,1}-B_{y,1}\right)\left(X_{\cdot}^{i}\right)|\label{renew2}\\
&&+\sum_{i=J_{1}}^{J_{U_{n}}-1}\left(A_{x,1}-B_{y,1}\right)\left(X_{\cdot}^{i}\right)+\sum_{i=J_{U_{n}}}^{J_{U_{n}+1}-1}|\left(A_{x,1}-B_{y,1}\right)\left(X_{\cdot}^{i}\right)|.\nn
\eea
Set $n'=n-1$. Since $U_{n}=\mbox{Bin}\left(n',p_{3}\right)$,
there exists a constant $\bar c$ (independent of all other parameters) so that
\begin{eqnarray}
  \label{eq-barc}
\Pbm\left(U_{n}\ge n'p_{3}\left(1+\delta\right)\right)
&\le& e^{-\bar c\delta^{2}n'},\\
\Pbm\left(U_{n}\le n'p_{3}\left(1-\delta\right)\right)&\le&
e^{-\bar c\delta^{2}n'}. \nonumber
\end{eqnarray}
Let
\begin{equation}
\delta=\sqrt{\frac{1}{\bar cn'}}\,
\theta\leq \frac{c_{0}}{\sqrt{\bar c}}  \leq 1/2,\label{1.8d}
\end{equation}
by the assumptions on $\th$, see the statement of Lemma
\ref{lem-cont1}, after taking $c_{0}$ sufficiently small. With $u_{1}=n'p_{3}\left(1-\delta\right)$ and $u_{2}=n'p_{3}\left(1+\delta\right)-1$
let
\begin{eqnarray}
&&\Phi=\sum_{i=1}^{J_{1}-1}|\left(A_{x,1}-B_{y,1}\right)\left(X_{\cdot}^{i}\right)| \label{renew3}\\
&&\hspace{.3 in}
 +\sum_{i=J_{1}}^{J_{u_{1}}-1}\left(A_{x,1}-B_{y,1}\right)\left(X_{\cdot}^{i}\right)+\sum_{i=J_{u_{1}}}^{J_{u_{2}}-1}|\left(A_{x,1}-B_{y,1}\right)\left(X_{\cdot}^{i}\right)|.
\nn
\end{eqnarray}
By \eqref{eq-barc}
we have
\be
\mathbb{P}\left(\sum_{i=1}^{n}\left(A_{x,1}-B_{y,1}\right)\left(X_{\cdot}^{i}\right)\ge\theta\sqrt{n'q}\right)\le\mathbb{P}\left(\Phi\ge\theta\sqrt{n'q}\right)+2\exp\left(-\theta^{2}\right).\label{renew3a}
\ee
Using the independence properties of the sequences (\ref{ren1p}) we have
\begin{eqnarray}
\mathbb{E}\(e^{\la \Phi}\)
&=& \mathbb{E}  \left(\exp\left(\lambda\sum_{i=1}^{J_{1} -1}|\left(A_{x,1}-B_{y,1}\right)\left(X_{\cdot}^{i}\right)| \right)\right)  \nn\\
&&\hspace{.4 in}\times\prod_{j=1}^{ u_{1}-1} \mathbb{E}\left(\exp\left(\lambda\sum_{i=J_{j}}^{J_{j+1} -1}\left(A_{x,1}-B_{y,1}\right)\left(X_{\cdot}^{i}\right)\right)\right) \nn  \\
&&\hspace{.6 in}\times \prod_{j=u_{1}}^{ u_{2}-1}\mathbb{E}  \left(\exp\left(\lambda\sum_{i=J_{j}}^{J_{j+1} -1} |\left(A_{x,1}-B_{y,1}\right)\left(X_{\cdot}^{i}\right)| \right)\right) .\label{renew6p}
\end{eqnarray}
It follows that for all $\lambda>0$,
\begin{eqnarray}
&&\quad \quad \mathbb{E}\(e^{\la \Phi}\)
\leq
\left(\mathbb{E} \left(\exp\left(\lambda\sum_{i=J_{1}}^{J_{2}-1} \left(A_{x,1}-B_{y,1}\right)\left(X_{\cdot}^{i}\right) \right)\right)\right)^{u_{1}-1} \\
&&
\hspace{-0.2in}\times\left(\sup_{u\in\partial B_{d}\left( \npp,h(\tilde{r})\right)}\mathbb{E}^{u} \left(\exp\left(\lambda\sum_{i=1}^{J -1} | A_{x,1}-B_{y,1} | \left(X_{\cdot}^{i}\right)\right)\right)\right)^{u_{2}-u_{1}+1}
%\nn\\
\!\!=:A_1\times A_2.\label{renew6}
\nn
\end{eqnarray}
Using  (\ref{basicexp1}), the definition of $u_1,u_2$ 
and then (\ref{1.8d}) we have
\be
\label{inst}
 A_2\leq \exp\left(c_{4}q \lambda  \left(u_{2}-u_{1}+1\right)/p_3\right)
=\exp\left(2c_{4}q \lambda n'\delta\right)=\exp\left(2c_{4}q \lambda\sqrt{n'/\bar c}\,\theta\right),
\ee
for $\lambda\le  p_{3}/2 $.
For such $\lambda$, using (\ref{basicexp2}),
\[
A_1\le\exp\left(c_{4}q (\lambda/p_3)^2u_{1}\right)\le\exp\left(c_{4}q \lambda^{2}n'/p_3\right).
\]
Thus we get that for $\lambda\le  p_{3}/2 $,
\be\mathbb{P}\left(\Phi\ge\theta\sqrt{n'q}\right)\label{renew7}
\le
\exp\left(c_{4}q \lambda^{2}n'/p_3+2c_{4}q \lambda\sqrt{n'/\bar c}\,\th-\lambda\theta\sqrt{n'q}\right).
\ee
If $\theta\le \(p_{3}/4C_{0}\)\sqrt{n' }$ we set
\[
\lambda= 2C_{0}\frac{\theta}{\sqrt{n' }},
\] and conclude that
\be
\mathbb{P}\left(\Phi\ge\theta\sqrt{n'q}\right)\label{renew73}
\leq
\exp\left(2C_0\theta^2 \sqrt{q}
\left(2c_{4}C_0 \sqrt{q}/p_3+2c_{4}\sqrt{q/\bar c}-1\right)\right), 
\ee
which together with (\ref{renew3a}) gives (\ref{1.5t}) for $q$ sufficiently small.
\end{proof}

 The next lemma is a variant of Lemma \ref{lem-cont1}, allowing one to consider larger values of $\theta$ (in particular, exhibiting
 the transition from Gaussian moderate deviations to 
 exponential large deviations). It is
 %will be
useful  in the proof of Lemma \ref{lem:4.58m}.

\bl\label{lem-cont2}  Fix $0<a<b<1$ and $0<a_0<1$.
 Let $r<R<\tilde r<\tilde R$ so that
  \eqref{eq: conditioned density bounded below}
  and \eqref{eq-ratiobounds} hold.
Let $x,y$ be such that $B_{d}\left(x,h(R)\right),B_{d}\left(y,h(R)\right)\subset B_{d}\left( \npp,h(\tilde{r})\right)$.
Then for any $C_{0}<\ff$ there exist small $c_{0},q_{0} >0$ and $c_{1}>0$,
depending on $a,b,C_0$ only,
such that if $q=d\left(x,y\right)/R\leq q_{0}$,   $c_{1}\leq \theta\leq c_{0}(n-1)$,
and $  \theta \le  ((n-1)q)^{2} $, then
\begin{align}
\nn&\Pbm\left(
\left| \trav{n}{\npp}{\tilde R}{\tilde r}{x}{R}{r}
        - \trav{n}{\npp}{\tilde R}{\tilde r}{y}{R}{r}
        \right|
\ge\theta\sqrt{  q(n-1)}\,\,\,\Bigg |\,\,Y_{m}\in\mathcal{I}_{m},\,m=1,\ldots,n\right)\\
&\qquad \le \exp\left(-C_{0}\theta \right).\label{cb2.1}
\end{align}
\el

\begin{proof} Let $n'=n-1$. We return to (\ref{1.8d}) but with $\bar c$ as in \eqref{eq-barc} we now take
\begin{equation}
\delta=\sqrt{\frac{C_{0}\theta}{\bar c n'}}\leq \sqrt{\frac{c_{0}C_{0}}{\bar c}}  \leq 1/2,\label{1.8d2}
\end{equation}
by our assumptions on $\th$, and after taking $c_{0}$ sufficiently small. Then instead of (\ref{renew3a})
we obtain
\be
\mathbb{P}\left(\sum_{i=1}^{n}\left(A_{x,1}-B_{y,1}\right)\left(X_{\cdot}^{i}\right)\ge\theta\sqrt{n'q}\right)\le\mathbb{P}\left(\Phi\ge\theta\sqrt{n'q}\right)+2\exp\left(-C_{0}\theta \right).\label{renew3a2}
\ee
With this choice of $\de$, instead of (\ref{inst}) we obtain
\[A_2\le
\exp\left(2c_{4}q \lambda n'\delta\right)=\exp\left(2c_{4}q \lambda\sqrt{n'C_0\th/\bar c} \right),\]
for $\lambda\le  p_{3}/2 $.
Thus, instead of  (\ref{renew7}) we see that
for $\lambda\le  p_{3}/2 $,
\[\mathbb{P}\left(\Phi\ge\theta\sqrt{n'q}\right)
\le\exp\left(c_{4}q \lambda^{2}n'/p_3+2c_{4}q \lambda\sqrt{C_0 n'\th/\bar c} -
\lambda\theta\sqrt{n'q}\right).\]
 If $  \theta \le   (n'q)^{2}  $   we set
$\lambda= {\theta^{1/4}} p_{3}/(2{\sqrt{n'q })}$
 and see that
 \begin{eqnarray}
 &&\hspace{.3in}
\mathbb{P}\left(\Phi\ge\theta\sqrt{n'q}\right)\label{2renew73}\\
 &&\le\exp\left(c_{4}\th^{1/2}p_3/4 +c_{4}p_{3}\sqrt{q}\th^{3/4}\sqrt{C_0/\bar c}
-\th^{5/4} p_{3}/2\right)
\le \exp\left(-C_{0}\theta \right),\nn
 \end{eqnarray}
for $q_{0}$ sufficiently small and all $\th  \geq c_{1} $ sufficiently large
\end{proof}

Before proceeding to the proof of Lemma \ref{lem-equalmean}  we state a preliminary Lemma. 
Let
\begin{equation}
\ka_{a,b}=\mathbb{E}^{u}\(\hit_{\partial B_{d}(x,b)}\)+\mathbb{E}^{v}\(\hit_{\partial B_{d}(x,a)}\),
\quad u\in \partial B_{d}(x,a), v\in \partial B_{d}(x,b).\label{kap.1}
\end{equation}
By symmetry, $\ka_{a,b}$ does not depend on $x$, $u$ or $v$.

\bl\label{lem-hitrun}
On $\S^{2}$, for all $0<a<b<\pi$ we have that
\begin{equation}
\ka_{a,b}=4\log\left(\frac{\tan (b/2)}{\tan (a/2)}\right).\label{eq: hit and return}
\end{equation}

In addition, for any $x\in \S^{2}$ and $0<b<\pi$,
\begin{equation}
\sup_{y\in \S^{2}}\mathbb{E}^{y}\(\hit_{\partial B_{d}(x,b)}\)<\ff,\label{expbound.1}
\end{equation}
and $\hit_{\partial B_{d}(x,b)}$ has an exponential tail.
\el

\begin{proof}
  %{\bf  Proof:}
By the last formula in \cite[Section 3]{B}, for $u\in\partial B_{d}(x,b)$,
\begin{equation}
\mathbb{E}^{u}\(\hit_{\partial B_{d}(x,a)}\)
=2\log\left(\frac{1-\cos (b)}{1-\cos (a)}\right).\label{br.1}
\end{equation}
This formula requires $a<b$.
If $x^{\ast}$ denotes the antipode of $x\in\S^{2}$ then by symmetry, 
for $u\in\partial B_{d}(x,a)$ and $v\in\partial B_{d}(x^*,\pi-a)$
\begin{equation}
\mathbb{E}^{u}\(\hit_{\partial B_{d}(x,b)}\)=\mathbb{E}^{v}\(\hit_{\partial B_{d}(x^{\ast},\pi-b)}\).\label{br.2}
\end{equation}
Hence using (\ref{br.1}), for $u\in\partial B_{d}(x,a)$,
\begin{equation}
\mathbb{E}^{u}\(\hit_{\partial B_{d}(x,b)}\)=2\log\left(\frac{1-\cos (\pi-a)}{1-\cos (\pi-b)}\right)=2\log\left(\frac{1+\cos (a)}{1+\cos (b)}\right).\label{br.3}
\end{equation}
Then using the half-angle formula for tangents, $\frac{ 1-\cos (u)}{1+\cos (u)}=\tan^{2}(u /2)$ we obtain, using \eqref{kap.1},
\begin{equation}
  \ka_{a,b}=2\log\left(\frac{\frac{1-\cos (b)}{1+\cos (b)}}{\frac{1-\cos (a)}{
   1+\cos (a)}}\right)
=2\log\left(\frac{\tan^{2}\(b / 2\)}{\tan^{2}\(a / 2\)}\right),\label{br.4}
\end{equation}
which gives (\ref{eq: hit and return}).

(\ref{expbound.1}) follows from (\ref{br.1}) and (\ref{br.2}). By the Kac moment formula \cite{Fitzsimmons-Pitman}  this implies that  $\hit_{\partial B_{d}(x,b)}$ has an exponential tail.
\end{proof}

\begin{proof}[Proof of Lemma \ref{lem-equalmean}]
Recall, see \eqref{not.1},
that $\wt D_{0,n}$ is the time until the $n'th $ excursion from 
$\partial B_{d}(\npp,h(\tilde{r}))$
to $\partial B_{d}(\npp,h(\tilde{R}))$.
Let $\wt D_{0,0}=0$ and note that we can write
\be
\wt D_{0,n}= \sum_{i=1}^{n}T_{i},\label{4.3a}
\ee
where
\[
T_{i}=\left(\hit_{\partial B_{d}\left( \npp,h(\tilde{R})\right)}\circ\theta_{\hit_{\partial B_{d}\left( \npp,h(\tilde{r})\right)}}+\hit_{\partial B_{d}\left( \npp,h(\tilde{r})\right)}\right)\circ\theta_{\wt D_{0,i-1}}.
\]
Using the symmetry of the sphere and the Markov property we see that $T_{2},T_{3},\ldots$ are iid 
with  $\mathbb{E}(T_{i})=\ka_{h(\tilde{r}), h(\tilde{R})}$, $i=2,3,\ldots$ by (\ref{kap.1}). Hence by the Strong Law of
Large numbers
\be
\frac{\wt D_{0,n}}{n}\to \ka_{h(\tilde{r}), h(\tilde{R})}\mbox{ a.s.}\label{4.4}
\ee
Similarly, if $D_{x, m}$ denotes  the time until the $m'th $ excursion from $\partial B_{d}\left(x,h(r)\right)$ to $\partial B_{d}\left(x,h(R)\right)$, then for any $x\in \S^{2}$
\be
\frac{ D_{x,m}}{m}\to \ka_{h(r), h(R)}\mbox{ a.s.}\label{4.4x}
\ee

Recalling Definition \ref{def-trav}, let
\[V_{x,n}=:\sum_{i=1}^{n}A_{x,i}=\trav{n}{0}{\tilde R}{\tilde r}{x}{R}{r}\]
be the number
of traversals from $\partial B_{d}(x,h(R))$ 
 to $\partial B_{d}(x,h(r))$ before $\wt D_{0,n}$.
Then,
\begin{equation}
D_{x,V_{x,n}}\leq \wt D_{0,n}\leq D_{x,V_{x,n}+1}.\label{4.8}
\end{equation}
Hence using (\ref{4.4}) and (\ref{4.4x}) we see that
\begin{equation}
\ka_{h(\tilde{r}), h(\tilde{R})}=\lim_{n\to\infty}\frac{\wt D_{0,n}}{n}=\lim_{n\to\infty}\frac{D_{x,V_{x,n}}}{V_{x,n}}\frac{V_{x,n}}{n}=\ka_{h(r), h(R)}\lim_{n\to\infty}\frac{V_{x,n}}{n}.\label{4.8a}
\end{equation}
It follows that
\begin{equation}
\lim_{n\to\infty}\frac{1}{n}\sum_{i=1}^{n}A_{x,i}=\ka_{h(\tilde{r}), h(\tilde{R})}/\ka_{h(r), h(R)}.\label{4.9}
\end{equation}
Since this holds for any $x$ with $\partial B_{d}\left(x,h(R)\right)\subset \partial B_{d}\left( \npp,h(\tilde{r})\right)$
it follows that 
\begin{equation}
\lim_{n\to\infty}\frac{1}{n}\sum_{i=1}^{n}A_{x,i}=\lim_{n\to\infty}\frac{1}{n}\sum_{i=1}^{n}B_{y,i}.\label{4.11}
\end{equation}

The increments $A_{x,i}-\mathbb{E}^{X_{\wt R_{0,i}}}\left(A_{ x,1}\right),i=1,2,\ldots$,
are orthogonal by the strong Markov property, and  
have bounded second moment by Lemma \ref{lem-hitrun}. Therefore Rajchman's strong law of large
numbers, \cite[Theorem 5.1.1]{C},
implies that
\[
\lim_{n\to\infty}\frac{1}{n}\sum_{i=1}^{n}\left(A_{x,i}-\mathbb{E}^{X_{\wt R_{0,i}}}\left(A_{ x,1}\right)\right)=0,
\]
with a similar result for $A_{x,i}$ replaced by $B_{y,i}$.
It then follows from (\ref{4.11}) that 
\[
\lim_{n\to\infty}\frac{1}{n}\sum_{i=1}^{n}\mathbb{E}^{X_{\wt R_{0,i}}}\left(A_{x,1}\right)=\lim_{n\to\infty}\frac{1}{n}\sum_{i=1}^{n}\mathbb{E}^{X_{\wt R_{0,i}}}\left(B_{y,1}\right),
\]
and by the Strong Law of Large Numbers for a general
state space Markov Chain applied to the chain $X_{\wt R_{0,i}}$ in $\partial B_{d}\left( \npp,h(\tilde{r})\right)$
\[
\lim_{n\to\infty}\frac{1}{n}\sum_{i=1}^{n}\mathbb{E}^{X_{\wt R_{0,i}}}\left(A_{x,1}\right)=\mathbb{E}^{\mu_{0,h(\tilde{r})}}\left(A_{x,1}\right),
\]
and
\[
\lim_{n\to\infty}\frac{1}{n}\sum_{i=1}^{n}\mathbb{E}^{X_{\wt R_{0,i}}}\left(B_{y,1}\right)=\mathbb{E}^{\mu_{0,h(\tilde{r})}}\left(B_{y,1}\right).
\]
This proves (\ref{4.2}).
\end{proof}

We turn now to the proof of Lemma \ref{lem:4.58m}.
\begin{proof}[Proof of Lemma \ref{lem:4.58m}]
  This is a direct application of Lemma
  \ref{lem-cont2}, taking $\tilde R=r_{l-3}$,
  $\tilde r=r_{l-2}$, \corO{$R= \tilde  r_{l-1}^+$}, $r=\tilde r_l$
  $x=y$, $y=\tilde y$, \corJ{$n=u^2/2$} with $u\in I_{\alpha(l)+j}$,
  and $\theta=d_0 jm/2$.
%  \corO{(Note the notation convention difference between Lemma \ref{lem:4.58m} and Lemma % \ref{lem-cont2}: in the latter the superscript counts excursions
%   from small to large balls, while  in the former this is reversed. This has no effect on the % % application of the lemma, since these numbers are equal.}
  %{(\alpha_-(l)+j)}/u$.
To apply the lemma, we must verify several points.

First, by taking $ j_{0}  $ sufficiently large we will have $\th>c_{1}$.
Next we need to verify that $ \theta\leq c_{0}(n-1)$. By halving $c_{0}$ it suffices to show that $ \theta\leq c_{0}n$   which is
$ d_{0}j m/2\leq \corJ{c_{0}(\al(l)+j)^{2}/2}$. For this it is enough
to show that
$ d_{0}  m/2\leq \corJ{c_{0}(\al(l) +j)/2}$, which follows from the fact that $m\leq k-l=\log (2(\al(l) +j))$, see (\ref{conv-k}),  and taking $ j_{0}  $ sufficiently large. 

Secondly, we need to show that $  \theta \le  ((n-1) q)^{2} $.
Since we have already seen that  $ \theta\leq c_{0}(n-1)$, it suffices to show that    $(n-1) q^{2}\geq   c^{2}_{2}$ for some $c_{2}>0$, or equivalently that $\corJ{\sqrt{2n}}\,\,    q\geq   c'_{2}>0$. That  is,
$(\al(l)+j) d\left(\tilde{y},y\right)/r_{l}  \geq c'_{2}>0$.     Assume that $d\left(\tilde{y},y\right)\ge c_{3}r_{k}$
for a small $c_{3}>0$, so that, see (\ref{conv-k}),
\begin{equation}
\left(\al(l)+j\right) d\left(\tilde{y},y\right)/r_{l}\ge c_{3}\left(\al(l)+j\right)e^{-\left(k-l\right)}=c_{3}/2.\nn
\end{equation}
Recall that $\tilde y$ is the ``parent" of $y$ defined in the paragraph following (\ref{bu2}).
With our construction of  $F_{l}$,
% constructed appropriately we can indeed assume that
either  $d\left(\tilde{y},y\right)\ge c_{3}r_{k}$
for a small universal $c_{3}>0$, or $y=\tilde{y}$, in which case the corresponding
term in the sum in (\ref{eq:chaining union bound}) is zero. Also, \corO{because $\tilde{c}=q_{0}/2$, see \eqref{eq-q0c},} we will have $d\left(\tilde{y},y\right)/r_{l}\leq q_{0}$.
\end{proof}

\section{Excursion time and real time}\label{sec-etrt}
We prove in this section Theorem \ref{theo-etrt}, which is used
to control the relation between excursion counts and real time, and
compare various excursions with different centering. It was crucially
used in the proof of the upper bound.

Recall that 
for $0<a<b<\pi$,   $ \tau_{x, a,b}(n)$ is
the time needed to complete $n$ excursions in $\S^{2}$ from
$\partial B_{d}\left(x,a\right)$ to $\partial B_{d}\left(x,b\right)$, see 
\eqref{eq-tauab}. Recall $\ka_{a,b}$ from (\ref{kap.1})-(\ref{eq: hit and return}).

\bl\label{lem-clt}
For any $0<a<b<\pi$ there exists a $c = c(a,b)$ such that for $\delta \in (0,1)$ and $x\in\S^{2}$,
\begin{equation}
P\( \tau_{x, a,b}(n)\leq (1-\de)\ka_{a,b}n\)\leq e^{-c\de^{2}n} \label{clt.1}
\end{equation}
and
\begin{equation}
P\( \tau_{x, a,b}(n)\geq (1+\de)\ka_{a,b}n\)\leq e^{-c\de^{2}n}. \label{clt.1a}
\end{equation}
\el
\begin{proof}
By symmetry we can take $x=0$, and then as in the first paragraph of the proof of Lemma \ref{lem-equalmean}, $\tau_{x, a,b}(n)=\sum_{i=1}^{n}T_{i}$ where the $T_{2},T_{3},\ldots$ are iid 
with  $\mathbb{E}(T_{i})=\ka_{a,b}$, $i=2,3,\ldots$
and all $T_{i}$ have exponential tails.
\end{proof}
  
  Throughout this section we assume that $z_{0}\leq z\leq L^{1/2}\log^{2} L$. 
Recall the notation $s(z)$, see \eqref{clt.2}.
It follows from Lemma \ref{lem-clt}
with $n=s\left(z\right)$ and $\delta={d\sqrt{z}}/{(2L)}$   that for
some $d<\ff$,
\begin{equation}
P\(  s\left(z-d\sqrt{z}\right)    \ka_{a,b}   \leq  \tau_{x,  a,b} \(s\left(z \right)\)\leq   s\left(z+d\sqrt{z}\right)   \ka_{a,b}\)\geq 1-e^{-10z},\label{clt.3}
\end{equation}
 uniformly in $x\in\S^{2}$.

It follows from
(\ref{clt.3}),   that  uniformly in $x\in\S^{2}$,
\be
\Pbm(4s(z-d\sqrt{z}) \log (b / a)  \leq  
\tau_{x, h(a),h(b)}(s(z))\leq   4s(z+d\sqrt{z})\log (b /a))
\geq 1-e^{-10z}.\label{clt.3a}
\ee
This is like Theorem \ref{theo-etrt}, except it applies
to one $x$ and not to all $x \in F_L$ simultaneously. We can not derive Theorem \ref{theo-etrt} from \eqref{clt.3a}
via a union bound over all $x \in F_L$, since there are  far too many elements in $F_L$.
To reduce the number of $x$ that need to be considered we will use a chaining argument, contained in Lemma \ref{rlem:Continuity lemma} below.
Before stating it, 
we set
notation.
Throughout the argument, we fix
$r_{0}$ small,  e.g. $r_{0}<10^{-6}$.
Then we define
\[
\tilde{r}_{1}<\tilde{r}_{0}<\tilde{r}_{-1}<\tilde{r}_{-2},
\]
by
\[
\tilde{r}_{0}=r_{0}+\frac{1}{L}\mbox{ and }\tilde{r}_{1}=r_{1}-\frac{1}{L},
\]
\[
\tilde{r}_{-1}=2\tilde{r}_{0} \mbox{ and }\tilde{r}_{-2}=2\tilde{r}_{0}\times\left(\frac{\tilde{r}_{0}}{\tilde{r}_{1}}\right).
\]
It follows from (\ref{clt.3a})   that for some $d<\ff$
\begin{equation}
\Pbm( 4s(z-d\sqrt{z})     \leq \tau_{0, h(\wt r_{-1}),h(\wt r_{-2})}(s(z))\leq   4s(z+d\sqrt{z}))\geq 1-e^{-10z},\label{r3.1}
\end{equation}
and
\begin{equation}
\Pbm(4s(z-d\sqrt{z})     \leq \tau_{0,h(\wt r_{1}),h(\wt r_{0})}(s(z))\leq  
4s(z+d\sqrt{z}))\geq 1-e^{-10z}.\label{r3.2}
\end{equation}
Here we used the fact  that
\be
\log\frac{\tilde{r}_{-2}}{\tilde{r}_{-1}}=\log\frac{\tilde{r}_{0}}{\tilde{r}_{1}}=\log\frac{r_{0}+\frac{1}{L}}{r_{1}-\frac{1}{L}}=\log\left[\frac{r_{0}}{r_{1}}\left(1+O\left(\frac1L\right)\right)\right]=1+O\left(\frac1L\right),\label{r2.7g}
\ee
 so that
\be
s(z\pm d\sqrt{z})\log\frac{\tilde{r}_{0}}{\tilde{r}_{1}} =L\left( 2L-\log L+z\pm d\sqrt{z}+O(1)\right).\label{r2.7h}
\ee

Recall Definition \ref{def-trav}, and abbreviate $T_{y,\tilde{r}_{1} }^{x,\tilde{r}_{-1},n}=
\trav{n}{x}{\tilde r_{-2}}{\tilde r_{-1}}{y}{\tilde r_0}{\tilde r_1}$, the number of 
traversals from  $\partial B_{d}\left(y,h(\tilde{r}_{0})\right)\to\partial B_{d}\left(y,h(\tilde{r}_{1})\right)$
during $n$ excursions from $\partial B_{d}\left(x,h(\tilde{r}_{-1})\right)\to\partial B_{d}\left(x,h(\tilde{r}_{-2})\right)$.
 Once again, possibly enlarging $d$, it follows from (\ref{r3.1}) and (\ref{r3.2}) that
\begin{equation}
P\( s\left(z-d\sqrt{z}\right)     \leq  T_{0,\tilde{r}_{1} }^{0,\tilde{r}_{-1},s\left(z\right)} \leq    s\left(z+d\sqrt{z}\right)\)\geq 1-ce^{-5z}.\label{r3.3}
\end{equation}

Let $a_{0}=\pi^{2}/3.$
The following lemma uses a chaining argument in its proof.

\bl
\label{rlem:Continuity lemma} There exist constants $\tilde c,
  z_{0}<\ff$ such that for $L $ sufficiently  large  and 
  all $z_{0}\leq z\leq L^{1/2}\log^{2} L$, 
\bea
&&\hspace{.3 in}
\mathbb{P}\left[\exists y\in F_{\frac{3}{2}\log L}\cap B_{d}\left(0,\tilde{c}\,\,h(\wt r_{0})\right)\right.\label{req: continuity goal}\\
&&\left.\hspace{.7 in}\mbox{ s.t. }\left|T_{0,\tilde{r}_{1} }^{0,\tilde{r}_{-1},s\left(z \right)}-T_{y,\tilde{r}_{1} }^{0,\tilde{r}_{-1},s\left(z \right)}\right|\ge a_{0}\sqrt{ z}L \right]\le ce^{-4z}.\nn
\eea
\el

\begin{proof}[Proof of Lemma \ref{rlem:Continuity lemma}]
We use Lemma \ref{lem-cont1} from Section \ref{sec-continest}.
Taking $R=\tilde r_0$, $r=\tilde r_1$,
$\tilde R=\tilde r_{-2}$, $\tilde r=\tilde r_{-1}$ and $n=s(z)$,
the lemma shows that that  for any $C_{0}>0$ there
exist  small $c_{0}, q_{0}>0$ such that if   $q=d(x,y)/\wt r_{0}\le q_{0}$,
and $\theta\leq c_{0}\sqrt{(n-1)}$,
\begin{equation}
\mathbb{P}\left[\left|T_{x,\tilde{r}_{1} }^{0,\tilde{r}_{-1},s\left(z\right)}-T_{y,\tilde{r}_{1} }^{0,\tilde{r}_{-1},s\left(z \right)}\right|\ge\theta\sqrt{(n-1)q}\right]\le\exp\left(- C_{0}\theta^{2}\sqrt{q}\right),\label{neq:continuity estimate}
\end{equation}
where $n=s\left(z\right)\sim 2L^{2}$. \corO{Recalling  $\tilde{c}$ from \eqref{eq-q0c}, 
we see  that for any $x,y\in B_{d}\left(0,\tilde{c}\,\,h(\wt r_{0})\right)$,} we
have $q=d(x,y)/\wt r_{0}\le q_{0}$ and $B_{d}\left(x, h(\wt r_{0})\right), B_{d}\left(y, h(\wt r_{0})\right)\subseteq  B_{d}\left(0, h(\wt r_{-1})\right)$. Furthermore, since in our present application of 
Lemma \ref{lem-cont1} there are no $\mathcal{I}_{i}$'s, condition (\ref{eq: conditioned density bounded below}) is easy to verify: 
just use  the Poisson kernel (\ref{poisson}) and the fact that the outside of a circle centered at the south pole is the inside of a circle centered at the north pole.

As is standard in continuity estimates,  we
use a chaining, that is, we construct a tree of points that are embedded in $\S^2$,  and ``cover'' $B_{d}\left(0,\tilde{c}\,\,h(\wt r_{0})\right)$
in the sense that
such that the $k$'th level of the tree has size $256^{k}$ and the
largest distance from any point in $B_{d}\left(0,\tilde{c}\,\,h(\wt r_{0})\right)$ to a point
in the $k$'th level is at most $16^{-k}$. We further require that  the last (i.e. $k= 3\log L/ (2\log16)$) level of the tree contains
$F_{\frac{3}{2}\log L}\cap B_{d}\left(0,\tilde{c}\,\,h(\wt r_{0})\right)$. 
An explicit construction of the tree is obtained by choosing for the $k$th level an arbitray  net $\mathcal{N}_k$ of $256^k$ points 
with maximal distance $16^{-k}$, and declaring the ``parent'' of $x\in \mathcal{N}_k$ to be the element of $\mathcal{N}_{k-1}$ closest to it, with ties broken
e.g. by lexicographic order, and $\mathcal{N}_0=0$. Since the spacing of  $F_{\frac{3}{2}\log L}$ is ${1}/{L^{3/2}}$, we can easily choose the last level
to satisfy the constraint.

Using this tree, we can connect a point  $y\in F_{\frac{3}{2}\log L}$ to
$0$ by a unique geodesic in the tree. This allows us 
 to bound the
left hand side of (\ref{req: continuity goal})
by
\be
\sum_{k=0}^{\frac{\frac{3}{2}\log L}{\log16}}256^{k}\label{neq: chaining}
\sup_{\stackrel{x,y\in B_{d}\left(0,\tilde{c}\,\,h(\wt r_{0})\right),}{d(x,y)\asymp16^{-k}}}\mathbb{P}\left[\left|T_{x,\tilde{r}_{1} }^{0,\tilde{r}_{-1},s\left(z\right)}-T_{y,\tilde{r}_{1} }^{0,\tilde{r}_{-1},s\left(z\right)}\right|\ge 2\sqrt{ z}L(k+1)^{-2}\right].
\ee
Here we use that 
%the $\frac{\frac{3}{2}\log L}{\log16}-$level of
%the tree can be made to include all of contains 
%(since the spacing of $F_{\frac{3}{2}\log L}$ is $\frac{1}{L^{3/2}}$),
%and 
all $x,y$ considered in the sup are such that $d(x,y)\ge c16^{-\frac{\frac{3}{2}\log L}{\log16}}={c}/{L^{3/2}} $,
(and that $\sum_{k=0}^{\infty}(k+1)^{-2}=\pi^{2}/6$).

We now apply (\ref{neq:continuity estimate}) with
 $\th=\frac{2  \sqrt{ z}L(k+1)^{-2}}{\sqrt{(n-1)q}}$ and $d(x,y)\asymp16^{-k}$, once we verify that $\theta\leq c_{0}\sqrt{(n-1)}$. Thus we need
 \begin{equation}
2  \sqrt{ z}L \leq  c_{0}(k+1)^{2}(n-1)\sqrt{q}.\label{cek.1}
 \end{equation}
But
 \begin{equation}
c_{0} (n-1) \sqrt{q}\geq c_{0}L^{2}\sqrt{d(x,y)/\wt r_{0}}\geq \wt c_{0}L^{2}4^{-k} \label{cek.2}
 \end{equation}
 and
 \begin{equation}
(k+1)^{2}4^{-k} \geq c \log^{2}(L) 4^{-\frac{\frac{3}{2}\log L} {\log16}}=c \log^{2}(L) L^{-3/4}.\label{cek.3}
 \end{equation}
 Thus (\ref{cek.1}) will hold as long as $z\leq cL^{1/2} \log^{4}(L) $.

 Clearly we can find $c_{1}$ such that $c_{1}4^{k}(k+1)^{-4}\geq 2(k+1)$ for all $k\geq 0$.
Then, using the fact that   $n\leq 2L^{2}$, we can choose $C_{0}$ sufficiently large so that
  (\ref{neq: chaining})
is at most
\bea
\sum_{k=0}^{\infty}256^{k}e^{-C_{0}\(\frac{2  \sqrt{ z}L(k+1)^{-2}}{\sqrt{(n-1)q}}\)^{2}\sqrt{q}}&\le&\sum_{k=0}^{\infty}256^{k}e^{- c_{1}\frac{4  z(k+1)^{-4}}{2\times 4^{-k}}}\nn\\
&=&\sum_{k=0}^{\infty}256^{k}e^{-4z (k+1)}\le ce^{-4z},
\eea
as long as $z_{0}> \log 16 $. This gives the claim (\ref{req: continuity goal}).\end{proof}
We have completed all preparatory steps and can turn to our main goal.
\begin{proof}[Proof of Theorem \ref{theo-etrt}.]
For any $x\in F_{L}\cap B_{d}\left(0,\tilde{c}\,\,h(\wt r_{0})\right)$, there exists a
$y\in F_{\frac{3}{2}\log L}$ such that
\be
B_{d}\left(y,h(\wt r_{1})\right)\subset B_{d}\left(x,h(r_{1})\right)\subset B_{d}\left(x,h(r_{0})\right)\subset B_{d}\left(y,h(\wt r_{0})\right),\label{r3.10}
\ee
since the ``spacing'' of $F_{\frac{3}{2}\log L}$ is $\frac{1}{L^{3/2}}$
and
\[
\frac{1}{L^{3/2}}+h(\wt r_{1})\le h(r_{1})\mbox{ and }h(r_{0})+\frac{1}{L^{3/2}}\le h(\wt r_{0}).
\]

It follows from (\ref{r3.10}), (\ref{req: continuity goal}) and (\ref{r3.3})  that
\be
\mathbb{P}\left[  T_{x,r_{1} }^{0,\tilde{r}_{-1},s\left(z \right)} \ge s\left(z-d\sqrt{z}\right)   -a_{0}\sqrt{ z}L,\, \forall x\in F_{L}\cap B_{d}\left(0,\tilde{c}\,\,h(\wt r_{0})\right)\right]\label{r3.11}\geq 1- ce^{-4z}.\nn
\ee
Since the event  $T_{x,r_{1} }^{0,\tilde{r}_{-1},s\left(z \right)} \ge s\left(z-d\sqrt{z}\right)   -a_{0}\sqrt{ z}L$ implies that
\begin{equation}
 \tau_{x} (   s( z-(d+ a_{0})\sqrt{z}  ))    \leq     \tau_{0, h(\wt r_{-1}),h(\wt r_{-2})}\(s\left(z \right)\),\label{r3.1f}
\end{equation}
it then follows from (\ref{r3.1}) that for   $d'=2d+a_{0}$
 \begin{equation}
P\(  \tau_{x}\(s\left(z \right) \) \leq  4s\left(z+d'\sqrt{z}\right), \, \forall x\in F_{L}\cap B_{d}\left(0,\tilde{c}\,\,h(\wt r_{0})\right)\)\geq 1-ce^{-4z}.\label{r3.12}
\end{equation}
Since we can cover $\S^{2}$ by a finite number of discs of radius $\tilde{c}\,\,h(\wt r_{0})$ this implies that
\begin{equation}
P\(    \tau_{x}\(s\left(z \right) \) \leq  4s\left(z+d'\sqrt{z}\right)  ,\,\forall x\in F_{L}\)\geq 1-ce^{-4z},\label{r1.13}
\end{equation}

To get the other direction of (\ref{r1.1}) we proceed as above but taking now
$\tilde{r}_{0}=r_{0}-\frac{1}{L}$ and
$\tilde{r}_{1}=r_{1}+\frac{1}{L}$.
\end{proof}

We had to work hard, using the full force of the continuity estimates,
to obtain Theorem \ref{theo-etrt}.
If we restrict to $x,y$ very close, 
as in Theorem \ref{theo-etrtc},
the situation is much simpler.
\begin{proof}[Proof of Theorem \ref{theo-etrtc}.]
 Let $x,y$ be such that $d(   x,y)\leq \De:=h(r_{L/2})\sim e^{ -L/2}$. By (\ref{eq:blax}) the probability that an excursion from $\partial B_{d}\left(x,h_{0}\right)$ to $\partial B_{d}\left(x,h_{1}\right)$ and back does not contain an excursion from $\partial B_{d}\left(y,h_{0}\right)$ to $\partial B_{d}\left(y,h_{1}\right)$ is $O(  \De )$. Hence during $2L^{ 2}$ excursions
  from $\partial B_{d}\left(x,h_{1}\right)$ to $\partial B_{d}\left(x,h_{0}\right)$ the number of `missed opportunities' for excursions from $\partial B_{d}\left(y,h_{1}\right)$ to $\partial B_{d}\left(y,h_{0}\right)$ is bounded by $\mbox{Bin}(2L^{ 2},   c\De )$. Since this has mean $\la=cL^{ 2}\De\to 0$ as $L\to \ff$ we can use the Poisson approximation to bound
  \begin{equation}
 P(  \tau_{x}\(t_{z}\)< \tau_{y}\(t_{z}-10\))\leq c\la^{ 10} =cL^{ 20}e^{ -5L}.\nn
  \end{equation}
Since there are $ce^{ 4L}$ pairs $x,y \in F_{L}$, (\ref{r1.1cl}) follows.
\end{proof}

\section{ Proof of Corollary \ref{cor-plane}: from $\S^2$ to $\R^2$}
\label{sec-genman}

The starting point for the proof is the following lemma concerning Brownian motion on $\S^2$. We use $s$ to denote  the south pole of $\S^2$ \corJ{and $\ast$ to denote the north pole}. 
Set $\la_{L}=\rho_{L}L$, see \eqref{eq:defofts}, with $L=\log (1/\ep)$.
%, and note that $m_{e^{-L},\S^2}=\sqrt{2}\,\,\la_{L}$.
The next lemma
follows by combining  Theorem  \ref{theo-tightness} for   $B_{d}\(   s,h_{2}\)\subseteq \S^2$ with (\ref{clt.3a}).
\begin{lemma}
  Fix $t>0$.
Let  $N_L$ be the number of excursions 
between $\partial B_{d}(s,h_{1})$
and $\partial B_{d}(s,h_{0})$ needed before 
$B_{d}\(   s,t\)$ is $h_{L}$-covered. Then
$\sqrt{2N_L}-\la_{L}$ is tight.
\end{lemma}

We next use stereographic projection $\si: \S^2\corJ{\backslash \{\ast\}}\to \R^{ 2}$.
With   $X_t$  Brownian motion on the sphere,   let $W_t=\si(X_t)$. 
Because $\si$ provides a system of isothermal coordinates in $\R^2$, we have 
that $W_t$ is a time-changed \corJ{version of}
standard Brownian motion in $\R^2$.

%We take
%$h_2=h(1)$, $h_1=h(r_1)$ and $h_0=h(r_0)$.
Let $N^P_L$  be the number of excursions between $\partial B_{e}(0,r_{1})$
and $\partial B_{e}(0,r_{0})$ needed before $B_{e}(0,1)$ is $r_{L}$-covered.
Since 
balls of radius \corJ{$h_{L}$} in $B_{d}\(   s,\corJ{h(1)}\)$ are of Euclidean radius in \corJ{$[r_{L}/c,cr_{L}]$}
for some fixed $c$, we have from the lemma,  simple inclusion 
and the fact that $|\la_{L\pm n}-\la_{L}|\leq cn$ for $L\gg n$, that
$\sqrt{2N^P_L}-\la_{L}$ is tight.  

Now note that $ \sum_{k=1}^{N^P_L-1} S_k\leq 
\CC^{\ast}_{\eps,P,R} \leq  \sum_{k=1}^{N^{P}_L} S_k$ where $S_k$ 
are the times spent by planar Brownian 
motion inside $\corJ{B_e(0,R)}$ during an excursion starting
at $\partial B_e(0,r_{1})$,
visiting 
$\partial B_e(0,r_{0})$, and then returning to 
$\partial B_e(0,r_{1})$. The $S_k$ are, by rotational invariance, i.i.d.,
of mean $a_R$,
and  with exponential
tails. We immediately conclude  that
$$\sqrt{\CC^{\ast}_{\eps,P,R}}-\sqrt{a_R} \,\,\la_L/\sqrt{2}$$
is tight.

It only remains to compute $a_R$.
%, which is the expected time spent inside a disc of radius $R$ during an excursion from radius $r_1$ to 
%radius $r_0$ and back.
By Doob's stopping theorem, the time to hit radius $r_0$ when starting at
$r_1$ is 
$(r_0^2-r_1^2)/2$. On the other hand, with $\tau_1$ denoting the hitting time
of $\corJ{\partial B_e(0,r_1)}$,
$$ u(x)=\Ebm^x(\int_0^{\tau_1} {\bf 1}_{|W_t|\corJ{\leq}R} dt)$$
satisfies the boundary value problem
$$\frac12(u''(r)+\frac{u'(r)}{r})=-{\bf 1}_{r\corJ{\leq}R}, \quad u(r_1)=0, u'(R)=0,$$
with solution
\corJ{$$ u(r)=\frac{r_1^2}{2}+R^2\log \(\frac{r}{r_1}\)-\frac{r^2}{2}, \hspace{.2 in}r\leq R.$$}
In particular, summing $u(r_0)$ with $(r_0^2-r_1^2)/2$ gives 
$ a_R=R^2$, which completes the proof. 

\section{Appendix I: Barrier estimates}\label{sec:BoundaryCrossing}

Recall that 
  $\al (l)=\rho_{L}(L-l)  - l^{\ga}_{L}$, $\ga<1/2$, see \eqref{eq:defofts} and 
\eqref{eq:AlphaBarrierDefd},
where
\begin{equation}
\rho_{L}= 2-\frac{\log L}{2L},\label{18.01}
\end{equation}
and  recall that 
\begin{equation}
t_{z}=\frac{\(\rho_{L}L+z\)^{ 2}}{2}.\label{18.02t}
\end{equation} 
Recall the notation $I_{y }=[y,y+1)$. 
(Note that in 
\cite{BRZ}, $I_{y }$ is denoted $H_{y }$.) 

\bl\label{prop:BarrierSecGWPropUB}
For any   $k<L$ and all $  j  \geq 0$ and $0\leq z\leq 100k$ with $t_{z}\in \mathbb{N}$,
\begin{eqnarray}
&&\hspace{.2 in}
K_{1}\,:=\Pbm\left(\al (   l) \le  \sqrt{2T_{l}^{0,t_{z}}},\,l=1,\ldots,k-1;\,  \sqrt{2T_{k}^{0,  t_{ z }}}\in I_{\al (k)+j}\right)\label{18.20}\\
&&\nn
\quad \leq  ce^{-2k-2z-2k^{\ga}_{L}+2j} \times \left(1+z+ k^{ \ga}_{L}\right)\left(1+j   \right)e^{-\frac{\left( z +  k_{L}^{\ga} - j \right)^{2}}{4k}}.\end{eqnarray}
\el

\begin{proof}[Proof of Lemma \ref{prop:BarrierSecGWPropUB}]
 Recall that we 
 denote by 
 $T_{l}, l=0,1,2,\ldots$ the critical Galton--Watson process
with geometric distribution, and let $\Pgw_n$ denote its law when 
$T_0=n$. 
 With   $v=\rho_{L}(L-k)- k^{ \ga}_{L}$ we rewrite $K_{1} $  as
 \bea
&&
K_{1}=\Pgw_{t_{z}} \left(\rho_{L}(L-l)- l^{ \ga}_{L}\le  \sqrt{2T_{l} } \mbox{\,\ for }l=1,\ldots, k-1;\right.\label{18.22}\\ 
&&\left.\hspace{3in}\sqrt{2T_{k} }\in  I_{v+ j}  \right).\nn
\eea

We show below  that for any $0<\ga<1$, and all $1\leq l\leq k-1$
\begin{equation}
l_{L}^{\ga}\leq k_{L}^{\ga}+l_{k}^{\ga},\label{bar.6}
\end{equation}
from which it follows that
\be
\rho_{L}\left(L-l\right)-   l_{L}^{\ga}\ge\rho_{L}\left(L-k\right)- k_{L}^{\ga}+\rho_{L}\left(k-l\right)-   l_{k}^{\ga}.\nn
\ee
Hence  
\[
  K_{1}\leq \Pgw_{t_{z}}\left(v+\rho_{L}\left(k-l\right)-   l_{k}^{\ga} \le\sqrt{2T_{l}},l=1,\ldots,k-1;\sqrt{2T_{k}}\in  I_{v+  j}  \right),
\]
and by (\ref{18.02t}),  $\sqrt{2t_{z}} =\rho_{L}L+ z=\rho_{L}k+v+ z +   k^{\ga}_{L} $.

Thus using  \cite[Theorem 1.1]{BRZ},   with $a=\rho_{L}k+v, x=\sqrt{2t_{z}}$ and $b=v, y=v+ j$,
\bea
&&
K_{1}\leq K_{2}=c\frac{\left(1+ z+ k_{L}^{\ga}\right)\left(1+ j\right)}{k}
 e^{-\frac{\left(\rho_{L}k+ z +   k_{L}^{\ga} - j \right)^{2}}{2k}},\nn
\eea
We write out
\begin{eqnarray}
  &&\hspace{0.5 in}
  \left(\rho_{L}k+ z +   k_{L}^{\ga} - j \right)^{2}/2k
\label{multdef.11}\\
&& = \left(2k+ (z +   k_{L}^{\ga} - j) -k\log (L)/2L\right)^{2}/2k \nn\\
&&=2k+ 2(z +   k_{L}^{\ga} - j) -k\log (L)/L 
%\nn \\
%&&\hspace{1 in}
+ \left( (z +   k_{L}^{\ga} - j) -k\log (L)/2L\right)^{2}/2k. \nn
\end{eqnarray}
Using the concavity of the logarithm, we have $k\log (L)/L \leq \log k$, and using $(r-s)^{2}\geq r^{2}/2-s^{2}$ we see that 
\begin{equation}
\left( (z +   k_{L}^{\ga} - j) -k\log (L)/2L\right)^{2}/2k\geq  (z +   k_{L}^{\ga} - j)^{2} /4k-o_{L}(1)\label{multdef.12}
\end{equation} 

 It follows that 
 \bea
 &&
K_{2}\leq c \left(1+ z+  k_{L}^{\ga}\right)\left(1+j\right)e^{-2k-2z-2k_{L}^{\ga}+2j} e^{-\frac{\left( z +  k_{L}^{\ga} - j \right)^{2}}{4k}}.\nn
\eea

We now prove (\ref{bar.6}).
This certainly holds if $l\leq k/2$, since then $l_{L}^{\ga}=l_{k}^{\ga}=l^{\ga}$.  
If $L/2\leq l\leq k$, then (\ref{bar.6}) says that
\begin{equation}
(L-l)^{\ga}\leq  (L-k)^{\ga}+ (k-l)^{\ga},\label{bar.7u}
\end{equation}
which follows from concavity.

(\ref{bar.6}) holds if  $k\leq L/2$,  since then $l_{L}^{\ga}=l^{\ga}\leq k^{\ga}= k_{L}^{\ga}$. 
Finally, if $k/2\leq l\leq L/2<k$, then (\ref{bar.6}) says that
\begin{equation}
l^{\ga}\leq   (L-k)^{\ga}+ (k-l)^{\ga}.\label{bar.8}
\end{equation}
Since for $  l\leq L/2$ we have $l^{\ga}\leq (L-l)^{\ga }$, (\ref{bar.8}) follows from (\ref{bar.7u}).
\end{proof}

Recall  that  
\begin{equation}
\gamma\left(l\right)=\gamma\left(l,L\right)= \rho_{L}(L-l)+  l^{ 1/4}_{L},\label{114.1}
\end{equation}
%and 
%\begin{equation}
%\delta_{z}\left(l\right)=\delta\left(l,L,z\right)= \rho_{L,z}(L-l) +6 l^{ 3/4}_{L},\label{114.3}
%\end{equation}
see \eqref{14.1}.

 \bl
\label{prop:BarrierSecGWProp} For all $L $ sufficiently  large,   and all $0\leq   z\leq  10 L$ with $t_{z}\in \mathbb{N}$,  
\bea 
&&\Pbm \left(\gamma(l)\le  \sqrt{2T_{l}^{0,t_{z}}} \mbox{\,\ for }l=1,\ldots,L-1;\,T_{L}^{0,t_{z}}=0 \right]\nn\\
&&\hspace{2.5 in}  \asymp  (1+z)e^{ -2L-2z}e^{-z^{2}/4L}.\label{18.6}
\eea
 Similar estimates hold if we delete the barrier condition on some fixed finite interval.
\el

\begin{proof}[Proof of Lemma \ref{prop:BarrierSecGWProp}]
The left hand side
of \eqref{18.6}   can be written 
in terms of the critical Galton-Watson process $T_{l}, l=0,1,2,\ldots$ as
\bea
&&
\Pgw_{t_{z}}\left(\rho_{L}(L-l)+  l^{ 1/4}_{L}\le  \sqrt{2T_{l} },   \mbox{\,\ for }l=1,\ldots, L-1;\,T_{L} =0 \right).\nn
\eea
(\ref{18.6}) follows immediately from \cite[Theorem 1.1]{BRZ}, with $a=\rho_{L}L,$ $x=\rho_{L}L+z$, $ b=y=0$,
since $\{   T_{L} =0\}=\{  T_{L} \in [0,1] \}$.

\corJ{For the last statement, we simply note that following the proof of
 \cite[Lemma 2.3]{BRZ}
we can show that the analogue of \cite[Theorem 1.1]{BRZ} holds where we skip some fixed finite interval.}
\end{proof}

\bl\label{prop:BarrierSecGWPropLk}
If $(L-k) \log L/L=o_{L}(1)$ then for all $L $ sufficiently  large,   and all 
$0\leq   z \leq \log L$ with $t_{z}\in \mathbb{N}$,  
\bea
&&\Pbm \left(\rho_{L}(L-l)\le  \sqrt{2T_{l}^{0,t_{z}}} \mbox{\,\ for }l=1,\ldots, k;\,T_{L}^{0,t_{z}}=0 \right)\label{18.7}\\
&& \le  c (1+z)(L-k)^{ 1/2}e^{ -2L-2z-z^{2}/4L}.\nn
\eea 
\el
\begin{proof}[Proof of Lemma \ref{prop:BarrierSecGWPropLk}]  
  We rewrite this in terms of the critical Galton--Watson process with geometric offspring distribution, $T_{l}, l=0,1,2,\ldots$, as 
 \begin{equation}
\Pgw_{t_{z}}\left(\rho_{L}(L-l)\le  \sqrt{2T_{l} } \mbox{\,\ for }l=1,\ldots, k;\,T_{L}=0 \right).\label{18.12}
\end{equation}

We condition on $T_{k}$ as follows: let $v=\rho_{L}\left(L-k\right)$. Then 
\bea
&&
J_{1}:=\Pgw_{t_{z}}\left(\rho_{L}\left(L-l\right)\le\sqrt{2T_{l}},l=1,\ldots,k,T_{L}=0\right)\nn\\
&&\le\sum_{j=0}^{\infty}\Pgw_{t_{z}}\left(\rho_{L}\left(L-l\right)\le\sqrt{2T_{l}},l=1,\ldots,k,\sqrt{2T_{k}}\in I_{v+j}\right)\nn\\
&& \hspace{2in}\times\sup_{u \in I_{j}}\Pgw_{(   v+u)^{2}/2}\left(T_{L-k}=0\right)\nn\\
&&=\sum_{j=0}^{\infty}\Pgw_{t_{z}}\left(\rho_{L}\left(L-l\right)\le\sqrt{2T_{l}},l=1,\ldots,k,\sqrt{2T_{k}}\in I_{v+j}\right)\nn\\
&& \hspace{2in}\times\sup_{u\in I_{j}}\left(1-\frac{1}{L-k}\right)^{(v+u)^{2}/2}\nn.
\eea
 Since, by  (\ref{eq: tail bound}),
\[\Pgw_{t_{z}}\left(\sqrt{2T_{k}}\ge100L\right)\le ce^{-\left(\rho_{L}L+z-100L\right)^{2}/2L}\le e^{-100L},\]
we obtain 
\bea
&&J_1 \le\sum_{j=0}^{100L}\Pgw_{t_{z}}\left(\rho_{L}\left(L-l\right)\le\sqrt{2T_{l}},l=1,\ldots,k,\sqrt{2T_{k}}\in I_{v+j}\right)\nn\\
&&\hspace{1in}\times e^{-\frac{(   v+j)^{2}}{2(L-k)}}+e^{-100L}=:
J_{1,1}+e^{-100L}.\nn
\eea

By \cite[Theorem 1.1]{BRZ},  with $a=\rho_{L}L$, $x=\rho_{L}L+z$ recall (\ref{18.02t}), and $b=v, y=v+j$,  
\begin{eqnarray*}
J_{1,1}&\le& \sum_{j=0}^{100L}c\frac{(1+z)(1+j) }{k} \sqrt{\frac{\rho_{L}L}{k(   v+j)}}\,\, e^{-\frac{\left(\rho_{L}L+z-v-j\right)^{2}}{2k}}e^{-\frac{(   v+j)^{2}}{2(L-k)}}\\
&\le& c\sum_{j=0}^{100L}\frac{(1+z)(   1+j) }{k}\sqrt{\frac{L}{k(v+j)}}\,\, e^{-\frac{\left(\rho_{L}k+z-j\right)^{2}}{2k}}e^{-\frac{\left(\rho_{L}\left(L-k\right)+j\right)^{2}}{2(L-k)}}\\
&\leq&c(1+z)e^{-\frac{ \rho_{L}^{2}L }{2}-2z}\sum_{j=0}^{100L}   \frac{(1+j) }{k}\sqrt{\frac{L}{k(   v+j)}}\,\,e^{-\frac{(j-z)^{2}}{2k}-\frac{j^{2}}{2(L-k)}} \nn\\
&\leq &c(1+z)e^{-2L-2 z-z^{2}/2k}\sum_{j=0}^{100L}(1+j) \frac{1}{\sqrt{  L-k}}\,\,e^{-\frac{j^{2}}{2(L-k)}}e^{jz/k}.\nn
\end{eqnarray*}
But 
\begin{equation}
e^{ -z^{2}/4k}e^{-\frac{j^{2}}{4(L-k)}}e^{jz/k}\leq 1\label{multdef.13}
\end{equation}
which follows by considering separately $j\geq 4(L-k)z/k$ and $j< 4(L-k)z/k$, since in this case 
$j<z/4$ because our condition on $k$ implies that $(L-k)/k $ is tiny. It then follows that 
\begin{eqnarray*}
J_{1,1}&\le&  c(1+z)e^{-2L-2 z-z^{2}/4L}  \sum_{j=0}^{100L} (   1+j)\frac{1}{\sqrt{  L-k}}\,\,e^{-\frac{j^{2}}{4(L-k)}}\\ 
& \le  &c(1+z)  e^{-2L-2 z-z^{2}/4L}\sqrt{L-k}.
\end{eqnarray*}
\end{proof}

\bl\label{prop:BarrierSecGWPropk}
If $k\log L/L=o_{L}(1)$ and  $m=\rho_{L}(L-k)+j$, then for all $L $ sufficiently  large,   with $m^{ 2}/2\in \mathbb{N}$,  
\bea
&&\Pbm \left(\rho_{L}(L-l)\le  \sqrt{2T_{l}^{0,k,m^{ 2}/2}}\mbox{\,\ for }l=k+1,\ldots, L-1;\,T_{L}^{0,k,m^{ 2}/2}=0 \right) \nn\\
&&  \hspace{1 in}  \le  c \left(1+j\right)e^{-2\left(L-k\right)-2j -\frac{j^{2}}{4(L-k)}},\label{18.8a}
\eea
and if $j\leq \eta L$ and we skip the barrier from $k+1$ to $k+s$, with 
$s\leq \eta'  \log L$, for some $\eta, \eta'<\ff $, then
\bea
&&\Pbm \left(\rho_{L}(L-l)\le  \sqrt{2T_{l}^{0,k,m^{ 2}/2}}\mbox{\,\ for }l=k+s,\ldots, L-1;\,T_{L}^{0,k,m^{ 2}/2}=0 \right) \nn\\
&&  \hspace{1 in}  \le c\,(   1+j+\sqrt{s}) e^{-2\left(L-k\right)-2j -\frac{j^{2}}{4(L-k)}}.\label{18.8as}
\eea

%In addition,  uniformly in $\gamma(k)\le m\le \rho_{L}(L-k)+j$, 
%\be
%\Pbm \left(\rho_{L}(L-l)\le  \sqrt{2T_{l}^{y,k,m^{ 2}}}\mbox{\,\ for }l=k+1,\ldots, L-1\,\Big |\,T_{L}^%{y,k,m^{ 2}}=0 \right)\le c{j \over L}.\label{18.8}
%\ee
\el

\begin{proof}  We first turn to the probability in  (\ref{18.8a}).
By the Markov property,  the probability in  question can be rewritten  in terms of the critical Galton-Watson process $T_{l}, l=0,1,2,\ldots$
as 
\[
M_{1}:=\Pgw_{m^{2}/2}\left(\rho_{L}\left((L-k)-l\right)\le\sqrt{2T_{l}},l=1,\ldots,L-k-1;\, T_{L-k}=0\right).
\]
By \cite[Theorem 1.1]{BRZ},  with  $a=\rho_{L}\left(L-k\right), x=m$ and $b=y=0$,  
\be
M_{1}\leq c\frac{\left(1+j\right) }{L-k} \,\, e^{-\frac{\left(\rho_{L}\left(L-k\right)+ j\right)^{2}}{2\left(L-k\right)}}.\label{18.3w}
\ee
We have
\bea
&&
e^{-\frac{\left(\rho_{L}\left(L-k\right)+ j\right)^{2}}{2\left(L-k\right)}}=e^{-\frac{\left(\rho_{L}\left(L-k\right)\right)^{2}}{2(L-k)}- j\rho_{L}-\frac{j^{2}}{2(L-k)}}\nn\\
&&\le ce^{-\frac{\left(2\left(L-k\right)-\frac{\log L}{2}\frac{L-k}{L} \right)^{2}}{2\left(L-k\right)}-2j-\frac{j^{2}}{ 3(L-k)}}\leq Ce^{\log L \frac{L-k}{L}}e^{-2\left(L-k\right)-2j -\frac{j^{2}}{3(L-k)}}.\nn
\eea
Using   the concavity of the logarithm, %our condition that  $\frac{k}{L}\log L=o_{L}(1)$,
%and $j\leq k$ ,
we get that 
\[ M_{1}\leq c \left(1+j\right)e^{-2\left(L-k\right)-2j -\frac{j^{2}}{3(L-k)}},\]
which gives (\ref{18.8a}).  

%Since   we have that   $\Pgw_{v^{2}/2}\left( T_{L-k}=0\right)\asymp e^{-\frac{\left(\rho_{L}\left(L-k%\right)+ j\right)^{2}}{2\left(L-k\right)}}$, while $L-k\asymp L$ and $k_{L}^{ 1/4}\leq j\leq 10L$, (\ref%{18.8}) follows from  (\ref{18.3w}).

For (\ref{18.8as}) with $k'=k+
s$, $v= \rho_{L}(L-k')$ we bound
\bea
&&\Pbm \left(\rho_{L}(L-l)\le  \sqrt{2T_{l}^{0,k,m^{ 2}/2}}\mbox{\,\ for }l=k+s,\ldots, L-1;\,T_{L}^{0,k,m^{ 2}/2}=0 \right) \nn\\
&& \hspace{.5 in} \leq \sum_{j'=0}^{\ff}\Pgw_{m^{2}/2}\left( \sqrt{2T_{s}}\in I_{v+j'}\right)M'_{1}\label{180.1}
\eea
where
\begin{align*}
&M'_{1}
=\\
&\sup_{u\in I_{v+j'}}\Pgw_{u^{2}/2}\left(\rho_{L}\left((L-k')-l\right)\le\sqrt{2T_{l}},l=1,\ldots,L-k'-1;\, T_{L-k'}=0\right).
\end{align*}
By \cite[Proposition 1.4, Remark 2.2]{BRZ}
\bea
&&\Pgw_{m^{2}/2}\left( \sqrt{2T_{s}}\in I_{v+j'}\right)\leq  c\sqrt{\frac{m}{(v+j')s}}\,\,e^{-\(m-(v+j')\)^{2}/2s}\label{180.2}\\
&&\leq c\sqrt{\frac{1}{s}}\,\,e^{-\(\rho_{L}s+(j-j')\)^{2}/2s}\leq c\sqrt{\frac{1}{s}}\,\,e^{-2s-2(j-j')- (j-j')^{2}/3s}\nn
\eea
and as in the first part of this proof
\begin{equation}
M'_{1}\leq c \left(1+j'\right)e^{-2\left(L-k'\right)-2j' -\frac{j'^{2}}{3(L-k')}}.\label{180.3}
\end{equation}
Thus
\begin{eqnarray}
&&
\sum_{j'=0}^{\ff}\Pgw_{m^{2}/2}\left( \sqrt{2T_{s}}\in I_{v+j'}\right)M'_{1}\label{180.4}\\
&&\leq   c e^{-2\left(L-k\right)-2j  }\sqrt{\frac{1}{s}}\,\,\sum_{j'=0}^{\ff}  \left(1+j'\right)e^{  -\frac{j'^{2}}{3(L-k')}}    e^{  -(j'-j)^{2}/3s}.  \nonumber
\end{eqnarray}
Considering separately the cases where $j'\leq 1.1 j$ and $j'> 1.1 j$ we see that for   $L $ sufficiently  large,
\begin{equation}
e^{  -\frac{j'^{2}}{3(L-k')}}    e^{  -(j'-j)^{2}/3s}\leq ce^{  -\frac{j^{2}}{4(L-k)}}    e^{  -(j'-j)^{2}/4s},\label{180.5}
\end{equation}
and (\ref{18.8as}) follows.
% follows by combining the estimates in  \cite[Lemma 7.2]{BK} for the Brownian bridge    with the %proof of \cite[Theorem 1.1]{BRZ} which was based on estimates for the Brownian bridge.
\end{proof}

\bl\label{prop:BarrierSecGWPropWas}
If  $v=\rho_{L}(L-k)+u$
then for $L$ large  with $v^{2}/2\in \mathbb{N}$,  and
$1\le k,\wt k, u,j\leq  L^{3/4}$,
\begin{align}
&\label{18.30}\\
&
\Pbm\left(\rho_{L}(L-l)\le\sqrt{2T^{0,k, v^{2}/2}_{l}},l=k+1,\ldots,k+\wt k;\sqrt{2T^{0,k, v^{2}/2}_{k+\wt k}}\in
I_{\rho_{L}(L-k-\wt k)+j}
\right)\nn\\
&\nn
%\hspace{1 in}
\le C\frac{\left(1+u\right)\left(1+j\right)}{\wt k^{3/2}}e^{-2\wt k-2\left(u-j\right)-\frac{(u-j)^{2}}{4\wt k}}.
\end{align}
In addition, if we skip the barrier from $k+1$ to $k+3$, 
\begin{align}
\label{18.30y} &\\
&
\Pbm\left(\rho_{L}(L-l)\le\sqrt{2T^{0,k, v^{2}/2}_{l}},l=k+3,\ldots,k+\wt k;\sqrt{2T^{0,k, v^{2}/2}_{k+\wt k}}\in
I_{\rho_{L}(L-k-\wt k)+j}
\right)\nn\\
&\le C\frac{\left(1+u \right)\left(1+j\right)}{\wt k^{3/2}}e^{-2\wt k-2\left(u-j\right)-\frac{(u-j)^{2}}{4\wt k}}.\nn
\end{align}
\el

\begin{proof}  
By the Markov property,   the probability in  question can be rewritten  in terms of the critical Galton-Watson process $T_{l}, l=0,1,2,\ldots$
as 
\bea
&&
 V_{1}:=\Pgw_{ v^{2}/2}\left(\rho_{L}\left(L-k-l\right)\le\sqrt{2T_{l}},l= 1,\ldots, \wt k;\right.\\
&&\left.\hspace{2 in}\sqrt{2T_{ \wt k}}\in
I_{\rho_{L}\left(L-k-\wt k\right)+ j } \nn
\right).\nn
\eea 
This is a linear barrier  of length $\wt k$. At the start of the barrier it is at distance
$ u$ from the starting point, and at the end it is at distance $ j$
from the end  point. Therefore by  \cite[Theorem 1.1]{BRZ} we
have that the probability is at most
\bea
&&
c\frac{\left(1+ u\right)\left(1+ j\right)}{\wt k^{3/2}}\sqrt{
  \frac{\rho_{L}\left(L-k\right)+ u} {\rho_{L}\left(L-k-\wt k\right)+ j}} \nn\\
&&\hspace{2 in}e^{-\frac{\left(\rho_{L}\left(L-k\right)+ u-\left(\rho_{L}\left(L-k-\wt k\right)+ j\right)\right)^{2}}{2\wt k}}\nn\\
&&\leq c\frac{\left(1+u\right)\left(1+j\right)}{\wt k^{3/2}}e^{-\frac{\left(\rho_{L}\wt k+ u-  j \right)^{2}}{2\wt k}},\nn
\eea
 where we have bounded the square root using the fact that    $k,\wt k, u,j<L^{3/4}$ so that the ratio inside the square root is $\asymp 2L/2L=1$.
Using  $k,\wt k, u,j<L^{3/4}$ again we see that  
\bea
e^{-\frac{\left(\rho_{L}\wt k+ u-  j \right)^{2}}{2\wt k}}&\leq &ce^{-\frac{\left(\rho_{L}\wt k\right)^{2}}{2\wt k}}e^{ -2\left(u- j\right) }e^{-\frac{\left(u- j \right)^{2}}{2\wt k}}\nn\\
&=&ce^{-\frac{\left(2\wt k-\frac{\log L}{2}\frac{\wt k}{L} \right)^{2}}{2\wt k}}e^{ -2\left(u- j \right) }e^{-\frac{\left(u- j \right)^{2}}{2\wt k}}\nn\\
&\leq &Ce^{- 2\wt k }e^{ -2\left(u- j \right) }e^{-\frac{\left(u- j \right)^{2}}{2\wt k}}.\nn
\eea
 This gives (\ref{18.30}).

  (\ref{18.30y}) follows as in the proof of the previous Lemma.
\end{proof}

\section{Appendix II: Conditional excursion probabilities}\label{probends}

The following is a modification of \cite{DPRZthick} adapted to our situation.
%Recall
%\begin{equation}
%k^{+}=k+\lceil10^{10}  \log L\rceil,\label{eq:DefOfKPlust}
%\end{equation}
%and abbreviate $k'=k^{+}$.  
Fix   $ k \geq 1$,   $k' \geq k+10$ and let ${\mathcal G}_k^y$ to be the $\sigma$-algebra generated by the excursions
from $\partial B_{d}(y,h_{k-1})$ to $\partial B_{d}(y,h_{k})$ (and if we start outside $\partial B_{d}(y,h_{k-1})$ we include the initial excursion to $\partial B_{d}(y,h_{k-1})$). 
Note that ${\mathcal G}_k^y$ includes the information on the end points of the excursions from  $\partial B_{d}(y,h_{k})$ to  $\partial B_{d}(y,h_{k-1})$.

Recall that $\trav{n}{y}{k-1}{k}{y}{l-1}{l}$ is the number of excursions    from
	$\partial B_{d}(y,h_{l-1})\to\partial B_{d}(y,h_{l})$ during the first $n$
 excursions    from
	$\partial B_{d}(y,h_{k })\to\partial B_{d}(y,h_{k-1})$, see \eqref{eq-calgary}.

\begin{lemma}\label{recursionends}
%There exists $C_0 <\infty$ such that
For any $L-2k\geq k '>k+10\geq 11$ and $n>1$,
let $\mathcal{A}_{k'}$ denote an event, measurable on the excursions of the Brownian motion 
inside $ B_{d}(y,h_{k'})$  during the first $n$
 excursions    from
	$\partial B_{d}(y,h_{k })\to\partial B_{d}(y,h_{k-1})$. Then,
\bea
&& \Pbm(\mathcal{A}_{k'} \,
| \, \trav{n}{y}{k-1}{k}{y}{k'-1}{k'} = m_{k'}, \, {\mathcal G}_k^y)
\nn\\
&&=
\left(1+ O\( (k'-k+1)\frac{h_{k'-1}}{h_{k-1}}\)\right)^{m_{k'}} \Pbm(\mathcal{A}_{k'}\,
| \, \trav{n}{y}{k-1}{k}{y}{k'-1}{k'} = m_{k'})  \,.
\label{m1.2e0}
\eea
In particular, 
%
% For any $2 \leq k \leq L-2k'$,
%if $n \leq L^{4}$ then 
for  all $m_l ;l=k',\ldots L$, and all $y \in \S^2$, 
\bea
&& \Pbm(\trav{n}{y}{k-1}{k}{y}{l-1}{l} = m_l ;l=k'+1,\ldots L \,
|\, \trav{n}{y}{k-1}{k}{y}{k'-1}{k'} = m_{k'}, \, {\mathcal G}_k^y) \nonumber \\ 
&&=
\left(1+ O\( (k'-k+1)\frac{h_{k'-1}}{h_{k-1}}\)\right)^{m_{k'}} 
\label{m1.2e}\\
&&\hspace{.4 in}\Pbm(\trav{n}{y}{k-1}{k}{y}{l-1}{l} = m_l ;l=k'+1,\ldots L \,
|\, \trav{n}{y}{k-1}{k}{y}{k'-1}{k'} = m_{k'} ).\nn
\eea
\end{lemma}
\noindent
In Lemma \ref{recursionends}
we have  $| O\( (k'-k+1)\frac{h_{k'-1}}{h_{k-1}}\)|\leq 50\( (k'-k+1)\frac{h_{k'-1}}{h_{k-1}}\)$. 

The key to the proof of Lemma \ref{recursionends} is the following Lemma. 

\begin{lemma}\label{lemprobends}
For $ k \geq 0$,   $k' \geq k+10$
and a Brownian path $X_\cdot$ 
starting at $z\in \partial B_{d}(y,h_{k'-1})$,
let $\bar{\tau}=\inf \{t>0\, :\,X_t\in \partial B_{d}(y,h_{k})\}$ and let 
$F$ denote an event measurable with respect to the path of the Brownian motion
inside $B_d(y,h_{k'})$, prior to $\bar \tau$.
Then,  
uniformly in $z\in \partial B_{d}(y,h_{k'-1})$, $v\in \partial B_{d}(y,h_{k})$ and $y$, 
\begin{equation}
\label{m1.5eua}
  \left|\frac{\Ebm^z (F\,\big |\,X_{\bar{\tau}}=v)}{\Ebm^z 
      (F   )}-1\right|
\leq  4.8 (k'-k)\frac{h_{k'-1}}{h_k}.
\end{equation}

 If
$Z_{k'}=\trav{1}{y}{k}{k'-1}{y}{k'-1}{k'}$ 
denotes the number of excursions of
%O
the path  from 
$\partial B_{d}(y,h_{k'-1})\to\partial B_{d}(y,h_{k'})$, prior to
$\bar{\tau}=\inf \{t>0\, :\,X_t\in \partial B_{d}(y,h_{k})\}$, then,  
uniformly in $z\in \partial B_{d}(y,h_{k'-1})$, $v\in \partial B_{d}(y,h_{k})$, $j$ and $y$, 
\begin{equation}
\label{m1.5eu}
  \left|\frac{\Ebm^z (F ;Z_{k'}=j \,\big |\,X_{\bar{\tau}}=v)}{\Ebm^z 
      (F;Z_{k'}=j )}-1\right|
\leq  4.8 (k'-k)\frac{h_{k'-1}}{h_k},
\end{equation}
and
\begin{equation}
\label{m1.5euc}
  \left|\frac{\Ebm^z (F \,\big |\,Z_{k'}=j, X_{\bar{\tau}}=v)}{\Ebm^z 
      (F\,\big |\,Z_{k'}=j )}-1\right|
\leq  9.8 (k'-k)\frac{h_{k'-1}}{h_k}.
\end{equation}
Further, uniformly in $x\in \partial B_{d}(y,h_{k+1})$, $v\in \partial B_{d}(y,h_{k})$, $j\geq 1$, and $y$,  
\begin{equation}
\label{m1.5eucc}
  \left|\frac{\Ebm^x (F \,\big |\,Z_{k'}=j, X_{\bar{\tau}}=v)}{\Ebm^{\la_{k'-1}} 
      (F\,\big |\,Z_{k'}=j )}-1\right|
\leq  28.2 (k'-k)\frac{h_{k'-1}}{h_k}.
\end{equation}
where $\la_{k'-1}$ denotes uniform measure on $\partial B_{d}(y,h_{k'-1})$.
\end{lemma}

In words, conditioning by the endpoint of the excursion at level $k$ has only minor influence on the probability
of events involving pieces of excursions inside the ball of radius $h_{k'}$.

\noindent{\bf Proof of Lemma \ref{lemprobends}:}
We first prove \eqref{m1.5eu}. \eqref{m1.5eua}
follows from it by multiplying both sides by
$\Ebm^z (F;
    Z_{k'}=j )$ and summing over $j$.
    
    Toward the proof of \eqref{m1.5eu}, it is enough to prove that
    \begin{equation}
\Big| \Ebm^z (F ;Z_{k'}=j \,\big |\,X_{\bar{\tau}}=v)-
\Ebm^z (F; Z_{k'}=j )\Big|\leq 
     4.8 (k'-k)\frac{h_{k'-1}}{h_k}\Ebm^z (F;
    Z_{k'}=j ).\label{92.1}
\end{equation}
    Without loss of generality we can take $y=0$.  
Fixing $k \geq 0$ and
$z \in \partial B_{d}(0, h_{k'-1})$,  
consider  a positive continuous function $g$ on $\partial B_{d}(0,h_{k})$.
Let $\bar{\tau} = \inf \{ t:\, w_t \in \partial B_{d}(0,h_{k})\}$,
$\tau_0=0$ and for $i=0,1,\ldots$ define
\begin{eqnarray*}
\tau_{2i+1} & = & \inf \{ t \geq \tau_{2i}  :\; X_t \in \partial
B_{d}(0,h_{k'}) \cup \partial B_{d}(0,h_{k}) \} \\
\tau_{2i+2} & = &
 \inf \{ t \ge \tau_{2i+1} :\; X_t \in \partial B_{d}(0,h_{k'-1})\}\,.
\end{eqnarray*}
 Then, by the strong Markov property at $\tau_{2j}$,
\beaa
&&\Ebm^z (g(X_{\bar\tau}); F
    ;Z_{k'}=j  )
=
\Ebm^z \big[ \E^{X_{\tau_{2j}}} (g(X_{\bar\tau}); Z_{k'}=0) ;F,
Z_{k'}=j, {\bar\tau} \geq \tau_{2j} \big]
\,.
\eeaa
and, substituting $g=1$,
\bea
&&
\Pbm^z \(F,Z_{k'}=j \)\nn
=
\Ebm^z \left[ \Pbm^{X_{\tau_{2j}}} ( Z_{k'}=0) ;F,
Z_{k'}=j, {\bar\tau} \geq \tau_{2j} \right]\nn
 \;.
\eea
Consequently,
\beaa
&&\Pbm^z \(F, Z_{k'}=j \)
\inf_{x\in \partial B_{d}(0,h_{k'-1})}
\frac{\E^{x} \( g(X_{\bar\tau});\, Z_{k'}=0 \)} {\Pbm^x \(Z_{k'}=0\)}\\
&&\leq \Ebm^z [ g(X_{\bar\tau}); F, Z_{k'}=j ]\\ 
&& \leq 
\Pbm^z \(F, Z_{k'}=j \)
\sup_{x\in \partial B_{d}(0,h_{k'-1})}
\frac{\E^{x} \( g(X_{\bar\tau});\, Z_{k'}=0 \)} {\Pbm^x \(Z_{k'}=0\)}
\,,
\eeaa
and, using again the strong Markov property at time $\tau_{2}$,
\begin{eqnarray*}
\E^{x} \( g(X_{\bar\tau});\, Z_{k'}=0 \) &=&
\E^{x} \( g(X_{\bar\tau}) \) -
\E^{x} \( \E^{X_{\tau_2}} (g(X_{\bar\tau}));\, Z_{k'} \geq 1 \) \\
&\leq& \E^{x} \( g(X_{\bar\tau}) \) -
\Pbm^{x} (Z_{k'} \geq 1) \inf_{y\in \partial B_{d}(0,h_{k'-1})} \E^{y} (g(X_{\bar\tau})),
\end{eqnarray*}
with the reversed inequality if the $\inf$ is replaced by $\sup$.
Writing $p=\Pbm^{x} (Z_{k'} \geq 1) =1-1/(k'-k)$
whenever $x\in \partial B_{d}(0,h_{k'-1})$, c.f.
(\ref{isc.5}), it thus
follows that
\begin{eqnarray}
\label{eq10.1}
&& \Ebm^z \left[  g(X_{\bar\tau}); F, Z_{k'}=j  \right] \\
&\leq&
\Pbm^z\(F, Z_{k'}=j  \)
\nonumber \\ & &\cdot 
\E^z\( g(X_{\bar\tau})\)
(1-p)^{-1}
\Big\{ \frac{\sup_{x\in \partial B_{d}(0,h_{k'-1})} E^x( g(X_{\bar\tau}))} 
{\inf_{y\in \partial B_{d}(0,h_{k'-1})} E^y( g(X_{\bar\tau}))} -p\Big\}, 
\nn
\end{eqnarray}
with the reversed inequality if the $\inf$ and $\sup$ are interchanged.

Let $p_{B_{d}(   0, h_{k})}(   z,x)$ denote 
 the Poisson kernel  for $B_{d}(   0, h_{k})\subseteq \S^{ 2}$,  see
 \eqref{poisson}. Then,
$$
\E^{z'} g (X_{\bar{\tau} })
= \int_{\partial B_{d}(0,h_{k})}  p_{B_{d}(0,h_{k})}(   z',u) g(u) du.
$$
Therefore,
%AAA/OOO
we get the Harnack inequality
\bea
 \frac{\sup_{x\in \partial B_{d}(0,h_{k'-1})} \E^x( g(X_{\bar\tau}))}
{\inf_{y\in \partial B_{d}(0,h_{k'-1})} \E^y( g(X_{\bar\tau}))}& \leq &
\frac{\max_{x\in \partial B_{d}(0,h_{k'-1}),u\in \partial B_{d}(0,h_{k })}p_{B_{d}(0,h_{k})}( x,u)}
{\min_{y\in \partial B_{d}(0,h_{k'-1}),u\in \partial B_{d}(0,h_{k })} p_{B_{d}(0,h_{k})}( y,u)}\nn \\
&=&\frac{\max_{y\in \partial B_{d}(0,h_{k'-1}),u\in \partial B_{d}(0,h_{k })}\sin^{ 2} ( d(  u,y)/2)}
{\min_{x\in \partial B_{d}(0,h_{k'-1}),u\in \partial B_{d}(0,h_{k })} \sin^{ 2} ( d(  u,x)/2)}
\nn \\
&\leq&
\frac{\sin^{ 2}((h_k+h_{k'-1})/2) }{\sin^{ 2}((h_k-h_{k'-1})/2) }=:B_{k,k'-1}^2.\label{eq10.2}
\eea
Writing $\al=(h_k-h_{k'-1})/2$, $\bb=h_{k'-1}$  we have
\be
B_{k,k'-1}=\frac{\sin (\al+\bb) }{\sin (\al) }=\cos(\bb)+\frac{\sin ( \bb)\cos (\al) }{\sin (\al) }.
%=1+O\(\frac{ h_{k'-1} }{h_k }\).
\label{eq10.2a}
\ee
Using the bounds  $\sin (\alpha)\geq .9\alpha $, $\sin (\bb)\leq \bb $, $\cos (\al) , \cos \beta\leq 1$, which are valid for  $0<\alpha,\beta <0.01$, we obtain that 
%the right side of \eqref{eq10.2a}
$B_{k,k'-1}$ is bounded above by 
\begin{equation}
1+ \bb/.9\al=1+2 h_{k'-1}/.9(h_k-h_{k'-1}).\label{jr,b}
\end{equation}
Since $k' \geq k+10$ we have that  $.99 h_k \leq h_k-h_{k'-1}$, so that 
(\ref{jr,b}) is bounded by 
$1+2.3h_{k'-1}/h_k$, and therefore $B_{k,k'-1}^2$ 
%the right side of
%\eqref {eq10.2} 
is bounded above by 
$1+4.8h_{k'-1}/h_k$ while its reciprocal is bounded below by
$1-4.8h_{k'-1}/h_k$.
Substituting this bound into 
%With $k'=k^{+}$ as in (\ref{eq:DefOfKPlust}), we get from
 (\ref{eq10.1}) (and its reversed version with inequality and sup/inf interchanged), and using the value of $p$, yields 
 (\ref{92.1}). 
 
 For (\ref{m1.5euc}) we have by (\ref{92.1}), first for $F$ and then replacing $F$ by $1$
%     \begin{equation}
\[ \Ebm^z (F ;Z_{k'}=j \,\big |\,X_{\bar{\tau}}=v) \leq 
     \(1+4.8 (k'-k)\frac{h_{k'-1}}{h_k}\)\Ebm^z (F; Z_{k'}=j ).\]
     %\label{92.1m}
%\end{equation}
and
%     \begin{equation}
\[ \Ebm^z (Z_{k'}=j \,\big |\,X_{\bar{\tau}}=v) \geq 
     \(1-4.8 (k'-k)\frac{h_{k'-1}}{h_k}\)\Ebm^z (Z_{k'}=j ).\]
     %\label{92.1m}
%\end{equation}
This gives the upper bound 
\begin{eqnarray}
\Ebm^z (F \,\big |\,Z_{k'}=j, X_{\bar{\tau}}=v)&=& { \Ebm^z (F ;Z_{k'}=j \,\big |\,w_{\bar{\tau}}=v) \over  \Ebm^z (Z_{k'}=j \,\big |\,w_{\bar{\tau}}=v)}
\nn\\
&\leq &       \(1+9.8 (k'-k)\frac{h_{k'-1}}{h_k}\)\Ebm^z (F \,\big |\,Z_{k'}=j). \label{}
\end{eqnarray}
The lower bound is obtained similarly.

We finally turn to the proof of  \eqref{m1.5eucc}. Let $\wh{\tau}=\inf \{t>0\, :\,X_t\in \partial B_{d}(   0, h_{k'-1})\}$. We will need to show that
the law of $X_{\wh\tau}$, started at $x\in \partial B_d(0, h_{k+1})$ and
conditioned on the event $D=\{\wh \tau< \bar \tau\}$, is close to uniform. Toward this end, we begin  with an unconditional statement. 
Let $0^{\ast}$ denote the antipode of $0\in \S^{ 2}$. Note that $\wh{\tau}=\inf \{t>0\, :\,X_t\in \partial B_{d}(   0, h_{k'-1})=\partial B_{d}(  0^{\ast}, \pi-h_{k'-1})\}$.
 Then for any positive  function $h$ on  $\partial B_{d}(   0, h_{k'-1})=\partial B_{d}(  0^{\ast}, \pi-h_{k'-1})$,  and $x\in \partial B_{d}(   0, h_{k+1 })=\partial B_{d}(  0^{\ast}, \pi-h_{k+1 })$, writing   $\bar h=h(X_{\wh \tau})$,
\be
\E^{x}\bar h
= \int_{\partial B_{d}(   0, h_{k'-1})}  p_{B_{d}(0^{\ast}, \pi-h_{k'-1})}( x,u) h(u) d\la_{k'-1}(u),\label{poissonk}
\ee
where $\la_{k'-1}$ is uniform measure on $\partial B_{d}(   0, h_{k'-1})$. We have,   see (\ref{poisson}),
  \bea
p_{B_{d}(0^{\ast}, \pi-h_{k'-1})}( x,u)&=&\frac{\sin^{ 2} (\pi/2-h_{k'-1}/2)-\sin^{ 2} (\pi/2-h_{k+1}/2)}{\sin^{ 2} ( d( u,x)/2)} \nn\\
%&=&\frac{\cos^{ 2} (h_{k'-1}/2)-\cos^{ 2} (h_{k+1}/2)}{\sin^{ 2} ( d( u,x)/2)}\nn\\
&=&\frac{\sin^{ 2} (h_{k+1}/2)-\sin^{ 2} (h_{k'-1}/2)}{\sin^{ 2} ( d( u,x)/2)}.\label{poissonj}
  \eea
For the upper bound, using $\sin^{ 2}(a)-\sin^{ 2}(b)=\sin (a+b)\sin (a-b)$
  \begin{equation}
 p_{B_{d}(0^{\ast}, \pi-h_{k'-1})}( x,u)\leq \frac{\sin  (  ( h_{k+1}+h_{k'-1})/2)}{\sin  (  ( h_{k+1}-h_{k'-1})/2)}, \label{poissonj1}
  \end{equation}
which is   $B_{k+1, k'-1}$, see (\ref{eq10.2a}),   so that 
   \begin{equation}
 p_{B_{d}(0^{\ast}, \pi-h_{k'-1})}( x,u)\leq   1+2.3 \(\frac{   h_{k'-1}  }{ h_{k+1}  }\). \label{poissonj3}
  \end{equation} 
  
  For the lower bound, using $\sin^{ 2}(a)-\sin^{ 2}(b)=\sin (a+b)\sin (a-b)$ again we  note that
    \begin{equation}
p_{B_{d}(0^{\ast}, \pi-h_{k'-1})}( x,u)\geq \frac{\sin  (  ( h_{k+1}-h_{k'-1})/2)}{\sin  (  ( h_{k+1}+h_{k'-1})/2)}={1 \over B_{k+1, k'-1}}.  \label{poissonj5}
  \end{equation}
It follows  that
     \begin{equation}
p_{B_{d}(0^{\ast}, \pi-h_{k'-1})}( x,u)\geq 1-2.4 \(\frac{   h_{k'-1}  }{ h_{k+1}  }\).  \label{poissonj8}
  \end{equation}
   Combining this with (\ref{poissonk}), we conclude that with $h^\lambda=\int h(u) d\lambda_{k'-1}(u)$,
%Since $(k'-k)\frac{h_{k+1}}{h_k}\geq 3$  we have shown that
%       \begin{equation}
%|p_{B_{d}(0^{\ast}, \pi-h_{k'-1})}( x,u)-1|\leq  2.4 \(\frac{   h_{k'-1}  }{ h_{k+1}  }\)\leq  0.8 (k'-k)\(\frac{   h_{k'-1}  }{ h_{k}  }\).  \label{poissonj9}
%  \end{equation}
 %
  \begin{equation}
  \label{eq-monday1}
  \big|\Ebm^x(\bar h)-h^\lambda\big| \leq 2.4 \(\frac{   h_{k'-1}  }{ h_{k+1}  }\) h^\lambda=:\delta_{k'-1,k+1}h^\lambda.
  \end{equation}
  %and (\ref{m1.5euc}) gives   (\ref{m1.5eucc}).
  
  We next turn to proving the analogous  conditional statement. Recall the event $D=\{\wh \tau< \bar \tau\}$.
  %, and write $\bar h=h(X_{\wh \tau})$.
By the Markov property $\Ebm^x(\bar h, \{\bar \tau<\wh \tau\})=\Ebm^x(\Ebm^{X_{\bar \tau}}(\bar h), \{\bar \tau<\wh \tau\})$, and hence by  \eqref{eq-monday1}  with $k+1$ replaced by $k$
  \begin{equation}
  \label{eq-monday2}
  \big|\Ebm^x(\bar h|D^{c})- h^\lambda\big| \leq\delta_{k'-1,k}h^\lambda.
  \end{equation}
  Note by (\ref{isc.5}) that  $\Pbm^x(D)= 1/(k'-k-1)=:\delta'_{k',k}$ and is independent of $x$. We have
  \[ \Ebm^x(\bar h)=\Ebm^x(\bar h|D)\Pbm^x(D)+\Ebm^x(\bar h|D^c)\Pbm^x(D^c)\]
  and therefore, using the last display and \eqref{eq-monday1} and \eqref{eq-monday2} in the inequality,
  \begin{align}
  \nn\big|\Ebm^x(\bar h|D)-h^\lambda\big|&=\big|\frac{1}{\Pbm^{x}(D)} \left(\Ebm^x(\bar h)-\Ebm^x(\bar h|D^c)\right)+\Ebm^x(\bar h|D^c) -h^\lambda\big|\\
  &\hspace{-.5 in}\leq \left(\frac{(\delta_{k'-1,k+1}+ \delta_{k'-1,k}  )}{\delta'_{k',k}} +\delta_{k'-1,k} \right) h^\lambda\leq 9.6(k'-k)\(\frac{   h_{k'-1}  }{ h_{k}  }\)h^\lambda.\label{eq-monday3}
  \end{align}
This gives the desired conditional estimate.

To obtain \eqref{m1.5eucc} we first consuder $H=F1_{\{Z_{k'}=j\}} g(X_{\bar{\tau}})$. For all $j\geq 1$ we have that $H=H\circ\th_{\wh{\tau}}\,\,\,1_{\{\wh{\tau}<\bar{\tau}\}}$.          Hence
 by the strong Markov property, 
$\Ebm^x (F1_{\{Z_{k'}=j\}}g(X_{\bar{\tau}}))=\Ebm^x \(\Ebm^{X_{\wh{\tau}}  } \(F1_{\{Z_{k'}=j\}} g(X_{\bar{\tau}})\), \,D\)$.  Let $h(y)= \Ebm^{y  } \(F1_{\{Z_{k'}=j\}} g(X_{\bar{\tau}})\)$. (\ref{eq-monday3}) then shows that
\begin{eqnarray}
&&|\Ebm^x (F1_{\{Z_{k'}=j\}}g(X_{\bar{\tau}}))-\Ebm^{\la} (F1_{\{Z_{k'}=j\}}g(X_{\bar{\tau}}))\Pbm^x(D)|
\label{1922}\\
&&   \leq 9.6(k'-k)\(\frac{   h_{k'-1}  }{ h_{k}  }\)\Ebm^{\la} (F1_{\{Z_{k'}=j\}}g(X_{\bar{\tau}}))\Pbm^x(D).\nonumber
\end{eqnarray}
Using this also with $F=1$, proceeding as in the proof of (\ref{m1.5euc}),  and then letting $g\to \de_{v}$ gives  
\begin{eqnarray}
&&|\Ebm^x (F\,\big |\,Z_{k'}=j, X_{\bar{\tau}}=v)-\Ebm^{\la} (F\,\big |\,Z_{k'}=j, X_{\bar{\tau}}=v)|
\label{1923}\\
&&   \leq 19.3(k'-k)\(\frac{   h_{k'-1}  }{ h_{k}  }\)\Ebm^{\la} (F\,\big |\,Z_{k'}=j, X_{\bar{\tau}}=v),\nonumber
\end{eqnarray}
and (\ref{m1.5eucc}) then follows from (\ref{m1.5euc}).
\qed

 \noindent{\bf Proof of Lemma \ref{recursionends}:} This follows from Lemma \ref{lemprobends}
in the same manner as \cite[Lemma 7.3]{DPRZthick} was derived from  \cite[Lemma 7.4]{DPRZthick}: using the strong Markov property, each excursion 
from $\partial B_{d}(y,h_{k'})$ to $\partial B_{d}(y,h_{k-1})$ will contribute a multiplicative
factor $(1+50 (k'-k+1)h_{k'-1}/{h_{k-1}})$ to the probability.
(We emphasize that no rotation invariance of the event 
${\mathcal A}_{k'}$ is used in the argument, since none was imposed on $F$
in Lemma \ref{lemprobends}.)

In more detail, for the $i$'th excursion from $\partial B_{d}(y,h_{k})$ to $\partial B_{d}(y,h_{k-1})$
let $x_{i}, v_{i}$ denote the the starting and endpoint of the excursion and 
let $Z_{k'}^{i}$ denote the number of excursions from $\partial B_{d}(y,h_{k'-1})$ to $\partial B_{d}(y,h_{k'})$. It suffices to show that
\bea
&&\Pbm \({\mathcal A}_{k'} \,\big |\,x_{i}, v_{i}, Z^{i}_{k'}=j_{i}; i\in [1, n]\)\label{indf.1}\\
&&=\left(1+O\( (k'-k+1)\frac{h_{k'-1}}{h_{k-1}}\)\right)^{\sum_{i=1}^{n}j_{i}}\Pbm \({\mathcal A}_{k'} \,\big |\,Z^{i}_{k'}=j_{i}; i\in [1, n]\),\nn
\eea
with $| O\( (k'-k+1)\frac{h_{k'-1}}{h_{k-1}}\)|\leq 50 (k'-k+1)\frac{h_{k'-1}}{h_{k-1}}$. 

Let $F=\prod_{i=1}^{n}F_{i}$ where $F_{i}$ depends on the excursions inside $B_{d}(y,h_{k'})$ during the $i$'th excursion from $\partial B_{d}(y,h_{k})$ to $\partial B_{d}(y,h_{k-1})$. We set $F_{i}=1$ if $Z^{i}_{k'}=0$. Then to prove (\ref{indf.1}) it suffices to show that
\bea
&&\Ebm \(F \,\big |\,x_{i}, v_{i}, Z^{i}_{k'}=j_{i}; i\in [1, n]\)\label{indf.2}\\
&&=\left(1+O\( (k'-k+1)\frac{h_{k'-1}}{h_{k-1}}\)\right)^{\sum_{i=1}^{n}j_{i}}\Ebm \(F \,\big |\,Z^{i}_{k'}=j_{i}; i\in [1, n]\).\nn
\eea
By the Markov property and (\ref{m1.5eucc}) with $k$ replaced by $k-1$ we have 
\bea
&&\hspace{.3 in}\Ebm \(\prod_{i=1}^{n}F_{i} \,\big |\,x_{i}, v_{i}, Z^{i}_{k'}=j_{i}; i\in [1, n]\)\label{indf.3}\\
&&=\Ebm \(\prod_{i=1}^{n-1}F_{i} \,\big |\,x_{i}, v_{i}, Z^{i}_{k'}=j_{i}; i\in [1, n-1]\)\Ebm^{x_{n}} \(F_{n} \,\big |\, v_{n}, Z^{n}_{k'}=j_{n}\)\nn\\
&&\leq \left(1+28.2\( (k'-k+1)\frac{h_{k'-1}}{h_{k-1}}\)\right)^{1_{\{j_{n}>0\}}}\nn\\
&&\hspace{.3 in}\Ebm \(\prod_{i=1}^{n-1}F_{i} \,\big |\,x_{i}, v_{i}, Z^{i}_{k'}=j_{i}; i\in [1, n-1]\)\Ebm^{\la_{k'-1}} \( F_{n} \,\big |\, Z^{n}_{k'}=j_{n}\),\nn
\eea
and by induction
\bea
&&\hspace{.3 in}\Ebm \(\prod_{i=1}^{n}F_{i} \,\big |\,x_{i}, v_{i}, Z^{i}_{k'}=j_{i}; i\in [1, n]\)\label{indf.4}\\ 
&&\leq \left(1+28.2\( (k'-k+1)\frac{h_{k'-1}}{h_{k-1}}\)\right)^{\sum_{i=1}^{n}j_{i}}\prod_{i=1}^{n}\Ebm^{\la_{k'-1}} \( F_{i} \,\big |\, Z^{i}_{k'}=j_{i}\).\nn
\eea
Similarly, by the analogue of (\ref{1923}), but without conditioning on the endpoint, which follows from (\ref{1922}) with $g\equiv1$, we have
\bea
&&\hspace{.3 in}\Ebm \(\prod_{i=1}^{n}F_{i} \,\big |\,  Z^{i}_{k'}=j_{i}; i\in [1, n]\)\label{indf.5}\\ 
&&\geq \left(1-19.3\( (k'-k+1)\frac{h_{k'-1}}{h_{k-1}}\)\right)^{\sum_{i=1}^{n}j_{i}}\prod_{i=1}^{n}\Ebm^{\la_{k'-1}} \( F_{i} \,\big |\, Z^{i}_{k'}=j_{i}\).\nn
\eea
Together this gives the upper bound in  (\ref{indf.2}) and the lower bound is proven similarly.\qed

\bigskip
\noindent
\begin{tabular}{lll} & David Belius   \\
& Institute of
Mathematics  \\
& University of Zurich  \\
& CH-8057 Zurich, Switzerland \\ 
& david.belius@cantab.net\\ 
& &\\
& & \\
& Jay Rosen\\
& Department of Mathematics\\
&  College of Staten Island, CUNY\\
& Staten Island, NY 10314 \\
& jrosen30@optimum.net\\
& &\\
& & \\
 & Ofer Zeitouni\\
& Department of Mathematics\\
& Weitzmann Institute and Courant Institute, NYU\\
& Rehovot 32000, Israel and NYC, NY 10012 \\
& ofer.zeitouni@weizmann.ac.il
\end{tabular}

\end{document}